\documentclass[11pt, reqno]{amsart}

\usepackage{amsmath, amsthm, amscd, amsfonts, amssymb, graphicx, color, mathtools, mathrsfs}
\usepackage[bookmarksnumbered, colorlinks, plainpages]{hyperref}
\usepackage{enumerate}
\usepackage{setspace}
\usepackage{multicol}
\usepackage[margin=1in]{geometry}
\usepackage{comment}
\usepackage{booktabs}
\usepackage{caption}
\usepackage{subcaption}
\usepackage{pgfplots}

\newcounter{lititem}[subsection]
\renewcommand{\thelititem}{\arabic{lititem}}

\newcommand{\lititem}[1]{%
	\par\medskip
	\refstepcounter{lititem}%
	\noindent\textnormal{(\thelititem)}~\emph{#1}%
	\par\nopagebreak\smallskip
}

\theoremstyle{definition}
\newtheorem{theorem}{Theorem}[section]
\newtheorem{lemma}[theorem]{Lemma}

\newtheorem{algorithm}[theorem]{Algorithm}
\newtheorem{assumption}[theorem]{Assumption}

\newtheorem{remark}[theorem]{Remark}
\numberwithin{equation}{section}

\newcommand{\bff}{\boldsymbol}
\newcommand{\bb}{\mathbb}

\newcommand{\dr}{\mathrm{d}r}
\newcommand{\dtt}{\mathrm{d}_t}

\newcommand{\ddt}{\frac{\mathrm{d}}{\mathrm{d}t}}
\newcommand{\dx}{\mathrm{d}x}
\newcommand{\ds}{\mathrm{d}s}

\newcommand{\kex}{\kappa_{\mathrm{e}}}
\newcommand{\kdm}{\kappa_{\mathrm{d}}}
\newcommand{\norm}[2]{\left\|{#1}\right\|_{#2}}
\newcommand{\inpro}[2]{\left\langle#1,#2\right\rangle}
\newcommand{\abs}[1]{\left|{#1}\right|}

\allowdisplaybreaks

\begin{document}
	\setcounter{page}{1}
	
	\title[A priori error analysis of a mass-lumped midpoint FEM for LLG with DMI]
	{A priori error analysis of a mass-lumped midpoint finite element method with a structure-preserving solver for the Landau--Lifshitz--Gilbert equation with Dzyaloshinskii--Moriya interaction}
	
	\author[Agus L. Soenjaya]{Agus L. Soenjaya}
	\address{School of Mathematics and Statistics, The University of New South Wales, Sydney 2052, Australia}
	\email{\textcolor[rgb]{0.00,0.00,0.84}{a.soenjaya@unsw.edu.au}}

		
	\date{1 June 2026}
	
	\begin{abstract}
		The Landau--Lifshitz--Gilbert (LLG) equation is a fundamental model in
		micromagnetics for the magnetisation dynamics of ferromagnets at temperatures
		well below the Curie temperature. In chiral magnetic materials, the bulk
		Dzyaloshinskii--Moriya interaction (DMI) contributes a first-order term to the
		effective field, induces a natural chiral boundary condition, and plays an
		important role in the dynamics of skyrmions. In this paper, we
		analyse a mass-lumped midpoint finite element scheme for the LLG equation with
		exchange interaction and bulk DMI. The fully discrete scheme preserves the
		nodal unit-length constraint exactly and satisfies a discrete energy law. Under suitable regularity assumptions on the exact solution, we prove an optimal-order convergence estimate in the energy norm, first order in space and second order in time. To our knowledge, this is the first such a priori error estimate for the mass-lumped midpoint finite element method, even in the exchange-only case. Since the midpoint scheme is nonlinear, its practical implementation requires an algebraic solver that is compatible with the geometric structure of the method. To this end, we propose a structure-preserving fixed-point iteration which preserves the nodal unit-length constraint at every iterate. We further prove that, under an appropriate residual stopping criterion, the resulting inexact scheme retains the same convergence rate as the ideal midpoint scheme, up to the contribution of the solver tolerance. Numerical experiments support the theoretical results and illustrate the robustness and structure-preserving properties of the method.
	\end{abstract}
	\maketitle
	
\section{Introduction}

The Landau--Lifshitz--Gilbert (LLG) equation~\cite{Brown1963,LanLif35} is a fundamental continuum model for the magnetisation dynamics of ferromagnetic materials at temperatures well below the Curie temperature. It plays a central role in micromagnetics~\cite{Lak11} and underpins many technological applications, including magnetic sensors, spintronic devices, and magnetic recording media~\cite{IkeMiuYam10,ParHayTho08,RalSti08}. A defining feature of the LLG equation is the pointwise modulus constraint on the magnetisation. This non-convex geometric constraint, together with the strongly nonlinear structure of the equation, makes the design and rigorous analysis of accurate structure-preserving numerical methods a challenging problem in computational micromagnetics~\cite{Pro01}.

In chiral magnetic materials, an important additional mechanism is the Dzyaloshinskii--Moriya interaction (DMI), an antisymmetric exchange interaction arising in systems without inversion symmetry~\cite{Dzy58, Mor60}. The DMI favours magnetic configurations with a preferred handedness and is a key mechanism behind chiral spin textures such as helices, domain walls, and skyrmions~\cite{FinEtAl16}. These textures are of significant interest in spintronics and magnetic-memory applications, including skyrmion-based information storage and racetrack-type devices~\cite{FinEtAl16, KanEtAl16, ZhaEtAl23}.

\subsection{Problem introduction}

In this paper, we consider the LLG equation with short-range exchange interaction and bulk DMI. The exchange interaction gives the leading second-order contribution to the effective field~\cite{FidSch00}, while the bulk DMI contributes a first-order term and induces the natural chiral boundary condition~\cite{FinEtAl16}. Thus, although the modulus constraint remains unchanged, the DMI modifies the variational structure, boundary behaviour, and energy balance of the model. These features make the construction and analysis of stable and accurate finite element schemes for the LLG equation with DMI particularly delicate.

We now introduce the model considered in this paper. Let $\Omega\subset\bb{R}^d$, $d=1,2,3$, be a bounded domain with outward unit normal vector $\bff n$, and let $T>0$. The dynamics of the magnetisation $\bff m:[0,T]\times\Omega\to\bb{R}^3$, under the influence of an effective field $\mathcal{H}(\bff m)$ and at temperatures far below the Curie temperature, is governed by the Landau--Lifshitz--Gilbert equation:
\begin{subequations}\label{equ:llg}
	\begin{alignat}{2}
		\label{equ:llg eq1}
		&\partial_t \bff{m}
		=
		\alpha \bff{m}\times \partial_t \bff{m}
		-
		\gamma \bff{m}\times \mathcal{H}(\bff{m})
		\qquad && \text{in } (0,T)\times\Omega,
		\\[1ex]
		\label{equ:llg eq2}
		&\mathcal{H}(\bff{m})
		=
		\kex \Delta \bff{m}
		-
		2\kdm \nabla\times \bff{m}
		\qquad && \text{in } (0,T)\times\Omega,
		\\[1ex]
		\label{equ:llg init}
		&\bff{m}(0,\bff{x})
		=
		\bff{m}^0(\bff{x})
		\qquad && \text{for } \bff{x}\in \Omega,
		\\[1ex]
		\label{equ:llg bound}
		&\kex\partial_{\bff{n}} \bff{m}
		+
		\kdm \bff{m}\times \bff{n}
		=
		\bff{0}
		\qquad && \text{on } (0,T)\times\partial \Omega.
	\end{alignat}
\end{subequations}
Here, $\alpha>0$ is the Gilbert damping coefficient, $\gamma>0$ is the gyromagnetic ratio, $\kex>0$ is the exchange stiffness constant, and $\kdm\in\bb{R}$ measures the strength of the DMI. The vector $\bff{n}$ is the outward-pointing unit normal vector. The initial datum is assumed to satisfy the pointwise constraint $|\bff m^0|=1$. The LLG dynamics formally preserves this constraint, so that
\begin{equation}\label{eq:nonconvex-constraint}
|\bff m(t,\bff{x})|=1
\qquad\text{for } (t,\bff{x})\in [0,T]\times\Omega.
\end{equation}
The boundary condition in \eqref{equ:llg bound} is the natural chiral boundary condition associated with the exchange and bulk DMI energy.

The effective field $\mathcal{H}(\bff m)$ in \eqref{equ:llg eq2} is obtained as the negative variational derivative of the micromagnetic energy, which we assume to consist of the exchange energy and the DMI energy:
\begin{equation}\label{eq:energy}
\mathcal E(\bff m)
=
\mathcal{E}_{\rm{exc}}(\bff{m}) + \mathcal{E}_{\rm{DMI}}(\bff{m})
:=
\frac{\kex}{2}\norm{\nabla \bff m}{\bb L^2}^2
+
\kdm \inpro{\bff{m}}{\nabla\times\bff{m}}.
\end{equation}
Indeed, under the natural boundary condition \eqref{equ:llg bound}, one has
\[
\ddt \mathcal E \big(\bff m(t)\big)
=
-\inpro{\mathcal H(\bff m)}{\partial_t\bff m}.
\]
The energy-dissipation structure of the LLG equation is then obtained by testing
\eqref{equ:llg eq1} with $\bff m\times\partial_t\bff m$. Since
$|\bff m|=1$ implies $\bff m\cdot\partial_t\bff m=0$, we obtain
\[
\gamma\inpro{\mathcal H(\bff m)}{\partial_t\bff m}
=
\alpha \norm{\partial_t\bff m}{\bb L^2}^2.
\]
Consequently, sufficiently smooth solutions formally satisfy the energy identity
\begin{equation}\label{eq:energy-identity}
\mathcal E \big(\bff m(t) \big)
+
\frac{\alpha}{\gamma}
\int_0^t \norm{\partial_t\bff m(s)}{\bb L^2}^2\,\ds
=
\mathcal E(\bff m^0),
\qquad \forall t\in[0,T].
\end{equation}
Although additional lower-order contributions, such as applied fields, uniaxial anisotropy, or magnetostatic effects, can also be incorporated into the effective field, we restrict the analysis in this paper to the exchange and bulk DMI contributions. This choice allows us to focus on the main analytical difficulties associated with the higher-order terms.

\subsection{Literature review}

We next review existing analytical and numerical results in the literature that are relevant to the present paper. 

The analytical theory of the LLG equation is by now rather well developed, especially for effective fields consisting of the exchange contribution. The seminal works~\cite{Vis85,AloSoy92} established the existence of weak solutions and revealed the possible nonuniqueness of such solutions. Existence and uniqueness results for more regular solutions were subsequently obtained in~\cite{CarFab01}. For sufficiently small initial data, global existence of arbitrarily smooth solutions to the LLG equation with natural boundary conditions was proved in~\cite{FeiTra17b}. More recently, the weak-strong uniqueness principle was established in~\cite{DiInnPra20}: any weak solution coincides with a strong solution as long as both solutions emanate from the same initial datum and coexist on the same time interval. Compared with the exchange-only model, the inclusion of the bulk DMI introduces a first-order chiral contribution to the effective field and changes the natural boundary condition from the homogeneous Neumann condition to the chiral boundary condition in~\eqref{equ:llg bound}. Existence of weak solutions in this case was shown in~\cite{HrkPfePraRug19}. Mathematical analysis of micromagnetic models involving DMI, particularly in thin-film regimes, has been studied in~\cite{DavFraPraRug22,IgnFra25}.

From the numerical point of view, a central difficulty in the approximation of the LLG equation is the preservation of the pointwise modulus constraint~\eqref{eq:nonconvex-constraint}, while retaining a discrete counterpart of the energy law~\eqref{eq:energy-identity}. Since the present paper is concerned with a finite element approximation, we first recall the main finite element strategies developed for the LLG equation and highlight their treatment of the length constraint, energy stability, and convergence. In the following, $h$ and $k$ denote the spatial mesh size and time-step size, respectively, and $\bff m_h^j$ denotes the finite element approximation of $\bff m$ at time level $t_j$.

\lititem{Tangent-plane schemes with sphere projection \cite{Alo08,AloJai06,LeTra13}.}

In this class of methods, one first computes an approximation $\bff v_h^{j+1}$ of $\partial_t\bff m$ in the discrete tangent space at $\bff m_h^j$. An intermediate value
$\widehat{\bff m}_h^{j+1}= \bff m_h^j+k\bff v_h^{j+1}$ is then normalised nodally to produce $\bff m_h^{j+1}$. These schemes are attractive because they lead to linear systems at each time step and enforce the nodal unit-length constraint after normalisation. Their stability analysis is typically based on a discrete energy inequality, but the proof may require geometric conditions on the mesh~\cite{AloKriSteTou14}. Moreover, the treatment of additional first-order contributions such as the DMI may introduce further conditional stability requirements~\cite{HrkPfePraRug19}. Convergence to weak solutions has been proved in~\cite{Alo08,AloKriSteTou14}; see also~\cite{LeTra13} for coupling with the quasi-static Maxwell equations. Error estimates under sufficient regularity assumptions have been obtained more recently in~\cite{AnLiSun25}. A formally almost second-order modification of this approach was proposed in~\cite{FraPfePra20}.

\lititem{Projection-free tangent-plane schemes~\cite{AbeHrkPagPra14,Bar16}.}

These schemes modify the tangent-plane approach by avoiding the nodal normalisation step. This removes the mesh restrictions associated with the projection and leads to a simpler stability mechanism. The price is that the unit-length constraint is no longer imposed exactly at the nodes; instead, it is recovered only asymptotically as the discretisation parameters tend to zero. Convergence to weak solutions has been proved in~\cite{AbeHrkPagPra14,Bar16}, while a priori error estimates were established in~\cite{FeiTra17a, BarKovWan24}. Higher-order and BDF-type variants, including convergence results and error estimates, have been developed in~\cite{AkrFeiKovLub21} and more recently in~\cite{AldFeiPra26a,AldFeiPra26b}.

\lititem{Mass-lumped midpoint schemes~\cite{BarPro06,Cim09}.}

The mass-lumped midpoint method is a formally second-order-in-time implicit geometric scheme in which the mass-lumped inner product is used to enforce the unit-length constraint at the nodes. A key advantage of this method is that it preserves the nodal modulus constraint exactly and satisfies a discrete energy law without imposing geometric restrictions on the mesh. Its main drawback is that the resulting algebraic problem is nonlinear, so that a suitable structure-preserving nonlinear solver is required at each time step~\cite{BarPro06,Cim09}. Convergence to weak solutions was proved in~\cite{BarPro06}, while the effect of inexact nonlinear solvers, at the level of subsequential weak convergence, has been investigated in~\cite{Cim09,FraPfePraRug23}. To the best of our knowledge, however, a priori error estimates for the mass-lumped midpoint finite element scheme are not available even for the exchange-only effective field, and no rate estimates are known when solver errors are included.

\lititem{Other schemes for the LLG equation.}

Several other discretisations of the LLG equation have also been proposed. Linearly implicit finite element schemes based on alternative formulations of the equation were studied, for instance, in~\cite{An16,Cim05,Gao14}. These methods are computationally attractive, but the unit-length constraint is typically satisfied only asymptotically, and they may not preserve the energy structure. Higher-order linearly implicit BDF methods were analysed in~\cite{AkrFeiKovLub21,YanJia21}; these schemes provide high-order accuracy and rigorous error estimates under suitable regularity assumptions, but again do not impose the unit-length constraint exactly at the nodes. A generalised SAV-type method preserving the length constraint was recently proposed in~\cite{LiSheZhe26}; however, the dissipated quantity is a modified energy, and the analysis in that work is carried out at the time-semidiscrete level. Its extension to a fully discrete finite element scheme retaining the same favourable structural properties appears to be nontrivial.

\lititem{Numerical methods for LLG with DMI.}

Numerical methods for the LLG equation with DMI have also been considered, although rigorous finite element error estimates in this setting remain scarce. A finite difference scheme for the bulk-DMI model was proposed in~\cite{LiGuLanCheRen23}. In the finite element context, convergent tangent-plane integrators for chiral ferromagnets were analysed in~\cite{HrkPfePraRug19}, while the mass-lumped midpoint method with DMI was investigated computationally in~\cite{FraPfePraRug23}. The latter work also highlights the suitability of the midpoint scheme for energy-sensitive skyrmion dynamics, providing further motivation to study this scheme in the presence of DMI. Nevertheless, a priori error estimates for the mass-lumped midpoint finite element approximation of the LLG equation with bulk DMI appear to be unavailable in the literature.

\subsection{Contributions of the paper}

The goal of this paper is to provide a quantitative error analysis for the mass-lumped midpoint finite element approximation of the LLG equation with bulk DMI. This complements the existing convergence theory for structure-preserving LLG discretisations~\cite{BarPro06, HrkPfePraRug19}, where the mass-lumped midpoint method is known to preserve the nodal length constraint and to satisfy a discrete energy law, but where a priori convergence rates appear to be unavailable even for the exchange-only model.

The first contribution of this paper is the formulation and analysis of a fully discrete mass-lumped midpoint finite element scheme for the LLG equation with exchange and bulk DMI under the natural chiral boundary condition. The DMI term is incorporated through the discrete exchange--DMI energy, so that the resulting scheme is compatible with the chiral boundary condition at the variational level. We prove that the scheme preserves the nodal unit-length constraint exactly and satisfies a discrete counterpart of the continuous energy identity. The stability estimates derived from this structure provide the foundation for the subsequent error analysis.

The second and main contribution is an a priori error estimate for the ideal nonlinear midpoint scheme. Under suitable regularity assumptions on the exact solution, we prove that the fully discrete approximation converges with the rate
\[
\max_{0\le j\le J}
\norm{\bff m(t_j)-\bff m_h^j}{\bb H^1}
+
\left(
k\sum_{j=0}^{J-1}
\norm{\partial_t\bff m(t_{j+1/2})-d_t\bff m_h^{j+1}}{\bb L^2}^2
\right)^{\frac12}
\le C(h+k^2),
\]
up to the precise form stated in Theorem~\ref{thm:energy-error-estimate}. The proof combines the geometric cancellations of the midpoint method and careful estimates of the first-order DMI contribution and the chiral boundary condition. A notable feature of the ideal scheme is that its stability and convergence analysis does not require a mesh-dependent time-step restriction. This is in contrast to some existing structure-preserving schemes, where stability and convergence are obtained under additional assumptions such as an angle condition on the triangulation and a CFL-type condition; see, for example,~\cite{HrkPfePraRug19}.

The third contribution concerns the practical solution of the nonlinear algebraic system arising at each time step. We propose a structure-preserving fixed-point iteration for the midpoint scheme; see Algorithm~\ref{alg:constraint-preserving-fixed-point}. Unlike a generic nonlinear solver, the proposed iteration preserves the nodal unit-length constraint at every iterate and satisfies an inexact discrete energy law, with an additional defect term measuring the solver error. This makes the iteration compatible with the geometric and dissipative structure of the underlying midpoint scheme.

Finally, we prove an \emph{a priori} error estimate for the fully discrete method with the solver error taken into account. More precisely, under a certain CFL-type condition, we show that if the fixed-point iteration is terminated according to a suitable residual criterion, then the numerical solution produced by the inexact solver retains the same convergence rate as the exact nonlinear midpoint solution up to the contribution of the prescribed solver tolerance; see Theorem~\ref{thm:error-practical-constraint-preserving}. This result provides a rigorous justification of the practical algorithm and shows that the structure-preserving nonlinear iteration does not destroy the stability and accuracy properties of the midpoint discretisation. These theoretical findings
are supported by numerical experiments illustrating convergence, constraint
preservation, energy dissipation, and the qualitative effect of the DMI-induced
chiral boundary condition; see Section~\ref{sec:numerics}.

\section{Preliminaries}\label{sec:prelim}

In this section, we collect the notation and finite element tools used throughout
the paper.

\subsection{Notations}\label{subsec:notation}

Let $\Omega\subset\bb{R}^d$, $d=1,2,3$, be an open and bounded convex
polytopal domain. For $p\in [1,\infty]$ and $s\ge 0$, we write
\[
\bb{L}^p
:=
\bb{L}^p(\Omega;\bb{R}^3),
\qquad
\bb{W}^{s,p}
:=
\bb{W}^{s,p}(\Omega;\bb{R}^3).
\]
We also set
\[
\bb{H}^s:=\bb{W}^{s,2},
\qquad
\bb{W}^{0,p}:=\bb{L}^p.
\]
The differential operators $\nabla$ and $\Delta$ are understood componentwise when acting on $\bb R^3$-valued functions. For the curl operator, we use the standard
three-dimensional curl when $d=3$. In two space dimensions, for
$\bff v=(v_1,v_2,v_3)$, we use the convention
\[
\nabla\times\bff v
=
\left(
\partial_2 v_3,\,
-\partial_1 v_3,\,
\partial_1 v_2-\partial_2 v_1
\right)^\top,
\]
and in one space dimension,
\[
\nabla\times\bff v
=
\left(
0,\,
-\partial_1 v_3,\,
\partial_1 v_2
\right)^\top.
\]
The outward unit normal is embedded in $\bb R^3$ in the natural way in the
cases $d=1,2$.

If $X$ is a Banach space, then $L^p(0,T;X)$ and $W^{s,p}(0,T;X)$ denote the
usual Lebesgue--Bochner and Sobolev--Bochner spaces of functions from $(0,T)$
into $X$. Throughout the paper, the scalar product in a Hilbert space $H$ is
denoted by $\inpro{\cdot}{\cdot}_H$, with associated norm $\norm{\cdot}{H}$.
When $H=\bb{L}^2$, we simply write
\[
\inpro{\bff u}{\bff v}
:=
\int_\Omega \bff u\cdot\bff v\,d\bff x,
\qquad
\norm{\bff u}{\bb L^2}
:=
\sqrt{\inpro{\bff u}{\bff u}}.
\]
We use the same notation $\inpro{\cdot}{\cdot}$ for the $\bb{L}^2$ inner
product of vector-valued and matrix-valued functions; the meaning will always
be clear from the context.

Finally, $C$ denotes a generic positive constant which may take different values
at different occurrences, but is independent of the discretisation parameters
$h$ and $k$. When the dependence on a parameter is relevant, it is indicated
explicitly, for example by writing $C(T)$.

\subsection{Finite element approximation}\label{subsec:fe-prelim}

Let $\{\mathcal T_h\}_{h>0}$ be a family of shape-regular and quasi-uniform triangulations of $\Omega$ with maximal mesh-size $h>0$. Denote by $\mathcal N_h$ the set of vertices of $\mathcal T_h$. We use the lowest-order conforming Lagrange finite element space
\begin{equation}\label{eq:Vh-def}
	\bb V_h
	:=
	\left\{
	\bff{\phi}_h\in C^0(\overline\Omega;\bb R^3)
	:
	\bff{\phi}_h|_K\in \mathcal P_1(K;\bb R^3),
	\;\, \forall K\in\mathcal T_h
	\right\}
	\subset \bb H^1.
\end{equation}
Let $\{\varphi_z\}_{z\in\mathcal N_h}$ be the scalar nodal basis of the
underlying scalar $P_1$ finite element space, and we denote by $I_h: C^0(\overline\Omega;\bb R^3)\to \bb V_h$ the nodal interpolation operator. For continuous vector fields $\bff{u},\bff{v}\in C^0(\overline\Omega;\bb R^3)$, we define the reduced-integration, or mass-lumped, inner product by
\begin{equation}\label{eq:mass-lumped-inner-product}
	\inpro{\bff{u}}{\bff{v}}_h
	:=
	\int_\Omega I_h(\bff{u}\cdot\bff{v}) \,\dx
	=
	\sum_{z\in\mathcal N_h}\beta_z\,\bff{u}(z)\cdot\bff{v}(z),
	\quad \text{where }\;
	\beta_z:=\int_\Omega\varphi_z\,\dx.
\end{equation}
We write
\[
\norm{\bff{v}_h}{h}^2
:=
\inpro{\bff{v}_h}{\bff{v}_h}_h.
\]
The norm $\norm{\cdot}{h}$ is uniformly equivalent to the standard $\bb L^2$-norm on $\bb V_h$, namely
\begin{equation}\label{eq:lumped-L2-equivalence}
	\norm{\bff{v}_h}{\bb L^2}^2
	\le
	\norm{\bff{v}_h}{h}^2
	\le
	5\norm{\bff{v}_h}{\bb L^2}^2,
	\qquad
	\forall \bff{v}_h\in\bb V_h;
\end{equation}
see~\cite[Lemma 3.9]{Bar15}.

First, we recall the following standard approximation and inverse estimates~\cite{BreSco08, ErnGue04}. There is a constant $C$ independent of $h$, such that for every $\bff{v}\in\bb H^2$,
\begin{equation}\label{eq:Ih-H1-approx}
	\norm{\bff{v}-I_h\bff{v}}{\bb L^2}
	+
	h\norm{\nabla(\bff{v}-I_h\bff{v})}{\bb L^2}
	\le
	Ch^2\norm{\bff{v}}{\bb H^2}.
\end{equation}
More generally, if $p\in (d,\infty)$, then
\begin{equation}\label{eq:Ih-W1p-approx}
	\norm{\bff{v}-I_h\bff{v}}{\bb L^p}
	+
	h\norm{\nabla(\bff{v}-I_h\bff{v})}{\bb L^p}
	\le
	Ch^2\norm{\bff{v}}{\bb W^{2,p}}.
\end{equation}
For every $\bff{v}_h\in\bb V_h$, $s\in\{0,1\}$, and $1\le q\le p\le\infty$, the inverse estimate holds:
\begin{equation}\label{eq:inverse-estimate}
	\norm{\bff{v}_h}{\bb W^{s,p}}
	\le
	Ch^{-d\left(\frac1q-\frac1p\right)}
	\norm{\bff{v}_h}{\bb W^{s,q}}.
\end{equation}
Furthermore, we shall make use of the following quadrature estimate~\cite{BarPro06}: for all $\bff{v}_h,\bff{\chi}_h\in\bb V_h$,
\begin{equation}\label{eq:lumped-quadrature-Hminus1}
	\big|
	\inpro{\bff{v}_h}{\bff{\chi}_h}_h
	-
	\inpro{\bff{v}_h}{\bff{\chi}_h}
	\big|
	\le
	Ch\norm{\bff{v}_h}{\bb L^2}\norm{\bff{\chi}_h}{\bb H^1}.
\end{equation}
We also need a stability estimate for the nodal interpolant applied to piecewise smooth functions~\cite{BarKovWan24}. If $\bff{v}\in C^0(\overline\Omega;\bb R^3)$ and $\bff{v}|_K\in\bb H^2(K)$ for every $K\in\mathcal T_h$, then
\begin{equation}\label{eq:Ih-quasi-stability-H1}
	\norm{I_h\bff{v}}{\bb H^1}
	\le
	C\norm{\bff{v}}{\bb H^1}
	+
	Ch\norm{\mathrm D_h^2\bff{v}}{\bb L^2},
\end{equation}
where $\mathrm D_h^2$ denotes the elementwise Hessian.

Let $P_h:\bb L^2\to\bb V_h$ denote the usual $\bb L^2$-orthogonal projection onto $\bb{V}_h$,
defined by
\begin{equation}\label{eq:L2-projection}
	\inpro{P_h\bff{v}-\bff{v}}{\bff{\chi}_h}
	=
	0,
	\qquad
	\forall \bff{\chi}_h\in\bb V_h.
\end{equation}
On quasi-uniform triangulations, $P_h$ extends to a bounded operator on
$\bb L^p$ and is stable in $\bb W^{1,p}$~\cite{CroTho87, DouDupWah74}; more precisely, for any $p\in[1,\infty]$ and $s\in\{0,1\}$, there exists a constant $C$ independent
of $h$ such that
\begin{equation}\label{eq:Ph-Wsp-stability}
	\norm{P_h\bff{v}}{\bb W^{s,p}}
	\le
	C\norm{\bff{v}}{\bb W^{s,p}}.
\end{equation}
Moreover, for any $p\in (1,\infty)$ and $\bff{v}\in\bb W^{2,p}$,
\begin{equation}\label{eq:Ph-W2p-approx}
	\norm{\bff{v}-P_h\bff{v}}{\bb L^p}
	+
	h\norm{\nabla(\bff{v}-P_h\bff{v})}{\bb L^p}
	\le
	Ch^2\norm{\bff{v}}{\bb W^{2,p}}.
\end{equation}

We now introduce the discrete effective field operator corresponding to the effective field $\mathcal{H}$ in \eqref{equ:llg eq2}.
Define the bilinear form $\mathfrak a:\bb H^1\times\bb H^1\to\bb R$ by
\begin{equation}\label{eq:def-a-form}
	\mathfrak a(\bff{v},\bff{\chi})
	:=
	\kex\inpro{\nabla\bff{v}}{\nabla\bff{\chi}}
	+
	\kdm\inpro{\bff{v}}{\nabla\times\bff{\chi}}
	+
	\kdm\inpro{\nabla\times\bff{v}}{\bff{\chi}}.
\end{equation}
If $\bff{v}$ is sufficiently smooth and satisfies the chiral boundary condition
\begin{equation}\label{eq:chiral-boundary-condition}
	\kex\partial_{\bff{n}} \bff{v} + \kdm \bff{v}\times \bff{n}=\bff{0} \qquad \text{on } \partial\Omega,
\end{equation}
then integration by parts gives
\begin{equation}\label{eq:H-a-relation}
	\inpro{\mathcal H(\bff{v})}{\bff{\chi}}
	=
	-\mathfrak a(\bff{v},\bff{\chi}),
	\qquad
	\forall \bff{\chi}\in\bb H^1.
\end{equation}
The form $\mathfrak a$ given by \eqref{eq:def-a-form} is symmetric and continuous on $\bb H^1\times\bb H^1$. For fixed $\kex>0$ and $\kdm\in\bb R$, the form $\mathfrak a$ satisfies the G{\aa}rding inequality: there exist constants $C_0,C_1>0$ such that
\begin{equation}\label{eq:garding-a}
	\mathfrak a(\bff{v},\bff{v})
	+
	C_0 \norm{\bff{v}}{\bb L^2}^2
	\ge
	C_1 \norm{\bff{v}}{\bb H^1}^2,
	\qquad
	\forall \bff{v}\in\bb H^1.
\end{equation}

To mimic \eqref{eq:H-a-relation} at the discrete level, we define the (unshifted) discrete effective field operator $\mathcal H_h^0:\bb V_h\to\bb V_h$ by
\begin{equation}\label{eq:def-discrete-Hh-unshifted}
	\inpro{\mathcal H_h^0\bff{v}_h}{\bff{\chi}_h}_h
	=
	-\mathfrak a(\bff{v}_h,\bff{\chi}_h),
	\qquad
	\forall \bff{v}_h,\bff{\chi}_h\in\bb V_h.
\end{equation}
In the pure exchange case $\kdm=0$, this reduces to $\mathcal H_h^0\bff{v}_h=\kex\Delta_h\bff{v}_h$, where $\Delta_h:\bb V_h\to\bb V_h$ is the discrete Laplacian~\cite{BarPro06} defined by
\begin{equation*}
	\inpro{\Delta_h\bff{v}_h}{\bff{\chi}_h}_h
	=
	-\inpro{\nabla\bff{v}_h}{\nabla\bff{\chi}_h},
	\qquad
	\forall \bff{v}_h,\bff{\chi}_h\in\bb V_h.
\end{equation*}

Since $\mathfrak a$ is only G{\aa}rding coercive in the sense of \eqref{eq:garding-a}, to simplify our analysis we introduce a fixed shift parameter $\lambda>0$. We choose $\lambda$ sufficiently large such that the shifted discrete bilinear form
\begin{equation}\label{eq:def-shifted-form}
	\mathfrak a_{h,\lambda}(\bff{v}_h,\bff{\chi}_h)
	:=
	\mathfrak a(\bff{v}_h,\bff{\chi}_h)
	+
	\lambda\inpro{\bff{v}_h}{\bff{\chi}_h}_h,
	\qquad
	\forall \bff{v}_h,\bff{\chi}_h\in\bb V_h
\end{equation}
is uniformly coercive on $\bb V_h$, namely
\begin{equation}\label{eq:shifted-coercivity}
	\mathfrak a_{h,\lambda}(\bff{v}_h,\bff{v}_h)
	\ge
	c_\lambda\norm{\bff{v}_h}{\bb H^1}^2,
	\qquad
	\forall \bff{v}_h\in\bb V_h.
\end{equation}
Indeed, \eqref{eq:garding-a} and the norm equivalence
\eqref{eq:lumped-L2-equivalence} imply
\[
\mathfrak a_{h,\lambda}(\bff{v}_h,\bff{v}_h)
\ge
C_1 \norm{\bff{v}_h}{\bb H^1}^2
+
(\lambda - C_0) \norm{\bff{v}_h}{\bb L^2}^2,
\]
so \eqref{eq:shifted-coercivity} follows by choosing $\lambda\geq C_0$.

Moreover, we define the shifted continuous and discrete effective fields by
\begin{equation}\label{eq:def-shifted-fields}
	\mathcal H_\lambda(\bff{v})
	:=
	\mathcal H(\bff{v})-\lambda\bff{v},
	\qquad
	\mathcal H_{h,\lambda} \bff{v}_h
	:=
	\mathcal H_h^0\bff{v}_h-\lambda\bff{v}_h.
\end{equation}
Then, by \eqref{eq:def-discrete-Hh-unshifted},
\begin{equation}\label{eq:def-discrete-shifted-H}
	\inpro{\mathcal H_{h,\lambda} \bff{v}_h}{\bff{\chi}_h}_h
	=
	-\mathfrak a_{h,\lambda}(\bff{v}_h,\bff{\chi}_h),
	\qquad
	\forall \bff{v}_h,\bff{\chi}_h\in\bb V_h.
\end{equation}
For $\bff{v}\in \bb{H}^2$ satisfying \eqref{eq:chiral-boundary-condition}, we define the elliptic projection $R_h\bff{v}\in\bb V_h$ by
\begin{equation}\label{eq:def-DMI-Ritz-projection-field}
	\mathcal H_{h,\lambda} R_h\bff{v}
	=
	P_h\mathcal H_\lambda (\bff{v}).
\end{equation}
Equivalently, since \eqref{eq:def-DMI-Ritz-projection-field} is an equality in $\bb V_h$ represented through the mass-lumped inner product,
\begin{equation}\label{eq:def-DMI-Ritz-projection-weak}
	\mathfrak a_{h,\lambda}(R_h \bff{v},\bff{\chi}_h)
	=
	-\inpro{P_h\mathcal H_\lambda(\bff{v})}{\bff{\chi}_h}_h,
	\qquad
	\forall \bff{\chi}_h\in\bb V_h.
\end{equation}
By \eqref{eq:shifted-coercivity}, this elliptic projection is well-defined for every admissible $\bff{v}$.

\begin{remark}\label{rem:shift-does-not-change-LLG}
	Replacing $\mathcal{H}$ by $\mathcal{H}_\lambda$ does not change the LLG equation at the continuous level, since $\bff{m}\times\mathcal H_\lambda(\bff{m})= \bff{m}\times\mathcal H(\bff{m})$. Similarly, replacing $\mathcal{H}_h^0$ by $\mathcal{H}_{h,\lambda}$ at the discrete level will also not change the ideal midpoint scheme to be defined later.
\end{remark}

Other technical estimates used in the \emph{a priori} error analysis in the next section are proved in Appendix~\ref{sec:appendix}. They concern the consistency of the midpoint approximation, projection errors, and bounds related to the mass-lumped inner product and the discrete effective field. We will refer to these results when needed in order to keep the main argument focused.

\section{The ideal midpoint scheme and its error analysis}\label{sec:error-analysis}

Fix $N\in\mathbb N$ and let $\bb{V}_h$ be the finite element space defined in \eqref{eq:Vh-def}. Let $k:=T/N$ be the time step size and write $t_i:=ik$ for $i\in \{0,1,\ldots,N\}$. For a sequence $\{\bff v^i\}_{i=0}^N$, we write
\[
\dtt\bff v^{i+1}:=\frac{\bff v^{i+1}-\bff v^i}{k},
\qquad
\overline{\bff v}^{i+\frac12}:=\frac12(\bff v^{i+1}+\bff v^i).
\]
For the exact solution, we use the notation
\[
\bff m^i:=\bff m(t_i),
\qquad
\bff m^{i+\frac12}:=\bff m(t_{i+\frac12}),
\qquad
\dot{\bff m}^{i+\frac12}:=\partial_t\bff m(t_{i+\frac12}),
\]
and
\[
\overline{\bff m}^{i+\frac12}
:=
\frac12(\bff m^{i+1}+\bff m^i).
\]

We first present the fully implicit midpoint scheme in its ideal form, in which the nonlinear algebraic system arising at each time step is assumed to be solved exactly. This idealised scheme is the object of the stability and error analysis in the next sections, and it also serves as the reference scheme for the inexact solver considered later.

\begin{algorithm}[Ideal fully implicit midpoint scheme]
	\label{alg:exact-fully-implicit-midpoint}
	Set $\bff{m}_h^0:=I_h\bff{m}^0$.
	\\
	\textbf{For} $i=0$ to $N-1$, given $\bff{m}_h^i\in\bb V_h$, \textbf{do}:
	\\[1ex]
	 Find $\bff{m}_h^{i+1}\in\bb V_h$ such that, for all $\bff{\phi}_h\in\bb V_h$,
		\begin{equation}\label{eq:midpoint-scheme}
			\inpro{\dtt\bff m_h^{i+1}}{\bff\phi_h}_h
			-
			\alpha
			\inpro{
				\overline{\bff m}_h^{i+\frac12}
				\times
				\dtt\bff m_h^{i+1}
			}{\bff\phi_h}_h
			+
			\gamma
			\inpro{
				\overline{\bff m}_h^{i+\frac12}
				\times
				\mathcal H_{h,\lambda}\overline{\bff m}_h^{i+\frac12}
			}{\bff\phi_h}_h
			=
			0.
		\end{equation}
	\textbf{Output}: a sequence of approximations $\{\bff{m}_h^i\}_{1\leq i\leq N}$.
\end{algorithm}

Since we have
\begin{equation}\label{eq:shift-do-not-change}
\overline{\bff m}_h^{i+\frac12}
\times
\mathcal H_{h,\lambda}\overline{\bff m}_h^{i+\frac12}
=
\overline{\bff m}_h^{i+\frac12}
\times
\mathcal H_h^0\overline{\bff m}_h^{i+\frac12},
\end{equation}
the scheme \eqref{eq:midpoint-scheme} is identical to the unshifted midpoint scheme ($\lambda=0$). When $\kdm=0$, this scheme was considered in \cite{BarPro06}, where convergence towards a weak solution (along a subsequence and without rate) was shown. This scheme is well-posed by the same argument as in~\cite{BarPro06, FraPfePraRug23}.

For fixed $\bff m_h^i$, the scheme \eqref{eq:midpoint-scheme} is a finite-dimensional nonlinear algebraic system for the unknown $\bff m_h^{i+1}$. The nonlinearity comes from the precession term
\[
\overline{\bff m}_h^{i+\frac12}
\times
\mathcal H_{h,\lambda}\overline{\bff m}_h^{i+\frac12},
\]
where both factors depend on the unknown midpoint value. The Gilbert term is comparatively simpler, since
\[
\overline{\bff m}_h^{i+\frac12}
\times
\dtt\bff m_h^{i+1}
=
\frac1k\bff m_h^i\times\bff m_h^{i+1}.
\]
Thus, at each time step, one has to solve a nonlinear system, for instance by fixed-point iteration, Newton's method, or variants thereof; see, e.g.,~\cite{BarPro06,FraPfePraRug23} for detailed discussions. We now study the idealised scheme \eqref{eq:midpoint-scheme}, assuming that this nonlinear system is solved exactly. The effect of an inexact nonlinear solve is addressed later in Section~\ref{sec:inexact-solver}, where we propose a structure-preserving fixed-point iteration and incorporate the resulting solver error into the analysis.

The following regularity assumption is made to accommodate our \emph{a priori} error analysis.

\begin{assumption}[Regularity of the exact solution]\label{ass:exact-regularity}
	Let $p> \max\{d,2\}$ be fixed such that
	\[
	d\left(\frac12-\frac1p\right)\le 1.
	\]
	In particular, when $d=3$ we take $p\in (3,6]$. We assume that the exact solution $\bff m$ of \eqref{equ:llg} satisfies
	\begin{equation}\label{eq:exact-regularity-main}
		\bff m
		\in
		L^\infty(0,T;\bb H^4\cap\bb W^{3,p})
		\cap
		W^{1,\infty}(0,T;\bb H^2\cap\bb W^{1,p})
		\cap
		W^{2,\infty}(0,T;\bb H^3)
		\cap
		W^{3,\infty}(0,T;\bb H^1).
	\end{equation}
	Moreover, $\bff m(t)$ satisfies the chiral boundary condition \eqref{equ:llg bound}
	for every $t\in[0,T]$ in the sense required for the elliptic projection $R_h\bff m(t)$.
	
	The existence of arbitrarily regular solutions to the LLG equation, at least for sufficiently small initial data and with $\mathcal{H}(\bff{m})$ consisting of only the exchange field, was shown in~\cite{FeiTra17b}.
\end{assumption}

\begin{remark}\label{rem:regularity-sufficient-not-minimal}
	Assumption~\ref{ass:exact-regularity} is a convenient sufficient condition and is not intended to be minimal. The requirement $\bff m\in L^\infty(0,T;\bb H^4)$
	ensures that $\mathcal H_\lambda(\bff m)
	\in
	L^\infty(0,T;\bb H^2)$,
	which is used to estimate the projection defect $P_h\mathcal H_\lambda(\bff m)-\mathcal H_\lambda(\bff m)$ by \eqref{eq:Ph-W2p-approx}.
	
	The requirement $\bff m\in L^\infty(0,T;\bb W^{3,p})$
	ensures that
	$\mathcal H_\lambda(\bff m)
	\in
	L^\infty(0,T;\bb W^{1,p})$,
	which is needed for the bound
	\[
	\norm{P_h\mathcal H_\lambda(\overline{\bff m}^{i+\frac12})}{\bb W^{1,p}}
	\le C.
	\]
	
	The regularity of $\partial_t\bff m$ gives \eqref{eq:dt-Rhm-W1p-bound},
	while the assumptions on $\partial_{tt}\bff m$ and $\partial_{ttt}\bff m$ give the midpoint consistency estimates \eqref{eq:midpoint-consistency}, to be proven in Lemma~\ref{lem:midpoint-consistency}.
	These estimates are essential for our \emph{a priori} error analysis.
\end{remark}

At each time step, the solution of \eqref{eq:midpoint-scheme} belongs to the set
\begin{align}\label{eq:Mh}
	\mathcal{M}_h:= \{\bff{\phi}_h\in \bb{V}_h: \abs{\bff{\phi}_h}=1, \quad \forall \bff{z}\in \mathcal{N}_h\},
\end{align}
a property which we refer to as the nodal length preservation.
The following lemma proves this property and the energy stability of the scheme.

\begin{lemma}[Nodal length preservation and energy stability]\label{lem:scheme-stability}
	Suppose that $\bff m_h^{i+1}$ solves \eqref{eq:midpoint-scheme}. Then for any $i\in \{0,1,\ldots,N-1\}$,
	\begin{equation}\label{eq:nodal-length-preservation}
		|\bff m_h^{i+1}(z)|=|\bff m_h^i(z)|,
		\qquad
		\forall z\in\mathcal N_h.
	\end{equation}
	Equivalently, $\bff{m}_h^i\in \mathcal{M}_h$ for all $i$ if $\bff{m}_h^0\in \mathcal{M}_h$.
	Moreover, we have the discrete energy stability
	\begin{equation}\label{eq:discrete-energy-stability}
		\frac{\gamma}{2\alpha}
		\left[
		\mathfrak a_{h,\lambda}(\bff m_h^{i+1},\bff m_h^{i+1})
		-
		\mathfrak a_{h,\lambda}(\bff m_h^i,\bff m_h^i)
		\right]
		+
		k\norm{\dtt\bff m_h^{i+1}}{h}^2
		=
		0.
	\end{equation}
	Consequently, it holds that
	\begin{equation}\label{eq:scheme-H1-stability}
		\max_{0\le j\le N} \left(\norm{\bff m_h^j}{\bb L^\infty}^2 + \norm{\bff m_h^j}{\bb H^1}^2 \right)
		+
		k\sum_{i=0}^{N-1}\norm{\dtt\bff m_h^{i+1}}{h}^2
		\le
		C,
	\end{equation}
	where the constant $C$ may depend on $\norm{\bff{m}^0}{\bb{H}^1}$, but is independent of $n$, $h$, or $k$.
\end{lemma}

\begin{proof}
	Since the test function in \eqref{eq:midpoint-scheme} is arbitrary and the mass-lumped inner product is nodal, the scheme is equivalent to a system of nodal equations. Testing the nodal equation at $z\in\mathcal N_h$ with
	$\overline{\bff m}_h^{i+\frac12}(z)$ gives
	\[
	\dtt\bff m_h^{i+1}(z)\cdot
	\overline{\bff m}_h^{i+\frac12}(z)
	=
	0,
	\]
	since each cross-product term is orthogonal to
	$\overline{\bff m}_h^{i+\frac12}(z)$. Therefore
	\[
	|\bff m_h^{i+1}(z)|^2-|\bff m_h^i(z)|^2
	=
	2k\,\dtt\bff m_h^{i+1}(z)\cdot
	\overline{\bff m}_h^{i+\frac12}(z)
	=
	0,
	\]
	which proves \eqref{eq:nodal-length-preservation}.
	
	We now prove \eqref{eq:discrete-energy-stability}. Taking
	$\bff\phi_h=\dtt\bff m_h^{i+1}$ in \eqref{eq:midpoint-scheme} gives
	\begin{equation}\label{eq:scheme-test-dt}
		\norm{\dtt\bff m_h^{i+1}}{h}^2
		+
		\gamma
		\inpro{
			\overline{\bff m}_h^{i+\frac12}
			\times
			\mathcal H_{h,\lambda}\overline{\bff m}_h^{i+\frac12}
		}{
			\dtt\bff m_h^{i+1}
		}_h
		=
		0.
	\end{equation}
	Next, taking $\bff\phi_h
	=
	-\frac{\gamma}{\alpha}
	\mathcal H_{h,\lambda}\overline{\bff m}_h^{i+\frac12}$
	in \eqref{eq:midpoint-scheme}, we obtain
	\begin{align}
		&-\frac{\gamma}{\alpha}
		\inpro{
			\dtt\bff m_h^{i+1}
		}{
			\mathcal H_{h,\lambda}\overline{\bff m}_h^{i+\frac12}
		}_h
		+
		\gamma
		\inpro{
			\overline{\bff m}_h^{i+\frac12}
			\times
			\dtt\bff m_h^{i+1}
		}{
			\mathcal H_{h,\lambda}\overline{\bff m}_h^{i+\frac12}
		}_h
		=
		0.
		\label{eq:scheme-test-field}
	\end{align}
	By adding \eqref{eq:scheme-test-dt} and \eqref{eq:scheme-test-field}, the field-cross term cancels. We note that by \eqref{eq:def-discrete-shifted-H},
	\[
	\inpro{
		\dtt\bff m_h^{i+1}
	}{
		\mathcal H_{h,\lambda}\overline{\bff m}_h^{i+\frac12}
	}_h
	=
	-\mathfrak a_{h,\lambda}
	\left(
	\overline{\bff m}_h^{i+\frac12},
	\dtt\bff m_h^{i+1}
	\right).
	\]
	Furthermore, since $\mathfrak a_{h,\lambda}$ is symmetric,
	\[
	\mathfrak a_{h,\lambda}
	\left(
	\overline{\bff m}_h^{i+\frac12},
	\dtt\bff m_h^{i+1}
	\right)
	=
	\frac{1}{2k}
	\left[
	\mathfrak a_{h,\lambda}(\bff m_h^{i+1},\bff m_h^{i+1})
	-
	\mathfrak a_{h,\lambda}(\bff m_h^i,\bff m_h^i)
	\right].
	\]
	This proves \eqref{eq:discrete-energy-stability}.
	
	Finally, since $\bff{m}_h^i$ is piecewise affine, it is a convex combination of its nodal values on each element. Hence, \eqref{eq:nodal-length-preservation}, together with the fact that $\bff m_h^0=I_h\bff m^0$ and $|\bff m^0|=1$, implies $\norm{\bff{m}_h^i}{\bb{L}^\infty}\leq 1$. Bounds on the rest of the terms in \eqref{eq:scheme-H1-stability} follow from \eqref{eq:discrete-energy-stability} and the coercivity of $\mathfrak a_{h,\lambda}$.
\end{proof}

We next derive the consistency equation. Define
\begin{equation}\label{eq:def-consistency-defects}
	\bff E_m^{i+\frac12}
	:=
	R_h\overline{\bff m}^{i+\frac12}
	-
	\bff m^{i+\frac12},
	\qquad
	\bff E_t^{i+1}
	:=
	\dtt R_h\bff m^{i+1}
	-
	\dot{\bff m}^{i+\frac12},
\end{equation}
and
\begin{equation}\label{eq:def-field-defect}
	\bff E_H^{i+\frac12}
	:=
	P_h\mathcal H_\lambda(\overline{\bff m}^{i+\frac12})
	-
	\mathcal H_\lambda(\bff m^{i+\frac12}).
\end{equation}
By Lemmas~\ref{lem:DMI-Ritz-H1} and \ref{lem:midpoint-consistency}, as well as the approximation properties of $P_h$ in \eqref{eq:Ph-W2p-approx}, we have
\begin{equation}\label{eq:basic-consistency-defect-bounds}
	\norm{\bff E_m^{i+\frac12}}{\bb H^1}
	+
	\norm{\bff E_t^{i+1}}{\bb H^1}
	+
	\norm{\bff E_H^{i+\frac12}}{\bb H^1}
	\le
	C(h+k^2).
\end{equation}
Furthermore, Lemma~\ref{lem:smooth-W1p-bounds} and the stability of $P_h$ in \eqref{eq:Ph-Wsp-stability} imply
\begin{equation}\label{eq:coefficient-W1p-bounds}
	\norm{\dtt R_h\bff m^{i+1}}{\bb W^{1,p}}
	+
	\norm{P_h\mathcal H_\lambda(\overline{\bff m}^{i+\frac12})}{\bb W^{1,p}}
	\le
	C.
\end{equation}
These estimates will be repeatedly used in the sequel.

Next, we record in the following lemma the projected exact equation and a consistency estimate.

\begin{lemma}[Projected exact equation and consistency estimate]\label{lem:projected-exact-consistency}
	For each $i\in \{0,1,\ldots,N-1\}$, the projected exact solution satisfies
	\begin{align}
		&\inpro{\dtt R_h\bff m^{i+1}}{\bff\phi_h}_h
		-
		\alpha
		\inpro{
			R_h\overline{\bff m}^{i+\frac12}
			\times
			\dtt R_h\bff m^{i+1}
		}{
			\bff\phi_h
		}_h
		\notag
		\\
		&\qquad
		+
		\gamma
		\inpro{
			R_h\overline{\bff m}^{i+\frac12}
			\times
			\mathcal H_{h,\lambda}R_h\overline{\bff m}^{i+\frac12}
		}{
			\bff\phi_h
		}_h
		=
		\inpro{\bff\rho_h^{i+1}}{\bff\phi_h}_h
		\label{eq:projected-exact-equation}
	\end{align}
	for all $\bff\phi_h\in\bb V_h$, where the consistency residual
	\begin{align}
		\bff\rho_h^{i+1}
		&:=
		\dtt R_h\bff m^{i+1}
		-
		I_h\dot{\bff m}^{i+\frac12}
		\notag
		\\
		&\quad
		-
		\alpha
		I_h
		\left(
		R_h\overline{\bff m}^{i+\frac12}
		\times
		\dtt R_h\bff m^{i+1}
		-
		\bff m^{i+\frac12}
		\times
		\dot{\bff m}^{i+\frac12}
		\right)
		\notag
		\\
		&\quad
		+
		\gamma
		I_h
		\left(
		R_h\overline{\bff m}^{i+\frac12}
		\times
		P_h\mathcal H_\lambda(\overline{\bff m}^{i+\frac12})
		-
		\bff m^{i+\frac12}
		\times
		\mathcal H_\lambda(\bff m^{i+\frac12})
		\right).
		\label{eq:def-rho-consistency}
	\end{align}
	Moreover, we have the estimate
	\begin{equation}\label{eq:rho-estimate}
		\norm{\bff\rho_h^{i+1}}{\bb H^1}
		\le
		C(h+k^2).
	\end{equation}
\end{lemma}

\begin{proof}
	The exact solution satisfies the pointwise identity
	\[
	\dot{\bff m}^{i+\frac12}
	-
	\alpha\bff m^{i+\frac12}\times\dot{\bff m}^{i+\frac12}
	+
	\gamma\bff m^{i+\frac12}\times
	\mathcal H_\lambda(\bff m^{i+\frac12})
	=
	\bff0.
	\] 
	Applying $I_h$ and testing with $\bff\phi_h\in\bb V_h$ in the mass-lumped product gives
	\begin{align}\label{eq:exact-nodal-residual}
		&\inpro{I_h\dot{\bff m}^{i+\frac12}}{\bff\phi_h}_h
		-
		\alpha
		\inpro{
			I_h(\bff m^{i+\frac12}\times\dot{\bff m}^{i+\frac12})
		}{
			\bff\phi_h
		}_h
		+
		\gamma
		\inpro{
			I_h(\bff m^{i+\frac12}\times
			\mathcal H_\lambda(\bff m^{i+\frac12}))
		}{
			\bff\phi_h
		}_h
		=
		0.
	\end{align}
	Subtracting \eqref{eq:exact-nodal-residual} from the left-hand side of \eqref{eq:projected-exact-equation}, and using \eqref{eq:def-DMI-Ritz-projection-field} gives precisely the identity \eqref{eq:def-rho-consistency}.
	
	It remains to prove \eqref{eq:rho-estimate}. From \eqref{eq:def-rho-consistency}, by adding and subtracting the relevant terms, we can write
	\[
	\bff\rho_h^{i+1}
	=
	\bff\rho_{0,h}^{i+1}
	-
	\alpha\bff\rho_{1,h}^{i+1}
	+
	\gamma\bff\rho_{2,h}^{i+1},
	\]
	where, using the notations in \eqref{eq:def-consistency-defects} and \eqref{eq:def-field-defect},
	\begin{align*}
	\bff\rho_{0,h}^{i+1}
	&:=
	\bff E_t^{i+1}
	+
	\dot{\bff m}^{i+\frac12}
	-
	I_h\dot{\bff m}^{i+\frac12},
	\\
	\bff\rho_{1,h}^{i+1}
	&:=
	I_h
	\left(
	\bff E_m^{i+\frac12}
	\times
	\dtt R_h\bff m^{i+1}
	+
	\bff m^{i+\frac12}
	\times
	\bff E_t^{i+1}
	\right),
	\\
	\bff\rho_{2,h}^{i+1}
	&:=
	I_h
	\left(
	\bff E_m^{i+\frac12}
	\times
	P_h\mathcal H_\lambda(\overline{\bff m}^{i+\frac12})
	+
	\bff m^{i+\frac12}
	\times
	\bff E_H^{i+\frac12}
	\right).
	\end{align*}
	
	For the term $\bff\rho_{0,h}^{i+1}$, we have by \eqref{eq:basic-consistency-defect-bounds} and \eqref{eq:Ih-H1-approx},
	\[
	\norm{\bff\rho_{0,h}^{i+1}}{\bb H^1}
	\le
	\norm{\bff E_t^{i+1}}{\bb{H}^1} + \norm{\dot{\bff m}^{i+\frac12}- I_h\dot{\bff m}^{i+\frac12}}{\bb{H}^1}
	\le
	C(h+k^2).
	\]

	Next, with the aim of estimating $\bff\rho_{1,h}^{i+1}$ and $\bff\rho_{2,h}^{i+1}$, we set
	\begin{align*}
	\delta_m^{i+\frac12}
	&:=
	R_h\overline{\bff m}^{i+\frac12}
	-
	I_h\bff m^{i+\frac12},
	\\
	\delta_t^{i+1}
	&:=
	\dtt R_h\bff m^{i+1}
	-
	I_h\dot{\bff m}^{i+\frac12},
	\\
	\delta_H^{i+\frac12}
	&:=
	P_h\mathcal H_\lambda(\overline{\bff m}^{i+\frac12})
	-
	I_h\mathcal H_\lambda(\bff m^{i+\frac12}).
	\end{align*}
	Then $\delta_m^{i+\frac12},\delta_t^{i+1},\delta_H^{i+\frac12}\in\bb V_h$ and, by
	\eqref{eq:basic-consistency-defect-bounds} and the interpolation estimates \eqref{eq:Ih-H1-approx} and \eqref{eq:Ph-W2p-approx},
	\begin{equation}\label{eq:delta-est}
	\norm{\delta_m^{i+\frac12}}{\bb H^1}
	+
	\norm{\delta_t^{i+1}}{\bb H^1}
	+
	\norm{\delta_H^{i+\frac12}}{\bb H^1}
	\le
	C(h+k^2).
	\end{equation}
	Moreover, thanks to the fact that $I_h\bff{v}-\bff{v}$ vanishes at all nodes for any sufficiently smooth $\bff{v}$, we observe that
	\[
	I_h\left(
	\bff E_m^{i+\frac12}\times \dtt R_h\bff m^{i+1}
	\right)
	=
	I_h\left(
	\delta_m^{i+\frac12}\times \dtt R_h\bff m^{i+1}
	\right),
	\]
	\[
	I_h\left(
	\bff m^{i+\frac12}\times \bff E_t^{i+1}
	\right)
	=
	I_h\left(
	I_h\bff m^{i+\frac12}\times \delta_t^{i+1}
	\right),
	\]
	and similarly
	\[
	I_h\left(
	\bff E_m^{i+\frac12}\times
	P_h\mathcal H_\lambda(\overline{\bff m}^{i+\frac12})
	\right)
	=
	I_h\left(
	\delta_m^{i+\frac12}\times
	P_h\mathcal H_\lambda(\overline{\bff m}^{i+\frac12})
	\right),
	\]
	\[
	I_h\left(
	\bff m^{i+\frac12}\times \bff E_H^{i+\frac12}
	\right)
	=
	I_h\left(
	I_h\bff m^{i+\frac12}\times \delta_H^{i+\frac12}
	\right).
	\]
	Therefore, by Lemma~\ref{lem:discrete-product-estimates} and \eqref{eq:delta-est}, as well as the bound
	\eqref{eq:coefficient-W1p-bounds} and the estimate \eqref{eq:Ih-W1p-approx}, we obtain
	\[
	\begin{aligned}
		\norm{\bff\rho_{1,h}^{i+1}}{\bb H^1}
		&\le
		C\norm{\delta_m^{i+\frac12}}{\bb H^1}
		\norm{\dtt R_h\bff m^{i+1}}{\bb W^{1,p}}
		+
		C\norm{\delta_t^{i+1}}{\bb H^1}
		\norm{I_h\bff m^{i+\frac12}}{\bb W^{1,p}}     
		\\
		&\le
		C(h+k^2).
	\end{aligned}
	\]
	Similar estimate also holds for $\bff\rho_{2,h}^{i+1}$.
	Hence, \eqref{eq:rho-estimate} follows.
\end{proof}


In the error analysis, it is convenient to compare the numerical solution not directly with $\bff m^i$, but with its elliptic projection $R_h\bff m^i$. Thus we introduce
the discrete errors:
\begin{equation}\label{eq:def-error}
	\bff e_h^i
	:=
	\bff m_h^i-R_h\bff m^i,
	\qquad
	\overline{\bff e}_h^{i+\frac12}
	:=
	\overline{\bff m}_h^{i+\frac12}
	-
	R_h\overline{\bff m}^{i+\frac12},
	\qquad
	\dtt\bff e_h^{i+1}
	:=
	\dtt\bff m_h^{i+1}
	-
	\dtt R_h\bff m^{i+1}.
\end{equation}
Notice that $\bff e_h^i$ is not the full error. Instead,
\begin{equation}\label{eq:error-splitting-total}
	\bff m_h^i-\bff m^i
	=
	\bff e_h^i
	+
	(R_h\bff m^i-\bff m^i),
\end{equation}
where the second term is controlled by the elliptic projection estimate in
Lemma~\ref{lem:DMI-Ritz-H1}. Hence it remains to estimate $\bff e_h^i$ in a
suitable energy norm.

We now record the discrete error equation in the following lemma, obtained by subtracting the projected exact equation
\eqref{eq:projected-exact-equation} from the numerical scheme
\eqref{eq:midpoint-scheme}. Its particular form is chosen so that, when tested
with $\dtt\bff e_h^{i+1}$ and
$-\frac{\gamma}{\alpha}\mathcal H_{h,\lambda}\overline{\bff e}_h^{i+\frac12}$,
the leading precession terms cancel and the increment of
$\mathfrak a_{h,\lambda}(\bff e_h^i,\bff e_h^i)$ appears.

\begin{lemma}[Discrete error equation]\label{lem:error-equation}
	For every $\bff\phi_h\in\bb V_h$, the discrete error defined in \eqref{eq:def-error} satisfies
	\begin{align}
		&\inpro{\dtt\bff e_h^{i+1}}{\bff\phi_h}_h
		-
		\alpha
		\inpro{
			\overline{\bff m}_h^{i+\frac12}
			\times
			\dtt\bff e_h^{i+1}
		}{
			\bff\phi_h
		}_h
		+
		\gamma
		\inpro{
			\overline{\bff m}_h^{i+\frac12}
			\times
			\mathcal H_{h,\lambda}\overline{\bff e}_h^{i+\frac12}
		}{
			\bff\phi_h
		}_h
		\notag
		\\
		&=
		\alpha
		\inpro{
			\overline{\bff e}_h^{i+\frac12}
			\times
			\dtt R_h\bff m^{i+1}
		}{
			\bff\phi_h
		}_h
		-
		\gamma
		\inpro{
			\overline{\bff e}_h^{i+\frac12}
			\times
			P_h\mathcal H_\lambda(\overline{\bff m}^{i+\frac12})
		}{
			\bff\phi_h
		}_h
		-
		\inpro{\bff\rho_h^{i+1}}{\bff\phi_h}_h,
		\label{eq:error-equation}
	\end{align}
	where $\bff{\rho}_h^{i+1}$ is the consistency residual term given in \eqref{eq:def-rho-consistency}.
\end{lemma}

\begin{proof}
	We subtract the projected exact equation \eqref{eq:projected-exact-equation} from the numerical scheme \eqref{eq:midpoint-scheme}. For the damping term, using
	\[
	\overline{\bff m}_h^{i+\frac12}
	=
	R_h\overline{\bff m}^{i+\frac12}
	+
	\overline{\bff e}_h^{i+\frac12},
	\]
	we have
	\begin{align*}
		&\overline{\bff m}_h^{i+\frac12}\times
		\dtt\bff m_h^{i+1}
		-
		R_h\overline{\bff m}^{i+\frac12}
		\times
		\dtt R_h\bff m^{i+1}
		=
		\overline{\bff m}_h^{i+\frac12}
		\times
		\dtt\bff e_h^{i+1}
		+
		\overline{\bff e}_h^{i+\frac12}
		\times
		\dtt R_h\bff m^{i+1}.
	\end{align*}
	Similarly, using the linearity of $\mathcal H_{h,\lambda}$ and the defining identity \eqref{eq:def-DMI-Ritz-projection-field}, namely
	\[
	\mathcal H_{h,\lambda}R_h\overline{\bff m}^{i+\frac12}
	=
	P_h\mathcal H_\lambda(\overline{\bff m}^{i+\frac12}),
	\]
	we get
	\begin{align*}
		&\overline{\bff m}_h^{i+\frac12}
		\times
		\mathcal H_{h,\lambda}\overline{\bff m}_h^{i+\frac12}
		-
		R_h\overline{\bff m}^{i+\frac12}
		\times
		\mathcal H_{h,\lambda}R_h\overline{\bff m}^{i+\frac12}
		=
		\overline{\bff m}_h^{i+\frac12}
		\times
		\mathcal H_{h,\lambda}\overline{\bff e}_h^{i+\frac12}
		+
		\overline{\bff e}_h^{i+\frac12}
		\times
		P_h\mathcal H_\lambda(\overline{\bff m}^{i+\frac12}).
	\end{align*}
	Inserting these two identities into the subtracted equation gives \eqref{eq:error-equation}.
\end{proof}

We are now in a position to prove an error estimate in the energy norm for the ideal scheme given by Algorithm~\ref{alg:exact-fully-implicit-midpoint}.

\begin{theorem}[Energy-norm error estimate for the ideal scheme]\label{thm:energy-error-estimate}
	Let Assumption~\ref{ass:exact-regularity} hold, and let
	$\{\bff m_h^i\}_{i=0}^N$ be a sequence satisfying the ideal midpoint scheme (Algorithm~\ref{alg:exact-fully-implicit-midpoint}). Then, for $k>0$ sufficiently small,
	\begin{align}
		&\max_{0\le i\le N}
		\norm{\bff e_h^i}{\bb H^1}^2
		+
		k\sum_{i=0}^{N-1}
		\norm{\dtt\bff e_h^{i+1}}{h}^2
		\le
		C
		\left(
		\norm{\bff e_h^0}{\bb H^1}^2
		+
		(h+k^2)^2
		\right).
		\label{eq:error-estimate-eh}
	\end{align}
	Consequently, we have an \emph{a priori} error estimate
	\begin{equation}\label{eq:error-estimate-total}
		\max_{0\le i\le N}
		\norm{\bff m_h^i-\bff m^i}{\bb H^1}
		\le
		C
		\left(
		h+k^2
		\right),
	\end{equation}
	where $C$ is a constant depending on $T$, but is independent of $N$, $h$, and $k$.
\end{theorem}

\begin{proof}
	Fix $i\in\{0,\ldots,N-1\}$. Taking
	$\bff\phi_h=\dtt\bff e_h^{i+1}$ in the discrete error equation
	\eqref{eq:error-equation} gives
	\begin{align}
		&\norm{\dtt\bff e_h^{i+1}}{h}^2
		+
		\gamma
		\inpro{
			\overline{\bff m}_h^{i+\frac12}
			\times
			\mathcal H_{h,\lambda}\overline{\bff e}_h^{i+\frac12}
		}{
			\dtt\bff e_h^{i+1}
		}_h
		\notag
		\\
		&=
		\alpha
		\inpro{
			\overline{\bff e}_h^{i+\frac12}
			\times
			\dtt R_h\bff m^{i+1}
		}{
			\dtt\bff e_h^{i+1}
		}_h
		-
		\gamma
		\inpro{
			\overline{\bff e}_h^{i+\frac12}
			\times
			P_h\mathcal H_\lambda(\overline{\bff m}^{i+\frac12})
		}{
			\dtt\bff e_h^{i+1}
		}_h
		\notag
		\\
		&\quad
		-
		\inpro{\bff\rho_h^{i+1}}{\dtt\bff e_h^{i+1}}_h .
		\label{eq:error-test-dt-clean}
	\end{align}
	Next, taking $\bff\phi_h = -\frac{\gamma}{\alpha}
	\mathcal H_{h,\lambda}\overline{\bff e}_h^{i+\frac12}$ in \eqref{eq:error-equation}, we obtain
	\begin{align}
		&-\frac{\gamma}{\alpha}
		\inpro{
			\dtt\bff e_h^{i+1}
		}{
			\mathcal H_{h,\lambda}\overline{\bff e}_h^{i+\frac12}
		}_h
		-
		\gamma
		\inpro{
			\overline{\bff m}_h^{i+\frac12}
			\times
			\mathcal H_{h,\lambda}\overline{\bff e}_h^{i+\frac12}
		}{
			\dtt\bff e_h^{i+1}
		}_h
		\notag
		\\
		&=
		-\gamma
		\inpro{
			\overline{\bff e}_h^{i+\frac12}
			\times
			\dtt R_h\bff m^{i+1}
		}{
			\mathcal H_{h,\lambda}\overline{\bff e}_h^{i+\frac12}
		}_h
		\notag
		\\
		&\quad
		+
		\frac{\gamma^2}{\alpha}
		\inpro{
			\overline{\bff e}_h^{i+\frac12}
			\times
			P_h\mathcal H_\lambda(\overline{\bff m}^{i+\frac12})
		}{
			\mathcal H_{h,\lambda}\overline{\bff e}_h^{i+\frac12}
		}_h
		+
		\frac{\gamma}{\alpha}
		\inpro{
			\bff\rho_h^{i+1}
		}{
			\mathcal H_{h,\lambda}\overline{\bff e}_h^{i+\frac12}
		}_h .
		\label{eq:error-test-field-clean}
	\end{align}
	Adding \eqref{eq:error-test-dt-clean} and
	\eqref{eq:error-test-field-clean}, we observe that the terms containing $\inpro{
		\overline{\bff m}_h^{i+\frac12}
		\times
		\mathcal H_{h,\lambda}\overline{\bff e}_h^{i+\frac12}
	}{
		\dtt\bff e_h^{i+1}
	}_h$ cancel.
	Furthermore, note that by \eqref{eq:def-discrete-shifted-H},
	\begin{align}
		-\frac{\gamma}{\alpha}
		\inpro{
			\dtt\bff e_h^{i+1}
		}{
			\mathcal H_{h,\lambda}\overline{\bff e}_h^{i+\frac12}
		}_h
		&=
		\frac{\gamma}{\alpha}
		\mathfrak a_{h,\lambda}
		\left(
		\overline{\bff e}_h^{i+\frac12},
		\dtt\bff e_h^{i+1}
		\right)
		\notag
		\\
		&=
		\frac{\gamma}{2\alpha k}
		\left[
		\mathfrak a_{h,\lambda}(\bff e_h^{i+1},\bff e_h^{i+1})
		-
		\mathfrak a_{h,\lambda}(\bff e_h^i,\bff e_h^i)
		\right].
		\label{eq:error-energy-increment-clean}
	\end{align}
 Therefore, after adding \eqref{eq:error-test-dt-clean} and
 \eqref{eq:error-test-field-clean}, we have
	\begin{align}
		&\norm{\dtt\bff e_h^{i+1}}{h}^2
		+
		\frac{\gamma}{2\alpha k}
		\left[
		\mathfrak a_{h,\lambda}(\bff e_h^{i+1},\bff e_h^{i+1})
		-
		\mathfrak a_{h,\lambda}(\bff e_h^i,\bff e_h^i)
		\right]
		\notag
		\\
		&=
		\alpha
		\inpro{
			\overline{\bff e}_h^{i+\frac12}
			\times
			\dtt R_h\bff m^{i+1}
		}{
			\dtt\bff e_h^{i+1}
		}_h
		-
		\gamma
		\inpro{
			\overline{\bff e}_h^{i+\frac12}
			\times
			P_h\mathcal H_\lambda(\overline{\bff m}^{i+\frac12})
		}{
			\dtt\bff e_h^{i+1}
		}_h
		-
		\inpro{\bff\rho_h^{i+1}}{\dtt\bff e_h^{i+1}}_h
		\notag
		\\
		&\quad
		-
		\gamma
		\inpro{
			\overline{\bff e}_h^{i+\frac12}
			\times
			\dtt R_h\bff m^{i+1}
		}{
			\mathcal H_{h,\lambda}\overline{\bff e}_h^{i+\frac12}
		}_h
		\notag
		\\
		&\quad
		+
		\frac{\gamma^2}{\alpha}
		\inpro{
			\overline{\bff e}_h^{i+\frac12}
			\times
			P_h\mathcal H_\lambda(\overline{\bff m}^{i+\frac12})
		}{
			\mathcal H_{h,\lambda}\overline{\bff e}_h^{i+\frac12}
		}_h
		+
		\frac{\gamma}{\alpha}
		\inpro{
			\bff\rho_h^{i+1}
		}{
			\mathcal H_{h,\lambda}\overline{\bff e}_h^{i+\frac12}
		}_h
		\notag
		\\
		&=: J_1+J_2+\ldots+J_6.
		\label{eq:combined-error-identity-clean}
	\end{align}
	
	We now estimate each term on the right-hand side of \eqref{eq:combined-error-identity-clean}. Firstly, by H\"older's inequality, the Sobolev embedding $\bb{W}^{1,p}\hookrightarrow \bb{L}^\infty$,
	the bound \eqref{eq:coefficient-W1p-bounds}, the norm equivalence \eqref{eq:lumped-L2-equivalence}, and Young's inequality, we obtain for any $\varepsilon>0$,
	\begin{align}\label{eq:dt-lower-order-est-clean}
		|J_1|+|J_2|
		&\le
		\alpha \norm{\overline{\bff e}_h^{i+\frac12}\times \dtt R_h\bff m^{i+1}}{h} \norm{\dtt\bff e_h^{i+1}}{h}
		+
		\gamma \norm{\overline{\bff e}_h^{i+\frac12}\times P_h\mathcal H_\lambda(\overline{\bff m}^{i+\frac12})}{h} \norm{\dtt\bff e_h^{i+1}}{h}
		\notag
		\\
		&\le
		C \norm{\overline{\bff e}_h^{i+\frac12}}{\bb{L}^2} \norm{\dtt R_h\bff m^{i+1}}{\bb{L}^\infty} \norm{\dtt\bff e_h^{i+1}}{h}
		+ 
		C \norm{\overline{\bff e}_h^{i+\frac12}}{\bb{L}^2} \norm{P_h\mathcal H_\lambda(\overline{\bff m}^{i+\frac12})}{\bb{L}^\infty} \norm{\dtt\bff e_h^{i+1}}{h}
		\notag
		\\
		&\le 
		\varepsilon\norm{\dtt\bff e_h^{i+1}}{h}^2
		+
		C_\varepsilon\norm{\overline{\bff e}_h^{i+\frac12}}{\bb L^2}^2.
	\end{align}
	For the term $J_3$, by \eqref{eq:rho-estimate} and Young's inequality,
	\begin{equation}\label{eq:rho-dt-est-clean}
		|J_3|
		\le
		\varepsilon\norm{\dtt\bff e_h^{i+1}}{h}^2
		+
		C_\varepsilon\norm{\bff\rho_h^{i+1}}{h}^2
		\leq
		\varepsilon\norm{\dtt\bff e_h^{i+1}}{h}^2
		+
		C_\varepsilon (h+k^2)^2.
	\end{equation}
	To estimate the terms $J_4$ and $J_5$,
	we note the following identity. Since
	\[
	\inpro{
		\overline{\bff e}_h^{i+\frac12}
		\times
		\bff z_h
	}{
		\mathcal H_{h,\lambda}\overline{\bff e}_h^{i+\frac12}
	}_h
	=
	\inpro{
		I_h
		\Big(
		\overline{\bff e}_h^{i+\frac12}
		\times
		\bff z_h
		\Big)
	}{
		\mathcal H_{h,\lambda}\overline{\bff e}_h^{i+\frac12}
	}_h, 
	\qquad \forall \bff{z}_h\in \bb{V}_h,
	\]
	we have by the definition of $\mathcal H_{h,\lambda}$ in \eqref{eq:def-discrete-shifted-H},
	\[
	\inpro{
		\overline{\bff e}_h^{i+\frac12}
		\times
		\bff z_h
	}{
		\mathcal H_{h,\lambda}\overline{\bff e}_h^{i+\frac12}
	}_h
	=
	-
	\mathfrak a_{h,\lambda}
	\left(
	\overline{\bff e}_h^{i+\frac12},
	I_h
	\Big(
	\overline{\bff e}_h^{i+\frac12}
	\times
	\bff z_h
	\Big)
	\right).
	\]
	Applying this identity with $\bff z_h=\dtt R_h\bff m^{i+1}$ and
	$\bff z_h=P_h\mathcal H_\lambda(\overline{\bff m}^{i+\frac12})$, then using the discrete product estimate \eqref{eq:a-product-cancellation}, we obtain
	\begin{align}\label{eq:F-product-est-clean}
		|J_4|+|J_5|
		&\le
		C \norm{\overline{\bff e}_h^{i+\frac12}}{\bb H^1}^2 \left(\norm{\dtt R_h\bff m^{i+1}}{\bb{W}^{1,p}} + \norm{P_h\mathcal H_\lambda(\overline{\bff m}^{i+\frac12})}{\bb{W}^{1,p}} \right)
		\le
		C \norm{\overline{\bff e}_h^{i+\frac12}}{\bb H^1}^2.
	\end{align}
	Finally, noting that
	\[
	\inpro{
		\bff\rho_h^{i+1}
	}{
		\mathcal H_{h,\lambda}\overline{\bff e}_h^{i+\frac12}
	}_h
	=
	-\mathfrak a_{h,\lambda}
	\left(
	\overline{\bff e}_h^{i+\frac12},
	\bff\rho_h^{i+1}
	\right),
	\]
	we have by \eqref{eq:rho-estimate},
	\begin{equation}\label{eq:rho-F-est-clean}
	|J_6|
	\le
		C\norm{\overline{\bff e}_h^{i+\frac12}}{\bb H^1}^2
		+
		C\norm{\bff\rho_h^{i+1}}{\bb H^1}^2
	\leq
	C\norm{\overline{\bff e}_h^{i+\frac12}}{\bb H^1}^2
	+
	C(h+k^2)^2.
	\end{equation}
	Now, we substitute these estimates back into \eqref{eq:combined-error-identity-clean}.
	Choosing $\varepsilon>0$ sufficiently small in
	\eqref{eq:dt-lower-order-est-clean} and \eqref{eq:rho-dt-est-clean} and rearranging the terms, we obtain
	\begin{align}\label{eq:one-step-error-clean}
		\frac{k}{2}\norm{\dtt\bff e_h^{i+1}}{h}^2
		+
		\frac{\gamma}{2\alpha}
		\left[
		\mathfrak a_{h,\lambda}(\bff e_h^{i+1},\bff e_h^{i+1})
		-
		\mathfrak a_{h,\lambda}(\bff e_h^i,\bff e_h^i)
		\right]
		&\le
		Ck\norm{\overline{\bff e}_h^{i+\frac12}}{\bb H^1}^2
		+
		Ck(h+k^2)^2.
	\end{align}
	By the coercivity of $\mathfrak a_{h,\lambda}$, we have the estimate
	\[
	\norm{\overline{\bff e}_h^{i+\frac12}}{\bb H^1}^2
	\le
	C
	\left[
	\mathfrak a_{h,\lambda}(\bff e_h^{i+1},\bff e_h^{i+1})
	+
	\mathfrak a_{h,\lambda}(\bff e_h^i,\bff e_h^i)
	\right].
	\]
	Therefore, \eqref{eq:one-step-error-clean} implies
	\begin{align}
		&\frac{k}{2}\norm{\dtt\bff e_h^{i+1}}{h}^2
		+
		\frac{\gamma}{2\alpha}
		\left[
		\mathfrak a_{h,\lambda}(\bff e_h^{i+1},\bff e_h^{i+1})
		-
		\mathfrak a_{h,\lambda}(\bff e_h^i,\bff e_h^i)
		\right]
		\notag
		\\
		&\le
		Ck
		\left[
		\mathfrak a_{h,\lambda}(\bff e_h^{i+1},\bff e_h^{i+1})
		+
		\mathfrak a_{h,\lambda}(\bff e_h^i,\bff e_h^i)
		\right]
		+
		Ck(h+k^2)^2.
		\label{eq:one-step-gro-clean}
	\end{align}
	For $k>0$ sufficiently small, the term
	$Ck \mathfrak a_{h,\lambda}(\bff e_h^{i+1},\bff e_h^{i+1})$
	on the right-hand side can be absorbed into the left-hand side. This yields
	\begin{align*}
		\frac{k}{2}\norm{\dtt\bff e_h^{i+1}}{h}^2
		+
		c
		\mathfrak a_{h,\lambda}(\bff e_h^{i+1},\bff e_h^{i+1})
		\le
		(1+Ck)
		\mathfrak a_{h,\lambda}(\bff e_h^i,\bff e_h^i)
		+
		Ck(h+k^2)^2,
	\end{align*}
	where $c>0$ is independent of $h$ and $k$. In particular, we have
	\begin{equation}\label{eq:energy-recursion-clean}
		\mathfrak a_{h,\lambda}(\bff e_h^{i+1},\bff e_h^{i+1})
		-
		\mathfrak a_{h,\lambda}(\bff e_h^i,\bff e_h^i)
		\le
		Ck
		\mathfrak a_{h,\lambda}(\bff e_h^i,\bff e_h^i)
		+
		Ck(h+k^2)^2.
	\end{equation}
	Applying the discrete Gronwall lemma to \eqref{eq:energy-recursion-clean}
	yields
	\begin{equation}\label{eq:energy-uniform-bound-clean}
		\max_{0\le i\le N}
		\mathfrak a_{h,\lambda}(\bff e_h^i,\bff e_h^i)
		\le
		C
		\left[
		\mathfrak a_{h,\lambda}(\bff e_h^0,\bff e_h^0)
		+
		(h+k^2)^2
		\right].
	\end{equation}
	
	To estimate the accumulated time-difference term, we return to
	\eqref{eq:one-step-gro-clean} and sum it over $i=0,1,\ldots,j-1$. The energy
	increments telescope, and thus by applying \eqref{eq:energy-uniform-bound-clean} and $jk\le T$,
	\begin{align}\label{eq:telescoping-clean}
		&\frac{k}{2}
		\sum_{i=0}^{j-1}
		\norm{\dtt\bff e_h^{i+1}}{h}^2
		+
		\frac{\gamma}{2\alpha}
		\left[
		\mathfrak a_{h,\lambda}(\bff e_h^j,\bff e_h^j)
		-
		\mathfrak a_{h,\lambda}(\bff e_h^0,\bff e_h^0)
		\right]
		\notag
		\\
		&\le
		Ck
		\sum_{i=0}^{j-1}
		\left[
		\mathfrak a_{h,\lambda}(\bff e_h^{i+1},\bff e_h^{i+1})
		+
		\mathfrak a_{h,\lambda}(\bff e_h^i,\bff e_h^i)
		\right]
		+
		Cjk(h+k^2)^2
		\notag
		\\
		&\le
		C
		\left[
		\mathfrak a_{h,\lambda}(\bff e_h^0,\bff e_h^0)
		+
		(h+k^2)^2
		\right].
	\end{align}
	Dropping the nonnegative term
	$\mathfrak a_{h,\lambda}(\bff e_h^j,\bff e_h^j)$ from the left-hand side of
	\eqref{eq:telescoping-clean}, we conclude that
	\begin{equation}\label{eq:dt-error-sum-bound-clean}
		k
		\sum_{i=0}^{j-1}
		\norm{\dtt\bff e_h^{i+1}}{h}^2
		\le
		C
		\left[
		\mathfrak a_{h,\lambda}(\bff e_h^0,\bff e_h^0)
		+
		(h+k^2)^2
		\right].
	\end{equation}
	Combining \eqref{eq:energy-uniform-bound-clean} and
	\eqref{eq:dt-error-sum-bound-clean}, and again utilising the equivalence between
	$\mathfrak a_{h,\lambda}(\cdot,\cdot)$ and $\norm{\cdot}{\bb H^1}^2$ on
	$\bb V_h$, proves \eqref{eq:error-estimate-eh}.
	
	Finally, recall that we set $\bff{m}_h^0=I_h \bff{m}^0$. By the approximation properties of $I_h$ and $R_h$ in \eqref{eq:Ih-H1-approx} and \eqref{eq:DMI-Ritz-H1-est}, we have
	\begin{align}\label{eq:eh0-est}
	\norm{\bff{e}_h^0}{\bb{H}^1}^2
		\leq
		C \norm{\bff{m}^0-I_h \bff{m}^0}{\bb{H}^1}^2 + C \norm{\bff{m}^0-R_h \bff{m}^0}{\bb{H}^1}^2
		\leq
		Ch^2.
	\end{align}
	This, together with \eqref{eq:error-estimate-eh}, the error decomposition \eqref{eq:error-splitting-total}, and \eqref{eq:DMI-Ritz-H1-est}, yields \eqref{eq:error-estimate-total}.
\end{proof}

\section{The inexact midpoint scheme and its error analysis}
\label{sec:inexact-solver}

At each time step, the ideal midpoint scheme \eqref{eq:midpoint-scheme} requires the solution of a nonlinear algebraic system. A generic nonlinear solver need not preserve the geometric structure or the energy stability inherited from the midpoint discretisation. We therefore introduce in Algorithm~\ref{alg:constraint-preserving-fixed-point} a constraint-preserving fixed-point iteration, formulated in terms of the midpoint unknown
\[
\bff{\eta}_h^{i+\frac12}
:=
\overline{\bff{m}}_h^{i+\frac12}
=
\frac12(\bff{m}_h^{i+1}+\bff{m}_h^i),
\]
which is designed to preserve the nodal unit-length constraint and to satisfy an inexact discrete energy law.

Since
\[
\dtt\bff{m}_h^{i+1}
=
\frac{2}{k}
\left(
\bff{\eta}_h^{i+\frac12}-\bff{m}_h^i
\right),
\]
the exact midpoint scheme \eqref{eq:midpoint-scheme} is equivalent to finding
$\bff{\eta}_h^{i+\frac12}\in\bb V_h$ such that, for all $\bff{\phi}_h\in\bb V_h$,
\begin{align}\label{eq:midpoint-eta-form}
	&\inpro{\bff{\eta}_h^{i+\frac12}}{\bff{\phi}_h}_h
	+
	\alpha
	\inpro{
		\bff{\eta}_h^{i+\frac12}\times\bff{m}_h^i
	}{
		\bff{\phi}_h
	}_h
	+
	\frac{\gamma k}{2}
	\inpro{
		\bff{\eta}_h^{i+\frac12}
		\times
		\mathcal H_{h,\lambda}\bff{\eta}_h^{i+\frac12}
	}{
		\bff{\phi}_h
	}_h
	=
	\inpro{\bff{m}_h^i}{\bff{\phi}_h}_h.
\end{align}
After that, we set
\[
\bff{m}_h^{i+1}
=
2\bff{\eta}_h^{i+\frac12}-\bff{m}_h^i.
\]

Inspired by \cite{Cim09, FraPfePraRug23}, we now define a practical inexact solver to solve \eqref{eq:midpoint-eta-form} in the following algorithm. Recall that $\mathcal{M}_h$ is the set defined by \eqref{eq:Mh}.

\begin{algorithm}[Constraint-preserving inexact midpoint solver]
	\label{alg:constraint-preserving-fixed-point}
	Set $\bff{m}_{h,\varepsilon}^0:= \bff{m}_h^0= I_h \bff{m}^0$.
	\\
	\textbf{For} $i=0$ to $N-1$, given $\bff{m}_{h,\varepsilon}^i\in\mathcal M_h$ and a prescribed tolerance $\varepsilon_{i+1}>0$, \textbf{do}:
	\begin{enumerate}
		\item Set
		\[
		\bff{\eta}_{h,\varepsilon}^{i,0}
		:=
		\bff{m}_{h,\varepsilon}^i.
		\]
		
		\item For $\ell=0,1,2,\ldots$, given
		$\bff{\eta}_{h,\varepsilon}^{i,\ell}\in\bb V_h$, compute
		$\bff{\eta}_{h,\varepsilon}^{i,\ell+1}\in\bb V_h$ such that, for all $\bff{\phi}_h\in\bb V_h$,
		\begin{align}
			&\inpro{\bff{\eta}_{h,\varepsilon}^{i,\ell+1}}{\bff{\phi}_h}_h
			+
			\alpha
			\inpro{
				\bff{\eta}_{h,\varepsilon}^{i,\ell+1}
				\times
				\bff{m}_{h,\varepsilon}^i
			}{
				\bff{\phi}_h
			}_h
			+
			\frac{\gamma k}{2}
			\inpro{
				\bff{\eta}_{h,\varepsilon}^{i,\ell+1}
				\times
				\mathcal H_{h,\lambda}
				\bff{\eta}_{h,\varepsilon}^{i,\ell}
			}{
				\bff{\phi}_h
			}_h
			=
			\inpro{\bff{m}_{h,\varepsilon}^i}{\bff{\phi}_h}_h.
			\label{eq:constraint-preserving-fixed-point}
		\end{align}
		
		\item Stop at the first index $\ell_\ast\ge0$ satisfying
		\begin{equation}\label{eq:constraint-preserving-stopping}
			\norm{
				I_h
				\left[
				\bff{\eta}_{h,\varepsilon}^{i,\ell_\ast+1}
				\times
				\mathcal H_{h,\lambda}
				\left(
				\bff{\eta}_{h,\varepsilon}^{i,\ell_\ast+1}
				-
				\bff{\eta}_{h,\varepsilon}^{i,\ell_\ast}
				\right)
				\right]
			}{h}
			\le
			\varepsilon_{i+1}.
		\end{equation}
		
		\item Define
		\begin{equation}\label{eq:def-inexact-midpoint-update}
			\overline{\bff{m}}_{h,\varepsilon}^{i+\frac12}
			:=
			\bff{\eta}_{h,\varepsilon}^{i,\ell_\ast+1},
			\qquad
			\bff{m}_{h,\varepsilon}^{i+1}
			:=
			2\overline{\bff{m}}_{h,\varepsilon}^{i+\frac12}
			-
			\bff{m}_{h,\varepsilon}^i.
		\end{equation}
	\end{enumerate}
	\textbf{Output}: a sequence of approximations $\{\bff{m}_{h,\varepsilon}^i\}_{1\leq i\leq N}$.
\end{algorithm}

For the analysis, we also define the defect quantities:
\begin{equation}\label{eq:def-ri-si}
	\bff{r}_{h,\varepsilon}^{i+1}
	:=
	\mathcal H_{h,\lambda}
	\left(
	\bff{\eta}_{h,\varepsilon}^{i,\ell_\ast+1}
	-
	\bff{\eta}_{h,\varepsilon}^{i,\ell_\ast}
	\right),
	\qquad
	\bff{s}_{h,\varepsilon}^{i+1}
	:=
	I_h
	\left[
	\overline{\bff{m}}_{h,\varepsilon}^{i+\frac12}
	\times
	\bff{r}_{h,\varepsilon}^{i+1}
	\right].
\end{equation}
By the stopping criterion \eqref{eq:constraint-preserving-stopping} and the
definition \eqref{eq:def-inexact-midpoint-update}, we immediately have
\begin{equation}\label{eq:s-epsilon-bound}
	\norm{\bff{s}_{h,\varepsilon}^{i+1}}{h}
	\le
	\varepsilon_{i+1}.
\end{equation}

\begin{remark}\label{rem:shift-in-fixed-point}
	The shifted operator $\mathcal H_{h,\lambda}$ is still used here to align the
	fixed-point solver with the shifted formulation used in the error analysis in
	Section~\ref{sec:error-analysis}. The exact midpoint scheme is unchanged by
	this shift due to~\eqref{eq:shift-do-not-change}.
	
	At the level of the inexact inner iteration, however, the shift is not
	invisible. Indeed, $\bff\eta_h^{\ell+1}\times
	\lambda\bff\eta_h^\ell$
	does not vanish in general, since $\bff\eta_h^{\ell+1}$ and
	$\bff\eta_h^\ell$ need not coincide. Thus the value of $\lambda$ changes the
	algebraic path of the fixed-point iteration and enters the defect term in the
	inexact discrete energy law through
	\[
	\bff\eta_h^{\ell_\ast+1}
	\times
	\mathcal H_{h,\lambda}
	\left(
	\bff\eta_h^{\ell_\ast+1}
	-
	\bff\eta_h^{\ell_\ast}
	\right).
	\]
	In particular, the constants in the contraction and error estimates may
	depend on the fixed shift parameter $\lambda$, but remain independent of
	$h$ and $k$.
\end{remark}

We first verify that the inner iteration is well-defined and that the update
preserves the nodal unit-length constraint exactly. This is the key structural
advantage of the iteration defined by Algorithm~\ref{alg:constraint-preserving-fixed-point}.

\begin{lemma}[Well-posedness and nodal constraint of the inner iteration]
	\label{lem:constraint-preserving-inner-iteration}
	Assume that $\bff m_{h,\varepsilon}^i\in\mathcal M_h$.
	Then, for every $\ell\ge0$, the linear problem
	\eqref{eq:constraint-preserving-fixed-point} admits a unique solution
	$\bff\eta_{h,\varepsilon}^{i,\ell+1}\in\bb V_h$. Moreover,
	\begin{equation}\label{eq:eta-linfty-bound}
		\norm{\bff\eta_{h,\varepsilon}^{i,\ell+1}}{\bb L^\infty}
		\le 1.
	\end{equation}
	If the iteration is stopped according to
	\eqref{eq:constraint-preserving-stopping}, then the update
	\eqref{eq:def-inexact-midpoint-update} satisfies $\bff m_{h,\varepsilon}^{i+1}\in\mathcal M_h$.
\end{lemma}

\begin{proof}
	For fixed $\bff\eta_{h,\varepsilon}^{i,\ell}$, the left-hand side of
	\eqref{eq:constraint-preserving-fixed-point} defines a bilinear form in
	$\bff\eta_{h,\varepsilon}^{i,\ell+1}$ and $\bff\phi_h$. Taking
	$\bff\phi_h=\bff\eta_{h,\varepsilon}^{i,\ell+1}$, the cross-product terms vanish
	nodally.
	Therefore, the bilinear form is coercive with respect to $\norm{\cdot}{h}$,
	and the Lax--Milgram theorem yields existence and uniqueness.
	
	Next, let $z\in\mathcal N_h$. Testing
	\eqref{eq:constraint-preserving-fixed-point} with $\bff\phi_h
	=
	\varphi_z
	\bff\eta_{h,\varepsilon}^{i,\ell+1}(z)$ gives
	\[
	\beta_z
	|\bff\eta_{h,\varepsilon}^{i,\ell+1}(z)|^2
	=
	\beta_z
	\bff m_{h,\varepsilon}^{i}(z)
	\cdot
	\bff\eta_{h,\varepsilon}^{i,\ell+1}(z),
	\]
	Noting that $|\bff m_{h,\varepsilon}^i(z)|=1$, we obtain $\abs{\bff\eta_{h,\varepsilon}^{i,\ell+1}(z)} \le 1$ for all $z\in\mathcal N_h$.
	Since $\bff\eta_{h,\varepsilon}^{i,\ell+1}$ is affine on each element,
	\eqref{eq:eta-linfty-bound} follows.
	
	It remains to prove that the update belongs to $\mathcal M_h$. Since
	\[
	\bff m_{h,\varepsilon}^{i+1}(z)
	=
	2\bff\eta_{h,\varepsilon}^{i,\ell_\ast+1}(z)
	-
	\bff m_{h,\varepsilon}^i(z),
	\]
	we compute
	\begin{align*}
		|\bff m_{h,\varepsilon}^{i+1}(z)|^2
		&=
		4|\bff\eta_{h,\varepsilon}^{i,\ell_\ast+1}(z)|^2
		-
		4\bff\eta_{h,\varepsilon}^{i,\ell_\ast+1}(z)
		\cdot
		\bff m_{h,\varepsilon}^i(z)
		+
		|\bff m_{h,\varepsilon}^i(z)|^2
		=
		|\bff m_{h,\varepsilon}^i(z)|^2
		=
		1.
	\end{align*}
	This proves $\bff m_{h,\varepsilon}^{i+1}\in\mathcal M_h$, thus completing the proof of the lemma.
\end{proof}

Before deriving the perturbed midpoint equation satisfied by the stopped iterate,
we first show that the inner fixed-point iteration is contractive under the usual
condition $k=O(h^2)$ with a sufficiently small constant. The proof follows a similar structure as the contraction argument for the constraint-preserving midpoint
iteration in \cite{BarPro06, Cim09, FraPfePraRug23}, adapted here to the shifted discrete field
$\mathcal H_{h,\lambda}$.

\begin{lemma}[Contraction of the inner iteration]
	\label{lem:constraint-preserving-contraction}
	There exist constants $h_0>0$, $c_{\mathrm{FL}}>0$, and $q\in(0,1)$,
	independent of $h,k,i,\ell$, such that, for all $h\in (0,h_0)$, if
	\begin{equation}\label{eq:inner-cfl-condition}
		k\le c_{\mathrm{FL}}h^2,
	\end{equation}
	then
	\begin{equation}\label{eq:inner-contraction}
		\norm{
			\bff\eta_{h,\varepsilon}^{i,\ell+2}
			-
			\bff\eta_{h,\varepsilon}^{i,\ell+1}
		}{h}
		\le
		q
		\norm{
			\bff\eta_{h,\varepsilon}^{i,\ell+1}
			-
			\bff\eta_{h,\varepsilon}^{i,\ell}
		}{h},
		\qquad
		\forall \ell\ge0.
	\end{equation}
	Consequently, for every $\varepsilon_{i+1}>0$, the stopping criterion
	\eqref{eq:constraint-preserving-stopping} is satisfied after finitely many
	iterations.
\end{lemma}

\begin{proof}
	Set $\bff\delta_h^\ell
	:=
	\bff\eta_{h,\varepsilon}^{i,\ell+1}
	-
	\bff\eta_{h,\varepsilon}^{i,\ell}$.
	Subtracting \eqref{eq:constraint-preserving-fixed-point} at the iterations
	$\ell+1$ and $\ell$ gives, for all $\bff\phi_h\in\bb V_h$,
	\begin{align}
		&\inpro{\bff\delta_h^{\ell+1}}{\bff\phi_h}_h
		+
		\alpha
		\inpro{
			\bff\delta_h^{\ell+1}
			\times
			\bff m_{h,\varepsilon}^i
		}{
			\bff\phi_h
		}_h
		+
		\frac{\gamma k}{2}
		\inpro{
			\bff\delta_h^{\ell+1}
			\times
			\mathcal H_{h,\lambda}
			\bff\eta_{h,\varepsilon}^{i,\ell+1}
		}{
			\bff\phi_h
		}_h
		\notag
		\\
		&\qquad
		=
		-
		\frac{\gamma k}{2}
		\inpro{
			\bff\eta_{h,\varepsilon}^{i,\ell+1}
			\times
			\mathcal H_{h,\lambda}
			\bff\delta_h^\ell
		}{
			\bff\phi_h
		}_h .
		\label{eq:inner-difference-equation}
	\end{align}
	Choosing $\bff\phi_h=\bff\delta_h^{\ell+1}$, the two cross-product terms on
	the left-hand side vanish nodally. Therefore,
	\[
	\norm{\bff\delta_h^{\ell+1}}{h}^2
	\le
	Ck
	\norm{
		\mathcal H_{h,\lambda}\bff\delta_h^\ell
	}{h}
	\norm{\bff\delta_h^{\ell+1}}{h},
	\]
	where we used the nodal bound \eqref{eq:eta-linfty-bound}.
	We next use the inverse estimate (see~\cite{BarPro06}):
	\begin{equation}\label{eq:Hh-inverse-hnorm}
		\norm{\mathcal H_{h,\lambda}\bff v_h}{h}
		\le
		Ch^{-2}\norm{\bff v_h}{h},
		\qquad
		\forall \bff v_h\in\bb V_h
	\end{equation}
	to infer that
	\[
	\norm{\bff\delta_h^{\ell+1}}{h}
	\le
	Ckh^{-2}\norm{\bff\delta_h^\ell}{h}.
	\]
	By choosing $c_{\mathrm{FL}} >0$ sufficiently small in \eqref{eq:inner-cfl-condition},
	we obtain $Ckh^{-2}\le q<1$, and thus
	\eqref{eq:inner-contraction} follows.
	
	It remains to prove that the stopping criterion is reached in finitely many
	steps. Since the lumped norm is nodal, we have
	\begin{align*}
		\norm{
			I_h
			\left[
			\bff\eta_{h,\varepsilon}^{i,\ell+1}
			\times
			\mathcal H_{h,\lambda}
			(\bff\eta_{h,\varepsilon}^{i,\ell+1}
			-
			\bff\eta_{h,\varepsilon}^{i,\ell})
			\right]
		}{h}^2
		&=
		\sum_{z\in\mathcal N_h}
		\beta_z
		\left|
		\bff\eta_{h,\varepsilon}^{i,\ell+1}(z)
		\times
		\mathcal H_{h,\lambda}
		(\bff\eta_{h,\varepsilon}^{i,\ell+1}
		-
		\bff\eta_{h,\varepsilon}^{i,\ell})(z)
		\right|^2
		\\
		&\le
		\sum_{z\in\mathcal N_h}
		\beta_z
		\left|
		\mathcal H_{h,\lambda}
		(\bff\eta_{h,\varepsilon}^{i,\ell+1}
		-
		\bff\eta_{h,\varepsilon}^{i,\ell})(z)
		\right|^2
		\\
		&=
		\norm{
			\mathcal H_{h,\lambda}
			(\bff\eta_{h,\varepsilon}^{i,\ell+1}
			-
			\bff\eta_{h,\varepsilon}^{i,\ell})
		}{h}^2,
	\end{align*}
	where we again used \eqref{eq:eta-linfty-bound}.
	Thanks to \eqref{eq:Hh-inverse-hnorm}, we obtain
	\[
	\norm{
		I_h
		\left[
		\bff\eta_{h,\varepsilon}^{i,\ell+1}
		\times
		\mathcal H_{h,\lambda}
		(\bff\eta_{h,\varepsilon}^{i,\ell+1}
		-
		\bff\eta_{h,\varepsilon}^{i,\ell})
		\right]
	}{h}
	\le
	Ch^{-2}
	\norm{
		\bff\eta_{h,\varepsilon}^{i,\ell+1}
		-
		\bff\eta_{h,\varepsilon}^{i,\ell}
	}{h}.
	\]
	By \eqref{eq:inner-contraction},
	\[
	\norm{
		\bff\eta_{h,\varepsilon}^{i,\ell+1}
		-
		\bff\eta_{h,\varepsilon}^{i,\ell}
	}{h}
	\le
	q^\ell
	\norm{
		\bff\eta_{h,\varepsilon}^{i,1}
		-
		\bff\eta_{h,\varepsilon}^{i,0}
	}{h}.
	\]
	Thus the left-hand side of \eqref{eq:constraint-preserving-stopping}
	tends to zero geometrically as $\ell\to\infty$. Since
	$\varepsilon_{i+1}>0$, the stopping criterion is reached after finitely
	many iterations.
\end{proof}

Having shown that the inner iteration terminates under the time-step condition
in Lemma~\ref{lem:constraint-preserving-contraction}, we next identify the
equation satisfied by the stopped iterate. This equation is the ideal midpoint
scheme perturbed by an algebraic residual which is controlled directly by the
stopping criterion.

\begin{lemma}
	\label{lem:practical-perturbed-scheme}
	Let $\{\bff m_{h,\varepsilon}^i\}_{i=0}^N$ be generated by
	Algorithm~\ref{alg:constraint-preserving-fixed-point}. Then, for each $i\in \{0,\ldots,N-1\}$ and every $\bff\phi_h\in\bb V_h$,
	\begin{align}
		&\inpro{\dtt\bff m_{h,\varepsilon}^{i+1}}{\bff\phi_h}_h
		-
		\alpha
		\inpro{
			\overline{\bff m}_{h,\varepsilon}^{i+\frac12}
			\times
			\dtt\bff m_{h,\varepsilon}^{i+1}
		}{
			\bff\phi_h
		}_h
		+
		\gamma
		\inpro{
			\overline{\bff m}_{h,\varepsilon}^{i+\frac12}
			\times
			\mathcal H_{h,\lambda}
			\overline{\bff m}_{h,\varepsilon}^{i+\frac12}
		}{
			\bff\phi_h
		}_h
		=
		\gamma
		\inpro{
			\bff s_{h,\varepsilon}^{i+1}
		}{
			\bff\phi_h
		}_h,
		\label{eq:practical-perturbed-scheme}
	\end{align}
	where $\bff s_{h,\varepsilon}^{i+1}$ is defined by \eqref{eq:def-ri-si}, satisfying \eqref{eq:s-epsilon-bound}.
\end{lemma}

\begin{proof}
	By \eqref{eq:def-inexact-midpoint-update},
	\[
	\overline{\bff m}_{h,\varepsilon}^{i+\frac12}
	=
	\bff\eta_{h,\varepsilon}^{i,\ell_\ast+1},
	\qquad
	\dtt\bff m_{h,\varepsilon}^{i+1}
	=
	\frac{2}{k}
	\left(
	\bff\eta_{h,\varepsilon}^{i,\ell_\ast+1}
	-
	\bff m_{h,\varepsilon}^{i}
	\right).
	\]
	Hence, we have
	\[
	\bff\eta_{h,\varepsilon}^{i,\ell_\ast+1}
	\times
	\dtt\bff m_{h,\varepsilon}^{i+1}
	=
	-\frac{2}{k}
	\bff\eta_{h,\varepsilon}^{i,\ell_\ast+1}
	\times
	\bff m_{h,\varepsilon}^{i}.
	\]
	Multiplying \eqref{eq:constraint-preserving-fixed-point} with
	$\ell=\ell_\ast$ by $2/k$ then gives
	\begin{align}
		&\inpro{\dtt\bff m_{h,\varepsilon}^{i+1}}{\bff\phi_h}_h
		-
		\alpha
		\inpro{
			\overline{\bff m}_{h,\varepsilon}^{i+\frac12}
			\times
			\dtt\bff m_{h,\varepsilon}^{i+1}
		}{
			\bff\phi_h
		}_h
		+
		\gamma
		\inpro{
			\overline{\bff m}_{h,\varepsilon}^{i+\frac12}
			\times
			\mathcal H_{h,\lambda}
			\bff\eta_{h,\varepsilon}^{i,\ell_\ast}
		}{
			\bff\phi_h
		}_h
		=
		0.
		\label{eq:stopped-scheme-old-field-short}
	\end{align}
	Adding and subtracting the fully implicit field contribution gives
	\begin{align}
		&\inpro{\dtt\bff m_{h,\varepsilon}^{i+1}}{\bff\phi_h}_h
		-
		\alpha
		\inpro{
			\overline{\bff m}_{h,\varepsilon}^{i+\frac12}
			\times
			\dtt\bff m_{h,\varepsilon}^{i+1}
		}{
			\bff\phi_h
		}_h
		+
		\gamma
		\inpro{
			\overline{\bff m}_{h,\varepsilon}^{i+\frac12}
			\times
			\mathcal H_{h,\lambda}
			\overline{\bff m}_{h,\varepsilon}^{i+\frac12}
		}{
			\bff\phi_h
		}_h
		\notag
		\\
		&\qquad
		=
		\gamma
		\inpro{
			\overline{\bff m}_{h,\varepsilon}^{i+\frac12}
			\times
			\mathcal H_{h,\lambda}
			\left(
			\bff\eta_{h,\varepsilon}^{i,\ell_\ast+1}
			-
			\bff\eta_{h,\varepsilon}^{i,\ell_\ast}
			\right)
		}{
			\bff\phi_h
		}_h .
		\label{eq:perturbed-scheme-before-sh-short}
	\end{align}
	Using the notations in \eqref{eq:def-ri-si}, and using the fact that the mass-lumped inner product is nodal,
	\[
	\inpro{
		\overline{\bff m}_{h,\varepsilon}^{i+\frac12}
		\times
		\bff r_{h,\varepsilon}^{i+1}
	}{
		\bff\phi_h
	}_h
	=
	\inpro{
		\bff s_{h,\varepsilon}^{i+1}
	}{
		\bff\phi_h
	}_h,
	\qquad
	\forall \bff\phi_h\in\bb V_h.
	\]
	Therefore \eqref{eq:perturbed-scheme-before-sh-short} implies
	\eqref{eq:practical-perturbed-scheme}, as required.
\end{proof}

The identity above shows that the stopped iterate solves the ideal midpoint
scheme up to the explicitly controlled residual
$\gamma\bff s_{h,\varepsilon}^{i+1}$. We now derive the corresponding energy
identity, which is the starting point for the stability and error estimates.

\begin{lemma}[Perturbed energy identity]
	\label{lem:practical-energy-identity}
	Let $\{\bff m_{h,\varepsilon}^i\}_{i=0}^N$ be generated by
	Algorithm~\ref{alg:constraint-preserving-fixed-point}. Then, for each
	$i=0, 1,\ldots,N-1$,
	\begin{align}
		&\norm{\dtt\bff m_{h,\varepsilon}^{i+1}}{h}^2
		+
		\frac{\gamma}{2\alpha k}
		\left[
		\mathfrak a_{h,\lambda}
		(\bff m_{h,\varepsilon}^{i+1},\bff m_{h,\varepsilon}^{i+1})
		-
		\mathfrak a_{h,\lambda}
		(\bff m_{h,\varepsilon}^i,\bff m_{h,\varepsilon}^i)
		\right]
		\notag
		\\
		&\qquad
		=
		\gamma
		\inpro{
			\bff s_{h,\varepsilon}^{i+1}
		}{
			\dtt\bff m_{h,\varepsilon}^{i+1}
		}_h
		-
		\frac{\gamma^2}{\alpha}
		\inpro{
			\bff s_{h,\varepsilon}^{i+1}
		}{
			\mathcal H_{h,\lambda}
			\overline{\bff m}_{h,\varepsilon}^{i+\frac12}
		}_h .
		\label{eq:practical-perturbed-energy}
	\end{align}
	In particular, the energy identity \eqref{eq:discrete-energy-stability} is recovered whenever
	$\bff s_{h,\varepsilon}^{i+1}=\bff0$.
\end{lemma}

\begin{proof}
	By Lemma~\ref{lem:practical-perturbed-scheme}, the stopped iterate satisfies
	\eqref{eq:practical-perturbed-scheme}. Taking
	$\bff\phi_h=\dtt\bff m_{h,\varepsilon}^{i+1}$ in 	\eqref{eq:practical-perturbed-scheme} gives
	\begin{align}
		&\norm{\dtt\bff m_{h,\varepsilon}^{i+1}}{h}^2
		+
		\gamma
		\inpro{
			\overline{\bff m}_{h,\varepsilon}^{i+\frac12}
			\times
			\mathcal H_{h,\lambda}
			\overline{\bff m}_{h,\varepsilon}^{i+\frac12}
		}{
			\dtt\bff m_{h,\varepsilon}^{i+1}
		}_h
		=
		\gamma
		\inpro{
			\bff s_{h,\varepsilon}^{i+1}
		}{
			\dtt\bff m_{h,\varepsilon}^{i+1}
		}_h.
		\label{eq:practical-energy-test-dt}
	\end{align}
	Next, taking $\bff\phi_h
	=
	-\frac{\gamma}{\alpha}
	\mathcal H_{h,\lambda}
	\overline{\bff m}_{h,\varepsilon}^{i+\frac12}$
	in \eqref{eq:practical-perturbed-scheme} gives
	\begin{align}
		&-\frac{\gamma}{\alpha}
		\inpro{
			\dtt\bff m_{h,\varepsilon}^{i+1}
		}{
			\mathcal H_{h,\lambda}
			\overline{\bff m}_{h,\varepsilon}^{i+\frac12}
		}_h
		-
		\gamma
		\inpro{
			\overline{\bff m}_{h,\varepsilon}^{i+\frac12}
			\times
			\mathcal H_{h,\lambda}
			\overline{\bff m}_{h,\varepsilon}^{i+\frac12}
		}{
			\dtt\bff m_{h,\varepsilon}^{i+1}
		}_h
		\notag
		\\
		&\qquad
		=
		-\frac{\gamma^2}{\alpha}
		\inpro{
			\bff s_{h,\varepsilon}^{i+1}
		}{
			\mathcal H_{h,\lambda}
			\overline{\bff m}_{h,\varepsilon}^{i+\frac12}
		}_h .
		\label{eq:practical-energy-test-field}
	\end{align}
	We now add \eqref{eq:practical-energy-test-dt} and
	\eqref{eq:practical-energy-test-field} to observe that the leading precession terms cancel.
	Moreover, by the definition of $\mathcal H_{h,\lambda}$,
	\[
	\inpro{
		\dtt\bff m_{h,\varepsilon}^{i+1}
	}{
		\mathcal H_{h,\lambda}
		\overline{\bff m}_{h,\varepsilon}^{i+\frac12}
	}_h
	=
	-\mathfrak a_{h,\lambda}
	\left(
	\overline{\bff m}_{h,\varepsilon}^{i+\frac12},
	\dtt\bff m_{h,\varepsilon}^{i+1}
	\right).
	\]
	Since $\mathfrak a_{h,\lambda}$ is symmetric,
	\[
	\mathfrak a_{h,\lambda}
	\left(
	\overline{\bff m}_{h,\varepsilon}^{i+\frac12},
	\dtt\bff m_{h,\varepsilon}^{i+1}
	\right)
	=
	\frac1{2k}
	\left[
	\mathfrak a_{h,\lambda}
	(\bff m_{h,\varepsilon}^{i+1},\bff m_{h,\varepsilon}^{i+1})
	-
	\mathfrak a_{h,\lambda}
	(\bff m_{h,\varepsilon}^i,\bff m_{h,\varepsilon}^i)
	\right].
	\]
	This proves \eqref{eq:practical-perturbed-energy}.
\end{proof}

The identity \eqref{eq:practical-perturbed-energy} is no longer a pure
dissipation law because of the terms involving
$\bff s_{h,\varepsilon}^{i+1}$. However, the stopping criterion gives
$\norm{\bff s_{h,\varepsilon}^{i+1}}{h}\le\varepsilon_{i+1}$, which is enough to
recover an $\bb H^1$-stability estimate under a mild smallness condition on the
accumulated solver tolerances.

\begin{lemma}[$\bb H^1$-stability of the inexact scheme]
	\label{lem:inexact-H1-stability}
	Let $\{\bff m_{h,\varepsilon}^i\}_{i=0}^N$ be generated by
	Algorithm~\ref{alg:constraint-preserving-fixed-point}. Assume that
	\begin{equation}\label{eq:solver-tolerance-stability-condition}
		k\sum_{i=0}^{N-1}h^{-2}\varepsilon_{i+1}^2
		\le
		C_{\mathrm{st}}.
	\end{equation}
	Then, for $k>0$ sufficiently small,
	\begin{equation}\label{eq:inexact-H1-stability}
		\max_{0\le j\le N}
		\norm{\bff m_{h,\varepsilon}^j}{\bb H^1}^2
		+
		k\sum_{i=0}^{N-1}
		\norm{\dtt\bff m_{h,\varepsilon}^{i+1}}{h}^2
		\le
		C
		\left(
		\norm{\bff m_{h,\varepsilon}^0}{\bb H^1}^2
		+
		C_{\mathrm{st}}
		\right).
	\end{equation}
	In particular, for a uniform tolerance
	$\varepsilon_{i+1}=\varepsilon_{\rm fp}$, condition
	\eqref{eq:solver-tolerance-stability-condition} is satisfied if $\varepsilon_{\rm fp}= O(h)$.
\end{lemma}

\begin{proof}
	We estimate the two residual terms in the perturbed energy identity
	\eqref{eq:practical-perturbed-energy}. First, by the Cauchy--Schwarz inequality and \eqref{eq:s-epsilon-bound},
	\[
	\gamma
	\left|
	\inpro{
		\bff s_{h,\varepsilon}^{i+1}
	}{
		\dtt\bff m_{h,\varepsilon}^{i+1}
	}_h
	\right|
	\le
	\frac14
	\norm{\dtt\bff m_{h,\varepsilon}^{i+1}}{h}^2
	+
	C\varepsilon_{i+1}^2.
	\]
	For the second residual term, since
	$\bff s_{h,\varepsilon}^{i+1}\in\bb V_h$, the identity \eqref{eq:def-discrete-shifted-H} gives
	\[
	\inpro{
		\bff s_{h,\varepsilon}^{i+1}
	}{
		\mathcal H_{h,\lambda}
		\overline{\bff m}_{h,\varepsilon}^{i+\frac12}
	}_h
	=
	-\mathfrak a_{h,\lambda}
	\left(
	\overline{\bff m}_{h,\varepsilon}^{i+\frac12},
	\bff s_{h,\varepsilon}^{i+1}
	\right).
	\]
	We note that by the inverse estimate \eqref{eq:inverse-estimate},
	\[
	\norm{\bff s_{h,\varepsilon}^{i+1}}{\bb H^1}
	\le
	Ch^{-1}\norm{\bff s_{h,\varepsilon}^{i+1}}{h}
	\le
	Ch^{-1}\varepsilon_{i+1}.
	\]
	Hence, by the continuity of $\mathfrak a_{h,\lambda}$ and Young's inequality, we obtain
	\begin{align*}
		\left|
		\inpro{
			\bff s_{h,\varepsilon}^{i+1}
		}{
			\mathcal H_{h,\lambda}
			\overline{\bff m}_{h,\varepsilon}^{i+\frac12}
		}_h
		\right|
		&\le
		C
		\norm{\overline{\bff m}_{h,\varepsilon}^{i+\frac12}}{\bb H^1}
		\norm{\bff s_{h,\varepsilon}^{i+1}}{\bb H^1}
		\\
		&\le
		Ch^{-1}\varepsilon_{i+1}
		\norm{\overline{\bff m}_{h,\varepsilon}^{i+\frac12}}{\bb H^1}
		\\
		&\le
		C\norm{\overline{\bff m}_{h,\varepsilon}^{i+\frac12}}{\bb H^1}^2
		+
		Ch^{-2}\varepsilon_{i+1}^2 .
	\end{align*}
	Substituting these estimates into
	\eqref{eq:practical-perturbed-energy} and multiplying by $k$ yields
	\begin{align*}
		&\frac{k}{2}
		\norm{\dtt\bff m_{h,\varepsilon}^{i+1}}{h}^2
		+
		\frac{\gamma}{2\alpha}
		\left[
		\mathfrak a_{h,\lambda}
		(\bff m_{h,\varepsilon}^{i+1},\bff m_{h,\varepsilon}^{i+1})
		-
		\mathfrak a_{h,\lambda}
		(\bff m_{h,\varepsilon}^i,\bff m_{h,\varepsilon}^i)
		\right]
		\le
		Ck
		\norm{\overline{\bff m}_{h,\varepsilon}^{i+\frac12}}{\bb H^1}^2
		+
		Ck h^{-2}\varepsilon_{i+1}^2.
	\end{align*}
	Now, we aim to estimate the terms on the right-hand side.
	By the coercivity and continuity of $\mathfrak a_{h,\lambda}$,
	\[
	\norm{\overline{\bff m}_{h,\varepsilon}^{i+\frac12}}{\bb H^1}^2
	\le
	C
	\left[
	\mathfrak a_{h,\lambda}
	(\bff m_{h,\varepsilon}^{i+1},\bff m_{h,\varepsilon}^{i+1})
	+
	\mathfrak a_{h,\lambda}
	(\bff m_{h,\varepsilon}^{i},\bff m_{h,\varepsilon}^{i})
	\right].
	\]
	Therefore, we infer that
	\begin{align}
		&\frac{k}{2}
		\norm{\dtt\bff m_{h,\varepsilon}^{i+1}}{h}^2
		+
		\frac{\gamma}{2\alpha}
		\left[
		\mathfrak a_{h,\lambda}
		(\bff m_{h,\varepsilon}^{i+1},\bff m_{h,\varepsilon}^{i+1})
		-
		\mathfrak a_{h,\lambda}
		(\bff m_{h,\varepsilon}^i,\bff m_{h,\varepsilon}^i)
		\right]
		\notag
		\\
		&\qquad
		\le
		Ck
		\left[
		\mathfrak a_{h,\lambda}
		(\bff m_{h,\varepsilon}^{i+1},\bff m_{h,\varepsilon}^{i+1})
		+
		\mathfrak a_{h,\lambda}
		(\bff m_{h,\varepsilon}^{i},\bff m_{h,\varepsilon}^{i})
		\right]
		+
		Ck h^{-2}\varepsilon_{i+1}^2.
		\label{eq:inexact-stability-gro}
	\end{align}
	For $k>0$ sufficiently small, the next-step energy term on the right-hand
	side of \eqref{eq:inexact-stability-gro} can be absorbed into the left-hand side. By a similar argument as in \eqref{eq:energy-recursion-clean} and applying the discrete Gronwall lemma, we obtain
	\[
	\max_{0\le j\le N}
	\mathfrak a_{h,\lambda}
	(\bff m_{h,\varepsilon}^{j},\bff m_{h,\varepsilon}^{j})
	\le
	C
	\left(
	\mathfrak a_{h,\lambda}
	(\bff m_{h,\varepsilon}^{0},\bff m_{h,\varepsilon}^{0})
	+
	k\sum_{i=0}^{N-1}h^{-2}\varepsilon_{i+1}^2
	\right).
	\]
	Summing \eqref{eq:inexact-stability-gro} over $i=0,\ldots,j-1$, using the
	telescoping of the energy increments, and applying the preceding bound, we
	also obtain
	\[
	k\sum_{i=0}^{N-1}
	\norm{\dtt\bff m_{h,\varepsilon}^{i+1}}{h}^2
	\le
	C
	\left(
	\mathfrak a_{h,\lambda}
	(\bff m_{h,\varepsilon}^{0},\bff m_{h,\varepsilon}^{0})
	+
	k\sum_{i=0}^{N-1}h^{-2}\varepsilon_{i+1}^2
	\right).
	\]
	Finally, using the equivalence between
	$\mathfrak a_{h,\lambda}(\cdot,\cdot)$ and the $\bb H^1$-norm on $\bb V_h$,
	and then using \eqref{eq:solver-tolerance-stability-condition}, proves
	\eqref{eq:inexact-H1-stability}.
	
	If $\varepsilon_{i+1}=\varepsilon_{\rm fp}$, then $k\sum_{i=0}^{N-1}h^{-2}\varepsilon_{\rm fp}^2
	=
	Th^{-2}\varepsilon_{\rm fp}^2$.
	Thus, \eqref{eq:solver-tolerance-stability-condition} is satisfied whenever
	$\varepsilon_{\rm fp} = O(h)$. This completes the proof.
\end{proof}

With the stability estimate in hand, we are now ready to derive an \emph{a priori} error bound for the practical inexact scheme. As in the exact-solver case, we compare the numerical solution with the elliptic projection of the exact solution. Define
\begin{equation}\label{eq:def-inexact-error}
	\bff e_{h,\varepsilon}^i
	:=
	\bff m_{h,\varepsilon}^i-R_h\bff m^i,
	\qquad
	\overline{\bff e}_{h,\varepsilon}^{i+\frac12}
	:=
	\overline{\bff m}_{h,\varepsilon}^{i+\frac12}
	-
	R_h\overline{\bff m}^{i+\frac12}.
\end{equation}

\begin{theorem}[Energy-norm error estimate for the practical scheme]
	\label{thm:error-practical-constraint-preserving}
	Let Assumption~\ref{ass:exact-regularity} hold, and suppose that
	$k\le c_{\mathrm{FL}}h^2$, where $c_{\mathrm{FL}}$ is chosen as in
	Lemma~\ref{lem:constraint-preserving-contraction}. Let
	$\{\bff m_{h,\varepsilon}^i\}_{i=0}^N$ be generated by
	Algorithm~\ref{alg:constraint-preserving-fixed-point}. Then, for $k>0$
	sufficiently small,
	\begin{align}\label{eq:error-practical-energy}
		&\max_{0\le i\le N}
		\norm{\bff e_{h,\varepsilon}^i}{\bb H^1}^2
		+
		k\sum_{i=0}^{N-1}
		\norm{\dtt\bff e_{h,\varepsilon}^{i+1}}{h}^2
		\le
		C
		\left(
		\norm{\bff e_{h,\varepsilon}^0}{\bb H^1}^2
		+
		(h+k^2)^2
		+
		k\sum_{i=0}^{N-1}h^{-2}\varepsilon_{i+1}^2
		\right).
	\end{align}
	Consequently,
	\begin{align}
		\max_{0\le i\le N}
		\norm{\bff m_{h,\varepsilon}^i-\bff m^i}{\bb H^1}
		\le
		C
		\left( h+k^2
		+
		\left(
		k\sum_{i=0}^{N-1}h^{-2}\varepsilon_{i+1}^2
		\right)^{\frac12}
		\right).
		\label{eq:error-practical-total}
	\end{align}
	In particular, if
	\[
	k\sum_{i=0}^{N-1}h^{-2}\varepsilon_{i+1}^2
	\le
	C(h+k^2)^2,
	\]
	then the practical midpoint scheme preserves the same error rate as the ideal scheme; see~\eqref{eq:error-estimate-total}. For a uniform tolerance
	$\varepsilon_{i+1}=\varepsilon_{\rm fp}$, it is sufficient to take $\varepsilon_{\rm fp}\le Ch(h+k^2)$.
\end{theorem}

\begin{proof}
	Subtracting the projected exact equation from
	\eqref{eq:practical-perturbed-scheme}, analogously to \eqref{eq:error-equation}, we obtain the perturbed discrete error equation
	\begin{align}\label{eq:error-equation-practical}
		&\inpro{\dtt\bff e_{h,\varepsilon}^{i+1}}{\bff\phi_h}_h
		-
		\alpha
		\inpro{
			\overline{\bff m}_{h,\varepsilon}^{i+\frac12}
			\times
			\dtt\bff e_{h,\varepsilon}^{i+1}
		}{
			\bff\phi_h
		}_h
		+
		\gamma
		\inpro{
			\overline{\bff m}_{h,\varepsilon}^{i+\frac12}
			\times
			\mathcal H_{h,\lambda}
			\overline{\bff e}_{h,\varepsilon}^{i+\frac12}
		}{
			\bff\phi_h
		}_h
		\notag
		\\
		&=
		\alpha
		\inpro{
			\overline{\bff e}_{h,\varepsilon}^{i+\frac12}
			\times
			\dtt R_h\bff m^{i+1}
		}{
			\bff\phi_h
		}_h
		-
		\gamma
		\inpro{
			\overline{\bff e}_{h,\varepsilon}^{i+\frac12}
			\times
			P_h\mathcal H_\lambda
			(\overline{\bff m}^{i+\frac12})
		}{
			\bff\phi_h
		}_h
		\notag
		\\
		&\quad
		-
		\inpro{\bff\rho_h^{i+1}}{\bff\phi_h}_h
		+
		\gamma
		\inpro{
			\bff s_{h,\varepsilon}^{i+1}
		}{
			\bff\phi_h
		}_h.
	\end{align}
	Testing \eqref{eq:error-equation-practical} first with
	$\bff\phi_h=\dtt\bff e_{h,\varepsilon}^{i+1}$ and then with
	$\bff\phi_h
	=
	-\frac{\gamma}{\alpha}
	\mathcal H_{h,\lambda}
	\overline{\bff e}_{h,\varepsilon}^{i+\frac12}$,
	and adding the two identities, gives the same cancellation as in the exact
	solver case (Theorem~\ref{thm:energy-error-estimate}), with two additional terms. More precisely,
	\begin{align}
		&\norm{\dtt\bff e_{h,\varepsilon}^{i+1}}{h}^2
		+
		\frac{\gamma}{2\alpha k}
		\left[
		\mathfrak a_{h,\lambda}
		(\bff e_{h,\varepsilon}^{i+1},\bff e_{h,\varepsilon}^{i+1})
		-
		\mathfrak a_{h,\lambda}
		(\bff e_{h,\varepsilon}^i,\bff e_{h,\varepsilon}^i)
		\right]
		\notag
		\\
		&\le
		C\norm{\overline{\bff e}_{h,\varepsilon}^{i+\frac12}}{\bb H^1}^2
		+
		C\norm{\bff\rho_h^{i+1}}{\bb H^1}^2
		+
		\varepsilon_0
		\norm{\dtt\bff e_{h,\varepsilon}^{i+1}}{h}^2
		\notag
		\\
		&\quad
		+
		\gamma
		\inpro{
			\bff s_{h,\varepsilon}^{i+1}
		}{
			\dtt\bff e_{h,\varepsilon}^{i+1}
		}_h
		-
		\frac{\gamma^2}{\alpha}
		\inpro{
			\bff s_{h,\varepsilon}^{i+1}
		}{
			\mathcal H_{h,\lambda}
			\overline{\bff e}_{h,\varepsilon}^{i+\frac12}
		}_h ,
		\label{eq:error-practical-combined}
	\end{align}
	where $\varepsilon_0>0$ can be chosen arbitrarily small. It remains to estimate the last two terms on the right-hand side of \eqref{eq:error-practical-combined}. The first residual term is estimated by
	\[
	\left|
	\gamma
	\inpro{
		\bff s_{h,\varepsilon}^{i+1}
	}{
		\dtt\bff e_{h,\varepsilon}^{i+1}
	}_h
	\right|
	\le
	C\varepsilon_{i+1}
	\norm{\dtt\bff e_{h,\varepsilon}^{i+1}}{h}
	\le
	\varepsilon_0
	\norm{\dtt\bff e_{h,\varepsilon}^{i+1}}{h}^2
	+
	C\varepsilon_{i+1}^2.
	\]
	For the second solver term, we use the identity \eqref{eq:def-discrete-shifted-H} to obtain
	\begin{align*}
		\inpro{
			\bff s_{h,\varepsilon}^{i+1}
		}{
			\mathcal H_{h,\lambda}
			\overline{\bff e}_{h,\varepsilon}^{i+\frac12}
		}_h
		&=
		-
		\mathfrak a_{h,\lambda}
		\left(
		\overline{\bff e}_{h,\varepsilon}^{i+\frac12},
		\bff s_{h,\varepsilon}^{i+1}
		\right).
	\end{align*}
	Therefore, by the continuity of $\mathfrak a_{h,\lambda}$, inverse
	estimate \eqref{eq:inverse-estimate}, the norm equivalence \eqref{eq:lumped-L2-equivalence}, and the bound \eqref{eq:s-epsilon-bound}, we have
	\begin{align*}
		\left|
		\inpro{
			\bff s_{h,\varepsilon}^{i+1}
		}{
			\mathcal H_{h,\lambda}
			\overline{\bff e}_{h,\varepsilon}^{i+\frac12}
		}_h
		\right|
		&\le
		C
		\norm{\overline{\bff e}_{h,\varepsilon}^{i+\frac12}}{\bb H^1}
		\norm{\bff s_{h,\varepsilon}^{i+1}}{\bb H^1}
		\\
		&\le
		Ch^{-1}
		\norm{\overline{\bff e}_{h,\varepsilon}^{i+\frac12}}{\bb H^1}
		\varepsilon_{i+1}
		\\
		&\le
		C\norm{\overline{\bff e}_{h,\varepsilon}^{i+\frac12}}{\bb H^1}^2
		+
		Ch^{-2}\varepsilon_{i+1}^2,
	\end{align*}
	where we also used Young's inequality.
	Choosing $\varepsilon_0>0$ sufficiently small and using \eqref{eq:rho-estimate}, we get
	\begin{align*}
		&\frac{k}{2}
		\norm{\dtt\bff e_{h,\varepsilon}^{i+1}}{h}^2
		+
		\frac{\gamma}{2\alpha}
		\left[
		\mathfrak a_{h,\lambda}
		(\bff e_{h,\varepsilon}^{i+1},\bff e_{h,\varepsilon}^{i+1})
		-
		\mathfrak a_{h,\lambda}
		(\bff e_{h,\varepsilon}^i,\bff e_{h,\varepsilon}^i)
		\right]
		\notag
		\\
		&\le
		Ck
		\norm{\overline{\bff e}_{h,\varepsilon}^{i+\frac12}}{\bb H^1}^2
		+
		Ck(h+k^2)^2
		+
		Ckh^{-2}\varepsilon_{i+1}^2.
	\end{align*}
	The remainder of the proof is now identical to the exact-solver case in Theorem~\ref{thm:energy-error-estimate}, leading to
	\eqref{eq:error-practical-energy}. Finally,
	\eqref{eq:error-practical-total} follows from the splitting
	\[
	\bff m_{h,\varepsilon}^i-\bff m^i
	=
	\bff e_{h,\varepsilon}^i
	+
	(R_h\bff m^i-\bff m^i),
	\]
	as well as Lemma~\ref{lem:DMI-Ritz-H1} and an argument analogous to \eqref{eq:eh0-est}. The remaining statements of the theorem follows immediately as a consequence of \eqref{eq:error-practical-total}. This completes the proof.
\end{proof}

\section{Numerical experiments}\label{sec:numerics}

In this section, we present numerical experiments illustrating the convergence behaviour and structure-preserving properties of the proposed inexact midpoint scheme. All computations are performed on~\textsc{FEniCS}~\cite{AlnaesEtal15} using continuous piecewise affine finite elements. The nonlinear systems are solved by the constraint-preserving fixed-point iteration in Algorithm~\ref{alg:constraint-preserving-fixed-point}.
The detail of each experiment is specified below.

\subsection{Spatial and temporal convergence rates}
\label{subsec:numerics-rates}

We first verify the convergence rates predicted by the \emph{a priori} error estimates.
Since no closed-form solution is available for the problem with the chiral
boundary condition, the errors are computed against refined numerical reference
solutions. Let $\bff m_{h,k,\varepsilon}^N$ denote the numerical solution at
final time $T=t_N$, and let $\bff m_{\rm ref}$ be the corresponding reference
solution. For $s=0,1$, we define
\[
e_s(h,k)
:=
\norm{
	I_{h_{\rm ref}}\bff m_{h,k,\varepsilon}^N
	-
	\bff m_{\rm ref}
}{\bb H^s}.
\]
Here $\bb H^0=\bb L^2$, and $I_{h_{\rm ref}}$ denotes interpolation onto the
reference finite element space.

For two consecutive refinement levels with parameters
$\rho_j$ and $\rho_{j+1}$, where $\rho=h$ in the spatial convergence test and
$\rho=k$ in the temporal convergence test, the observed convergence rate in the
$\bb H^s$ norm is computed by
\[
\operatorname{rate}_s
:=
\frac{
	\log\!\left(e_s(\rho_j)/e_s(\rho_{j+1})\right)
}{
	\log\!\left(\rho_j/\rho_{j+1}\right)
}.
\]
In this uniform refinement test, $\rho_{j+1}=\rho_j/2$, so this reduces to
\[
\operatorname{rate}_s
=
\log_2\!\left(
\frac{e_s(\rho_j)}{e_s(\rho_{j+1})}
\right).
\]

We now describe the first experiment. Let $\Omega=(0,1)^2$. The parameters are chosen as
\[
\alpha=0.5,\qquad
\gamma=1,\qquad
\kex=1,\qquad
\kdm=0.2,\qquad
\lambda=0,
\]
and the final time is $T=2\times10^{-3}$. The fixed-point tolerance is set to $\varepsilon_i=10^{-10}$ for all time-steps. The initial datum is
\[
\bff m^0(x,y)
=
\frac{
	\bigl(
	a\sin(\pi x)\sin(\pi y),
	a\cos(\pi x)\sin(\pi y),
	1
	\bigr)
}{
	\left|
	\bigl(
	a\sin(\pi x)\sin(\pi y),
	a\cos(\pi x)\sin(\pi y),
	1
	\bigr)
	\right|
},
\qquad \text{where }\, a=0.35.
\]
This gives a smooth, unit-length, non-equilibrium initial configuration.

For the spatial convergence test, we use a coupled refinement strategy: the mesh
is refined uniformly and the time step is chosen according to $k=0.05h^2$.
This choice is compatible with the contraction condition for the fixed-point
iteration; see Lemma~\ref{lem:constraint-preserving-contraction}. It also
ensures that the temporal contribution in the estimate is of higher order than
the leading spatial contribution. Indeed, the error estimate predicts
$e_1(h,k)\lesssim h+k^2$, and hence, under the choice $k=0.05h^2$, the temporal
term satisfies $k^2=O(h^4)$. Thus the observed $\bb H^1$ rate is expected to
reflect the spatial convergence rate. The reference solution is computed on a
fine $200\times200$ triangulation using the same time-step rule. The errors
$e_0$ and $e_1$ are shown in Figure~\ref{fig:spatial-convergence}. The observed
$\bb H^1$ rate is close to first order, while the $\bb L^2$ error decays at a
higher rate.

For the temporal convergence test, the spatial mesh is fixed to a
$64\times64$ triangulation, and the time step is successively halved over four
refinement levels. Let $k_{\min}$ denote the smallest, that is, the finest,
time step among these four tested values. The reference solution is computed on
the same spatial mesh with $k_{\rm ref}= k_{\min}/6$.
Since the spatial mesh is fixed, the comparison isolates the temporal discretisation
error. The corresponding errors are shown in
Figure~\ref{fig:temporal-convergence}, and the observed rates are close to
second order.

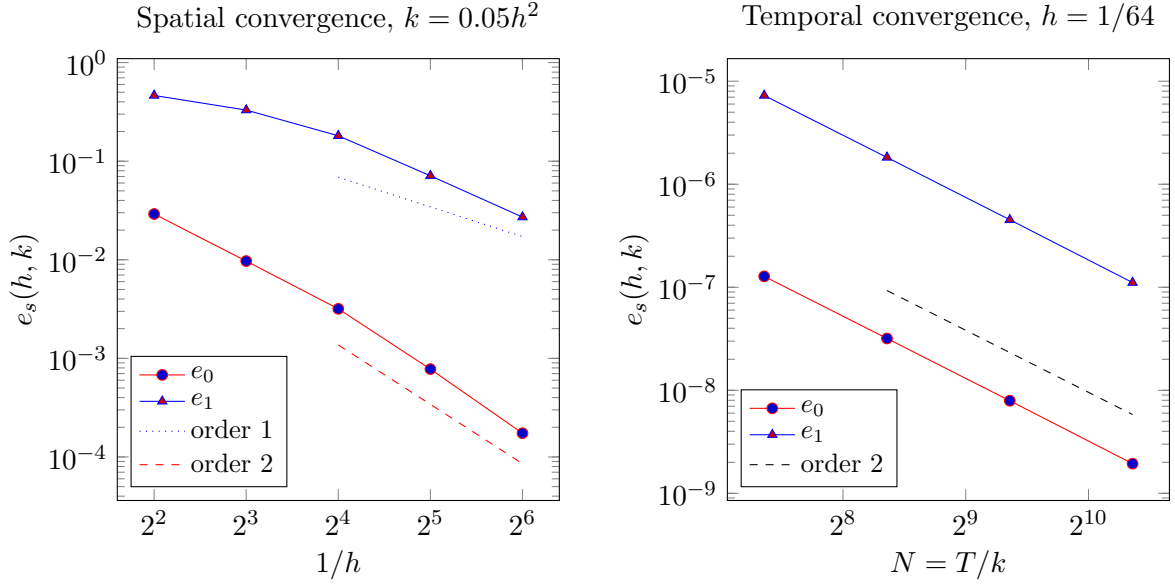
\begin{figure}[!htb]
	\centering
	\begin{subfigure}[b]{0.45\textwidth}
		\centering
		\begin{tikzpicture}
			\begin{axis}[
				title={Spatial convergence, $k=0.05h^2$},
				height=1\textwidth,
				width=1\textwidth,
				xlabel={$1/h$},
				ylabel={$e_s(h,k)$},
				xmode=log,
				ymode=log,
				log basis x={2},
				log basis y={10},
				legend pos=south west,
				legend cell align=left,
				]
				
				\addplot+[mark=*,red] coordinates {
					(4,  2.912e-2)
					(8,  9.700e-3)
					(16, 3.180e-3)
					(32, 7.780e-4)
					(64, 1.740e-4)
				};
				
				\addplot+[mark=triangle*,blue] coordinates {
					(4,  4.640e-1)
					(8,  3.300e-1)
					(16, 1.810e-1)
					(32, 7.110e-2)
					(64, 2.710e-2)
				};
				
				\addplot+[dotted,no marks,blue,domain=16:64] {1.1/x};
				\addplot+[dashed,no marks,red,domain=16:64] {0.35/x^2};
				
				\legend{
					\small{$e_0$},
					\small{$e_1$},
					\small{order $1$},
					\small{order $2$}
				}
			\end{axis}
		\end{tikzpicture}
		\caption{Spatial convergence under the coupled refinement $k=0.05h^2$.}
		\label{fig:spatial-convergence}
	\end{subfigure}
	\hspace{1em}
	\begin{subfigure}[b]{0.45\textwidth}
		\centering
		\begin{tikzpicture}
			\begin{axis}[
				title={Temporal convergence, $h=1/64$},
				height=1\textwidth,
				width=1\textwidth,
				xlabel={$N=T/k$},
				ylabel={$e_s(h,k)$},
				xmode=log,
				ymode=log,
				log basis x={2},
				log basis y={10},
				legend pos=south west,
				legend cell align=left,
				]
				
				\addplot+[mark=*,red] coordinates {
					(164, 1.276e-7)
					(328, 3.184e-8)
					(656, 7.92e-9)
					(1312, 1.94e-9)
				};
				
				\addplot+[mark=triangle*,blue] coordinates {
					(164, 7.27e-6)
					(328, 1.82e-6)
					(656, 4.52e-7)
					(1312, 1.11e-7)
				};
				
				\addplot+[dashed,no marks,black,domain=328:1312] {0.01/x^2};
				
				\legend{
					\small{$e_0$},
					\small{$e_1$},
					\small{order $2$}
				}
			\end{axis}
		\end{tikzpicture}
		\caption{Temporal convergence on a fixed $64\times64$ mesh.}
		\label{fig:temporal-convergence}
	\end{subfigure}
	\caption{Spatial and temporal convergence of the inexact midpoint scheme (Algorithm~\ref{alg:constraint-preserving-fixed-point}).}
	\label{fig:convergence-rates}
\end{figure}

\subsection{Influence of the fixed-point tolerance}
\label{subsec:numerics-fixed-point-tolerance}

We next examine the influence of the stopping tolerance in the constraint-preserving fixed-point iteration. This experiment is designed to illustrate the solver-error contribution in Theorem~\ref{thm:error-practical-constraint-preserving}. For a uniform tolerance $\varepsilon_{i+1}=\varepsilon_{\rm fp}$, the solver-error term in \eqref{eq:error-practical-total} becomes
\[
\left(
k\sum_{i=0}^{N-1}h^{-2}\varepsilon_{\rm fp}^2
\right)^{\frac12}
=
T^{\frac12}\frac{\varepsilon_{\rm fp}}{h}.
\]
Thus, for fixed $h$ and $k$, the theorem predicts that the additional error due to the inexact fixed-point solve is controlled by $\varepsilon_{\rm fp}/h$.

To test this behaviour, we fix the same parameters as before, fix the spatial mesh and the time step, and vary only the fixed-point tolerance. More precisely, we use a $64\times64$ triangulation and choose $k=0.05h^2$.
A reference solution $\bff m_{h,k,\varepsilon_{\rm ref}}^N$ is computed using a much smaller tolerance $\varepsilon_{\rm ref}=10^{-13}$. For several values of $\varepsilon_{\rm fp}>\varepsilon_{\rm ref}$, we compute
\[
\delta_s(\varepsilon_{\rm fp})
:=
\norm{
	\bff m_{h,k,\varepsilon_{\rm fp}}^N
	-
	\bff m_{h,k,\varepsilon_{\rm ref}}^N
}{\bb H^s},
\qquad s=0,1.
\]
Since the mesh and time step are fixed throughout this test, the quantities $\delta_s(\varepsilon_{\rm fp})$ measure only the effect of stopping the fixed-point iteration at different tolerances.

Figure~\ref{fig:fixed-point-tolerance} plots $\delta_0$ and $\delta_1$ against $\varepsilon_{\rm fp}/h$. The errors decrease as the tolerance is tightened, confirming that the practical inexact scheme approaches the fully converged midpoint solution. The decay is not perfectly linear at every point, which is expected since the stopping criterion changes the number of fixed-point iterations only discretely. Nevertheless, the observed behaviour is consistent with the solver-error contribution in Theorem~\ref{thm:error-practical-constraint-preserving}.

\begin{figure}[!htb]
	\centering
	\begin{tikzpicture}
		\begin{axis}[
			title={Influence of the fixed-point tolerance},
			height=0.50\textwidth,
			width=0.65\textwidth,
			xlabel={$\varepsilon_{\rm fp}/h$},
			ylabel={$\delta_s(\varepsilon_{\rm fp})$},
			xmode=log,
			ymode=log,
			legend pos=north west,
			legend cell align=left,
			]
			
			\addplot+[mark=*,red] coordinates {
				(6.400000e-08,4.735463e-14)
				(1.920000e-07,6.710274e-14)
				(6.400000e-07,8.759755e-13)
				(1.920000e-06,1.546222e-12)
				(6.400000e-06,1.916104e-12)
				(1.920000e-05,3.331702e-11)
				(6.400000e-05,6.180138e-11)
				(1.920000e-04,8.171541e-11)
				(6.400000e-04,1.523987e-09)
				(1.920000e-03,2.881410e-09)
				(6.400000e-03,3.917785e-09)
			};
			
			\addplot+[mark=triangle*,blue] coordinates {
				(6.400000e-08,5.283138e-12)
				(1.920000e-07,7.484853e-12)
				(6.400000e-07,8.517410e-11)
				(1.920000e-06,1.506198e-10)
				(6.400000e-06,1.867658e-10)
				(1.920000e-05,2.798550e-09)
				(6.400000e-05,5.201055e-09)
				(1.920000e-04,6.880745e-09)
				(6.400000e-04,1.079179e-07)
				(1.920000e-03,2.044520e-07)
				(6.400000e-03,2.781494e-07)
			};
			
			\addplot+[dashed,no marks,black,domain=6.4e-8:6.4e-3,samples=2] {1.0e-5*x};
			
			\legend{
				\small{$\delta_0$},
				\small{$\delta_1$},
				\small{order $1$}
			}
		\end{axis}
	\end{tikzpicture}
	\caption{Influence of the fixed-point tolerance on the final-time error for fixed $h=1/64$ and $k=0.05h^2$.}
	\label{fig:fixed-point-tolerance}
\end{figure}
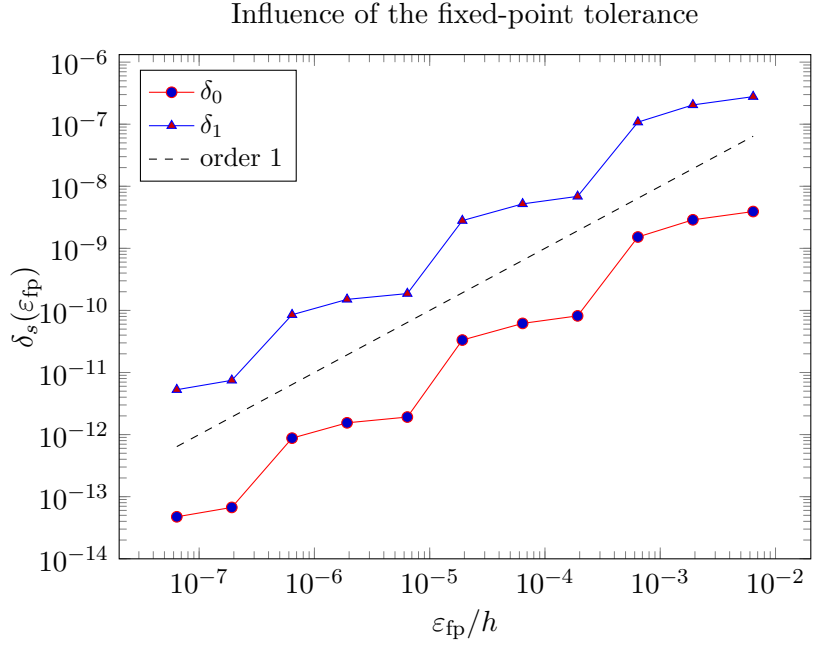

\subsection{Skyrmion relaxation and structure preservation}
\label{subsec:numerics-skyrmion}

We next test the robustness of the proposed scheme on a physically relevant chiral configuration. The purpose of this experiment is to illustrate the
behaviour of the inexact midpoint scheme when applied to a localised
skyrmion-type texture, and to verify the preservation of the geometric and
energetic structures during its dissipative relaxation.

Let $\Omega=(0,1)^2$. We consider a Bloch-type skyrmion initial state. Set
\[
r
:=
\sqrt{\left(x-\frac12\right)^2+\left(y-\frac12\right)^2} 
\quad \text{and} \quad 
\theta(r)
:=
\frac{\pi}{2}
\left(
1-\tanh\left(\frac{r-R}{w}\right)
\right),
\]
with $R>0$ and $w>0$ to be specified later.
For $r>0$, the initial magnetisation is given by
\[
\bff m^0(x,y)
=
\left(
-\left[\frac{y-\frac12}{r}\right] \sin\theta(r),\,
\left[\frac{x-\frac12}{r}\right] \sin\theta(r),\,
\cos\theta(r)
\right),
\]
while at the centre of the skyrmion we set
\[
\bff m^0\left(\frac12,\frac12\right)
=
(0,0,-1).
\]
In the experiment, we take $R=0.18$ and $w=0.04$.
This initial condition describes a localised Bloch-type skyrmion texture, namely the magnetisation points approximately downward near the centre of the domain, upward away from the core, and rotates tangentially in the transition layer between the two regions.

The parameters are chosen as
\[
\alpha=0.5,
\qquad
\gamma=1,
\qquad
\kex=1,
\qquad
\kdm=2\pi,
\qquad
\lambda=0.
\]
We use a $64\times64$ triangulation and choose $k=0.02h^2$ and $T=0.2$.
The fixed-point tolerance is set to $\varepsilon_i=10^{-10}$ for all time-steps.
To assess the structure-preserving behaviour of the method, we monitor the nodal length defect
\begin{equation}\label{eq:nodal-length-defect}
\mathrm{d}(t_i)
:=
\max_{z\in\mathcal N_h}
\left|
|\bff m_h^i(z)|-1
\right|,
\end{equation}
and the discrete energy
\begin{equation}\label{eq:discrete-energy}
\mathcal{E}(\bff m_h^i)
=
\frac{\kex}{2}
\norm{\nabla \bff m_h^i}{\bb L^2}^2
+
\kdm
\inpro{\bff m_h^i}{\nabla\times\bff m_h^i}.
\end{equation}

The results are shown in Figures~\ref{fig:skyrmion-snapshots} and~\ref{fig:skyrmion-diagnostics}. The numerical solution preserves the nodal
unit-length constraint up to round-off error throughout the simulation. The
discrete energy decreases over time, confirming the dissipative behaviour of
the scheme in this chiral setting. The snapshots show that the localised
skyrmion core persists during the relaxation: the central region remains
oriented approximately downward while the surrounding magnetisation remains
oriented upward. This persistence is consistent with the robustness usually
associated with skyrmion-type textures, whose localised reversed core and
chirality are important features in spintronic applications.

\begin{figure}[!htb]
	\centering
	\begin{subfigure}[b]{0.29\textwidth}
		\centering
		\includegraphics[width=\textwidth]{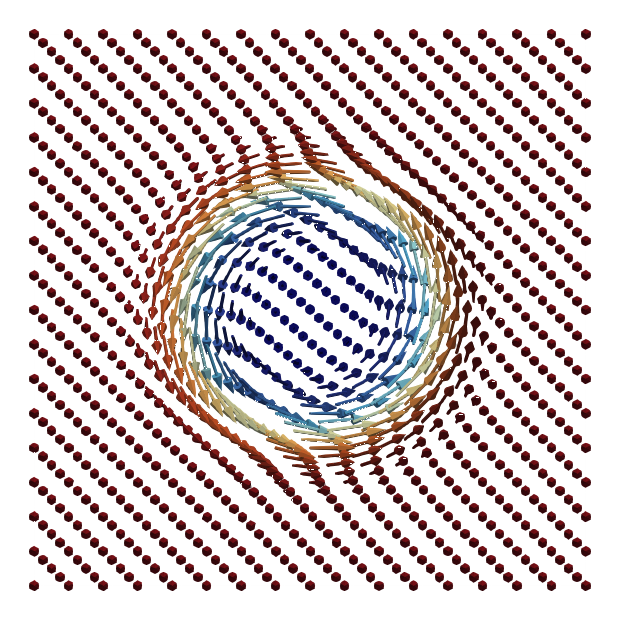}
		\caption{$t=0$}
	\end{subfigure}
	\begin{subfigure}[b]{0.29\textwidth}
		\centering
		\includegraphics[width=\textwidth]{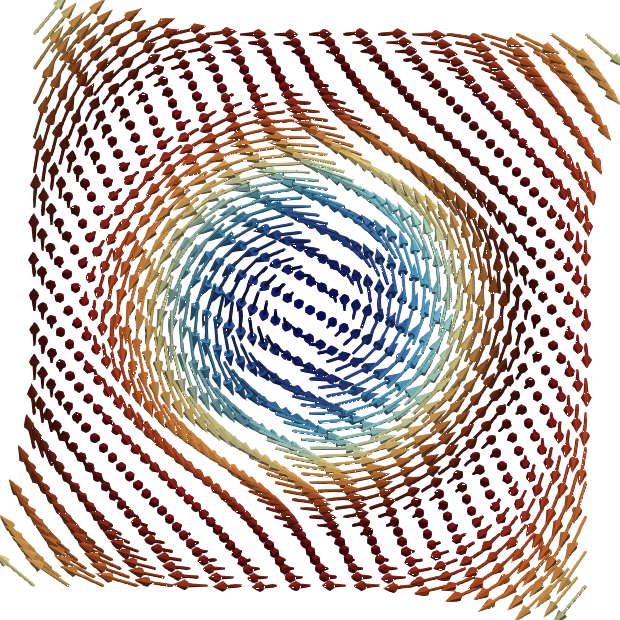}
		\caption{$t=0.01$}
	\end{subfigure}
	\begin{subfigure}[b]{0.29\textwidth}
		\centering
		\includegraphics[width=\textwidth]{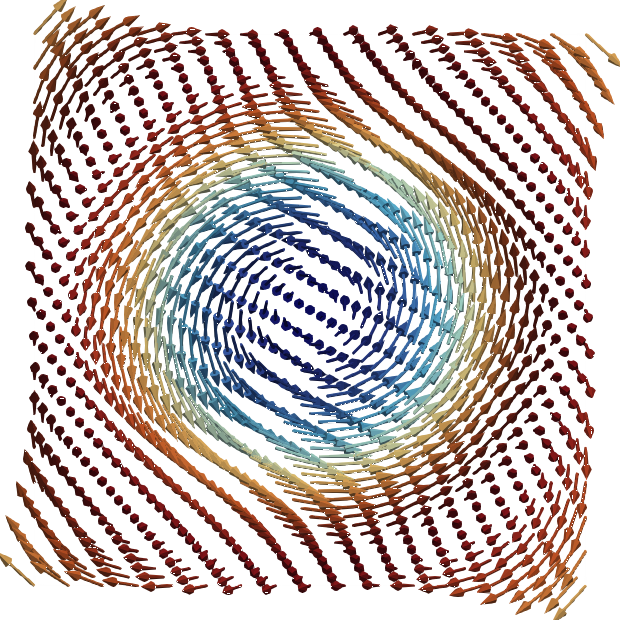}
		\caption{$t=0.03$}
	\end{subfigure}
	\begin{subfigure}[b]{0.09\textwidth}
		\centering
		\includegraphics[width=\textwidth]{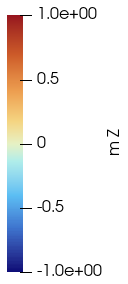}
	\end{subfigure}
	\caption{Evolution of the Bloch-type skyrmion initial state under the LLG equation with bulk DMI. The colour indicates the $z$-component of the magnetisation.}
	\label{fig:skyrmion-snapshots}
\end{figure}

\begin{figure}[!htb]
	\centering
	\includegraphics[width=0.48\textwidth]{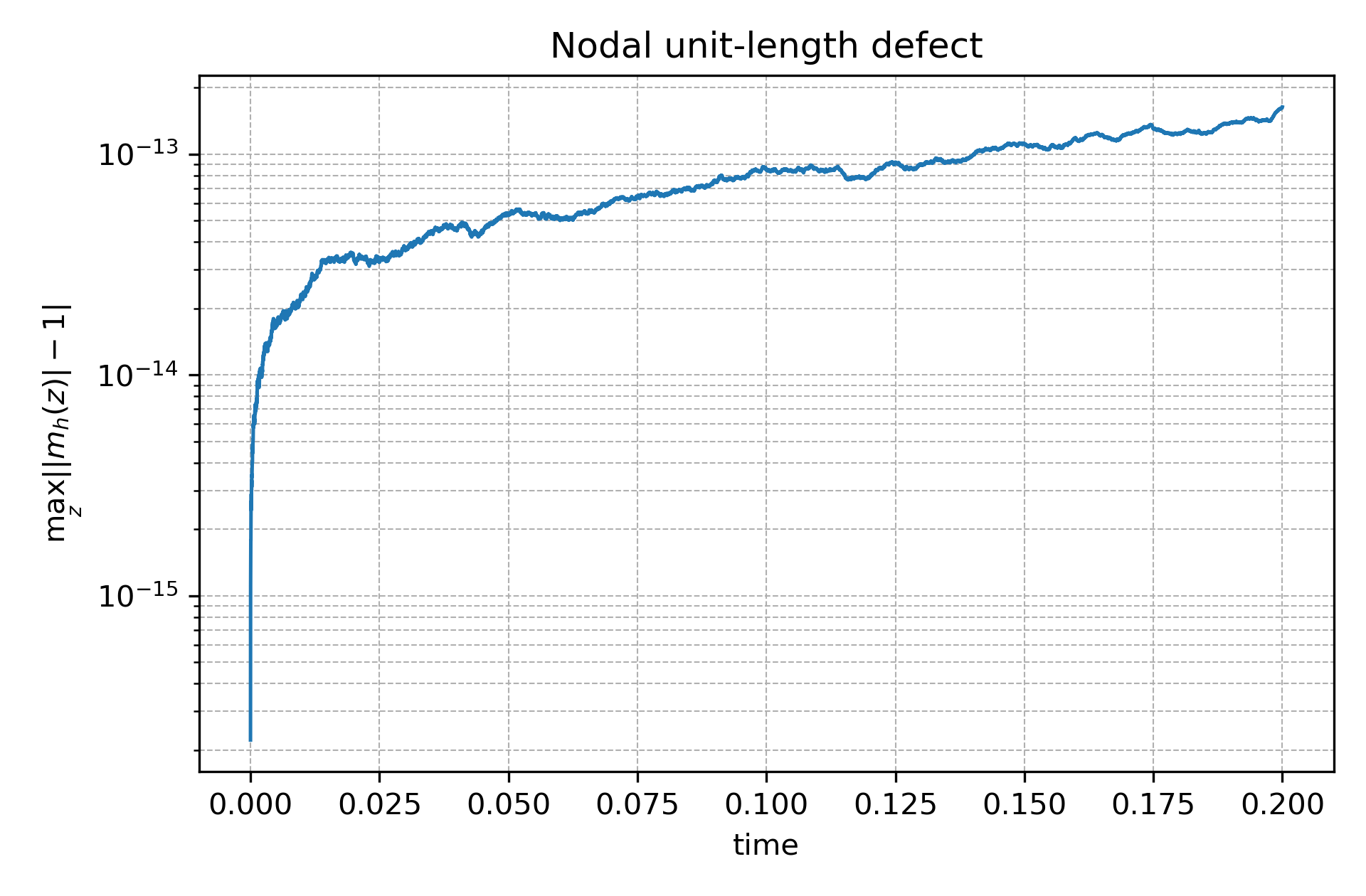}
	\includegraphics[width=0.48\textwidth]{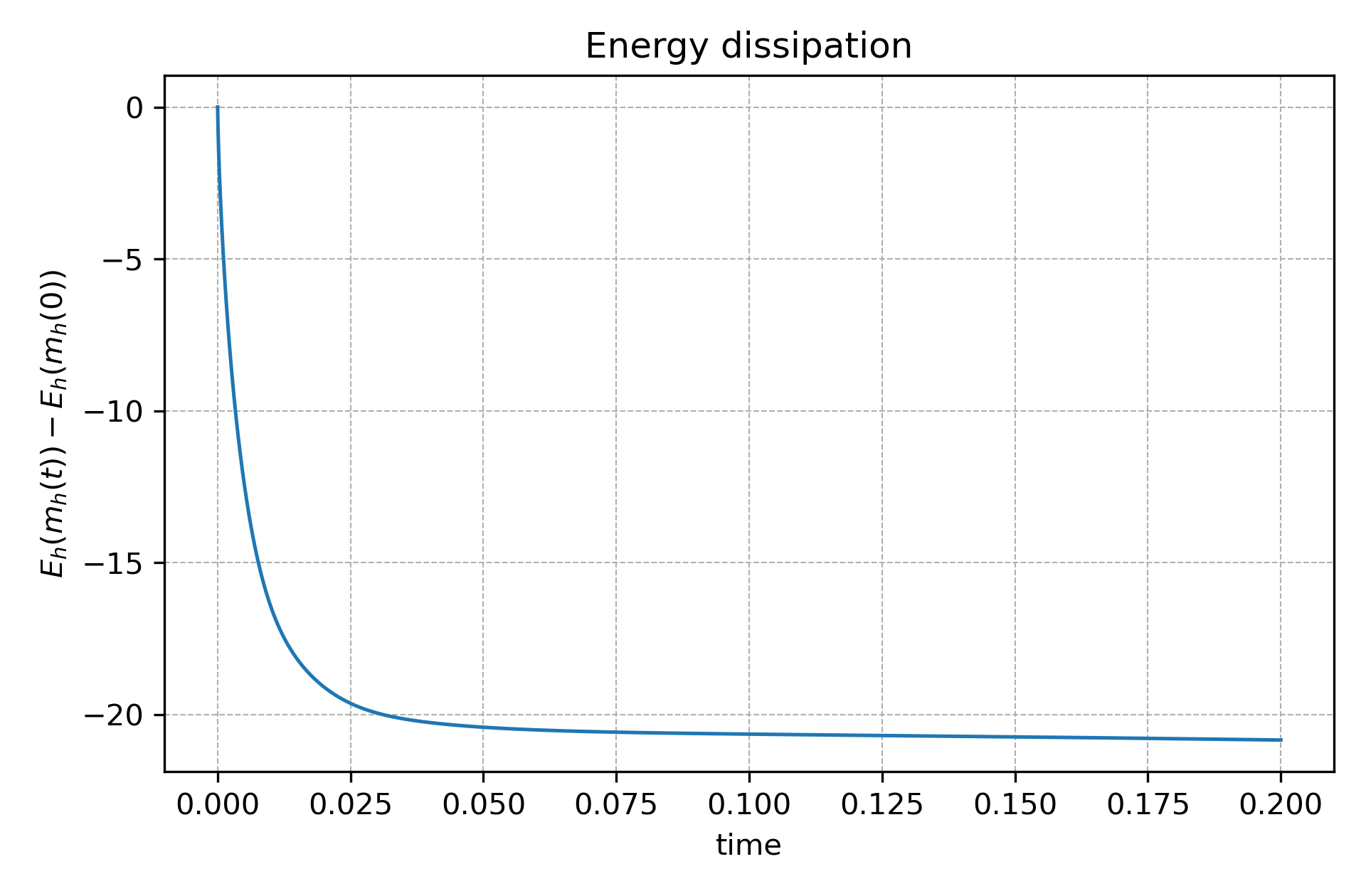}
	\caption{Nodal length defect $\mathrm{d}(t_i)$ and discrete energy change $\mathcal{E}(\bff{m}_h^i)-\mathcal{E}(\bff{m}_h^0)$ for the skyrmion relaxation test.}
	\label{fig:skyrmion-diagnostics}
\end{figure}

\subsection{Bubbling experiment with chiral boundary condition}
\label{subsec:numerics-bartels-prohl-chiral}

We finally test the proposed method on a high-gradient benchmark inspired by the bubbling experiment of Bartels and Prohl~\cite{BarPro06}. The purpose of this experiment is twofold. First, in the exchange-only case $\kdm=0$, it provides a reference test in which the maximum gradient is known numerically to develop a rapid growth. Second, for $\kdm\ne0$, it allows us to examine how the bulk DMI term and the corresponding chiral natural boundary condition affect the concentration and possible unwinding of the bubble.

Here, we take $\Omega=\left(-\frac12,\frac12\right)^2$. 
To describe the initial data, let $\bff x=(x_1,x_2)$, $r=|\bff x|$, and $A(r)=\frac1s(1-2r)^4$.
We then set
\[
\bff m^0(\bff x)
=
\begin{cases}
	(0,0,-1),
	& r\ge \frac12,\\[0.5em]
	\displaystyle
	\frac{
		\bigl(2x_1A(r),\,2x_2A(r),\,A(r)^2-r^2\bigr)
	}{
		A(r)^2+r^2
	},
	& r<\frac12 .
\end{cases}
\]
The discrete initial condition is given by $I_h \bff m^0$.
The parameters of the problem are chosen as
\[
\alpha=1.0,
\qquad
\gamma=2.0,
\qquad
\kex=1,
\qquad
\lambda=0,
\qquad
s=4.
\]
 We compare the five DMI strengths
\[
\kdm=0,\quad \kdm=\pm1,\quad \kdm=\pm2 .
\]
The parameter $\kdm=0$ corresponds to an effective field consisting of only the exchange interaction considered in~\cite{BarPro06}.
In this experiment, a $32\times 32$ triangulation is fixed throughout. We choose $k=0.05 h^2$ and run the experiment up to final time $T=0.1$.
The fixed-point tolerance is set to $\varepsilon_i=10^{-10}$ for all time-steps.

We monitor six diagnostic quantities: the nodal length defect \eqref{eq:nodal-length-defect}, the discrete energy change $\mathcal{E}(\bff{m}_h^i)-\mathcal{E}(\bff{m}_h^0)$, the exchange energy $\mathcal{E}_{\rm{exc}}(\bff{m}_h^i)$, the DMI energy $\mathcal{E}_{\rm{DMI}}(\bff{m}_h^i)$, the maximum gradient seminorm $G_\infty^i$, and the number of fixed-point iterations at each time-step. Here, $\mathcal{E}_{\rm{exc}}$ and $\mathcal{E}_{\rm{DMI}}$ were defined in \eqref{eq:energy}, and
\[
G_\infty^i
:=
\max_{K\in\mathcal T_h}
\norm{\nabla\bff m_h^i}{L^\infty(K)} .
\]

Figures~\ref{fig:bp-chiral-snapshots-minus} and~\ref{fig:bp-chiral-snapshots-plus}  show snapshots of the magnetisation for $\kdm=-2$ and $\kdm=+2$, respectively. When $\kdm=-2$, the initially localised bubble undergoes a rapid concentration near the origin, producing a strongly localised structure. In contrast, when $\kdm=+2$, the bubble appears to unwind smoothly. 

The corresponding diagnostic quantities are shown in Figure~\ref{fig:bp-chiral-diagnostics}. The nodal length defect remains close to machine precision throughout the computation, confirming that the fixed-point iteration preserves the nodal unit-length constraint. The discrete energy decreases in time as shown. The plot of $G_\infty^i$ exhibits the high-gradient behaviour, indicating a possible gradient blow-up when $\kdm\leq 0$. On the other hand, in the tested range, positive values of $\kdm$ appear to delay the gradient concentration or suppress it altogether. The fixed-point iteration count remains moderate, although it becomes more demanding near the time interval where the gradient becomes large.

A notable feature of the DMI comparison is the asymmetry with respect to the sign of $\kdm$. In the present experiment, negative values of $\kdm$ accelerate the growth of $G_\infty^i$, while positive values of $\kdm$ tend to suppress the concentration and promote relaxation. This behaviour appears to be consistent with the chiral nature of the DMI interaction. The initial bubble has a prescribed handedness, and changing the sign of $\kdm$ changes whether the DMI contribution favours or penalises this chirality. In particular, the chiral boundary condition may either assist the concentration of the bubble or facilitate its unwinding, depending on the sign of $\kdm$. A systematic investigation of this is beyond the scope of the present work, but appears to be an interesting topic for further study.

\begin{figure}[!htb]
	\centering
	\begin{subfigure}[b]{0.29\textwidth}
		\centering
		\includegraphics[width=\textwidth]{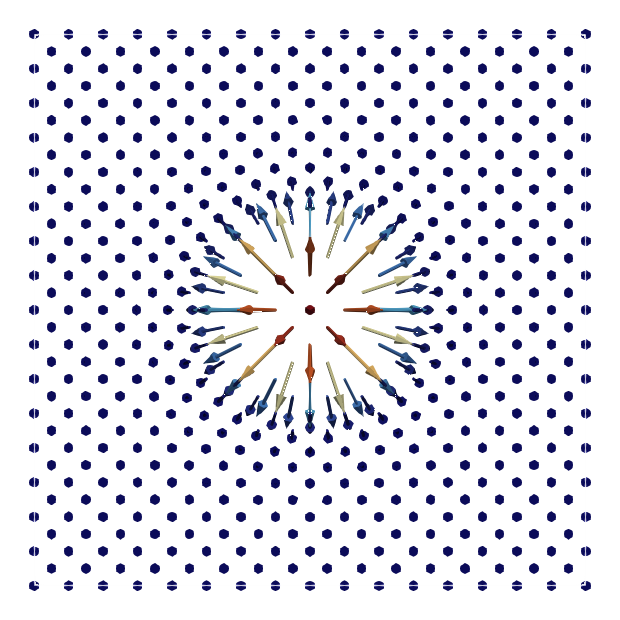}
		\caption{$t=0$}
	\end{subfigure}
	\begin{subfigure}[b]{0.29\textwidth}
		\centering
		\includegraphics[width=\textwidth]{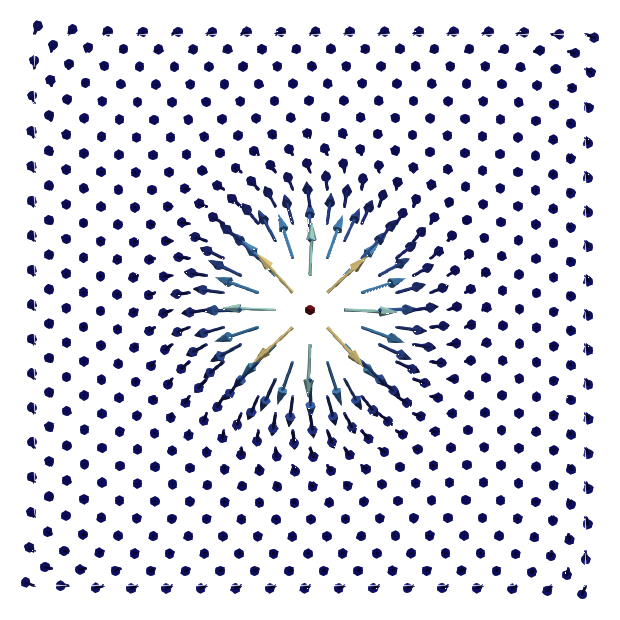}
		\caption{$t=0.005$}
	\end{subfigure}
	\begin{subfigure}[b]{0.29\textwidth}
		\centering
		\includegraphics[width=\textwidth]{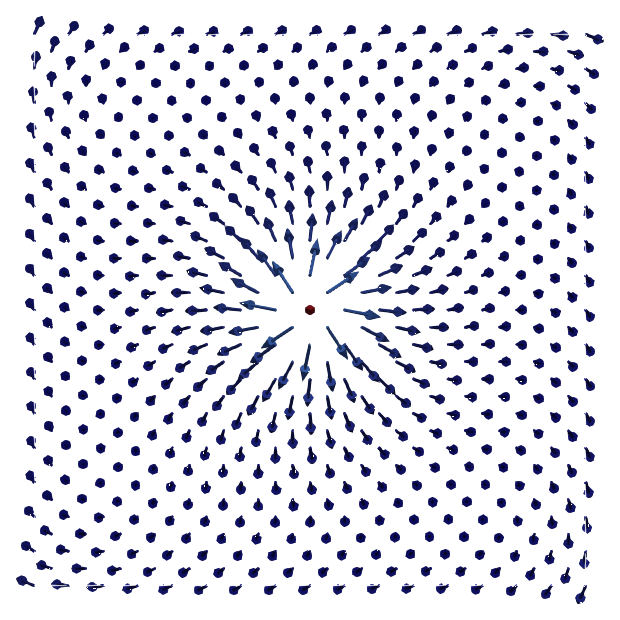}
		\caption{$t=0.01$}
	\end{subfigure}
	\begin{subfigure}[b]{0.09\textwidth}
		\centering
		\includegraphics[width=\textwidth]{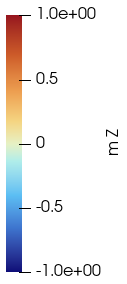}
	\end{subfigure}
	\begin{subfigure}[b]{0.29\textwidth}
		\centering
		\includegraphics[width=\textwidth]{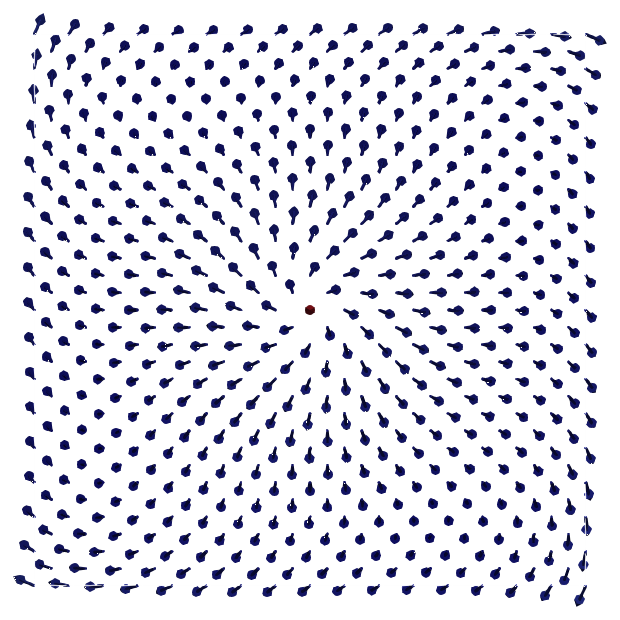}
		\caption{$t=0.015$}
	\end{subfigure}
	\begin{subfigure}[b]{0.29\textwidth}
		\centering
		\includegraphics[width=\textwidth]{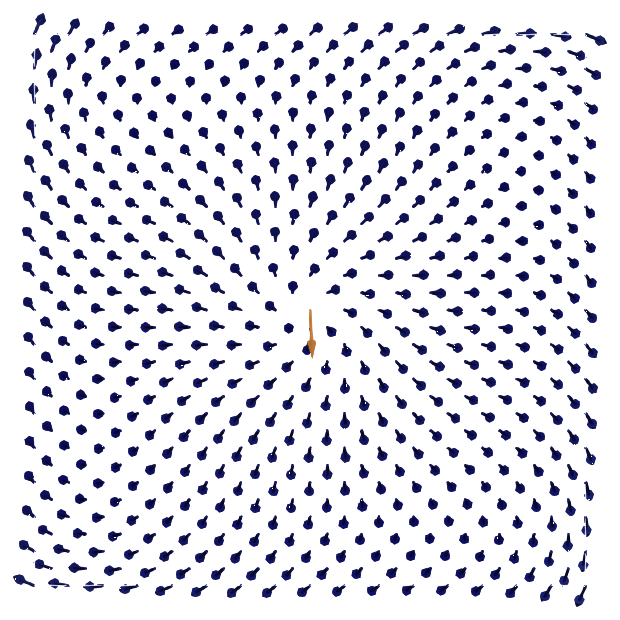}
		\caption{$t=0.016$}
	\end{subfigure}
	\begin{subfigure}[b]{0.29\textwidth}
		\centering
		\includegraphics[width=\textwidth]{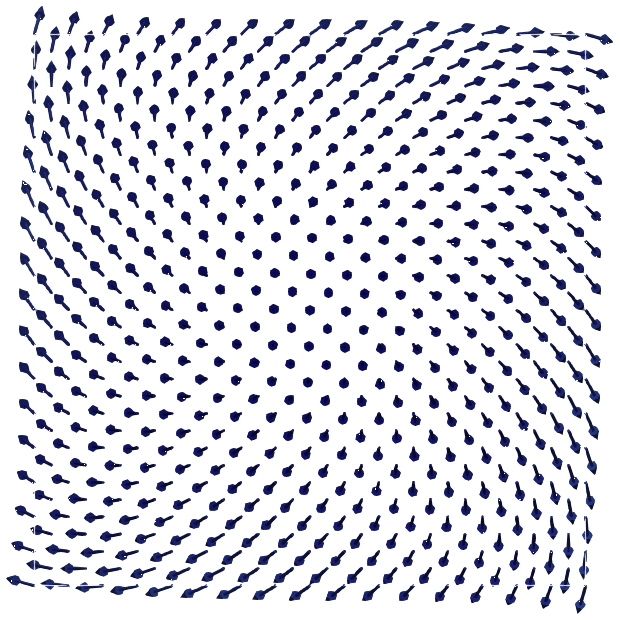}
		\caption{$t=0.03$}
	\end{subfigure}
	\begin{subfigure}[b]{0.09\textwidth}
		\centering
		\includegraphics[width=\textwidth]{blow_legend.png}
	\end{subfigure}
	\caption{Snapshots of the bubbling experiment with $\kdm=-2$. The snapshots illustrate the concentration of the bubble under the chiral boundary condition. The colour indicates the $z$-component of the magnetisation.}
	\label{fig:bp-chiral-snapshots-minus}
\end{figure}

\begin{figure}[!htb]
	\centering
	\begin{subfigure}[b]{0.29\textwidth}
		\centering
		\includegraphics[width=\textwidth]{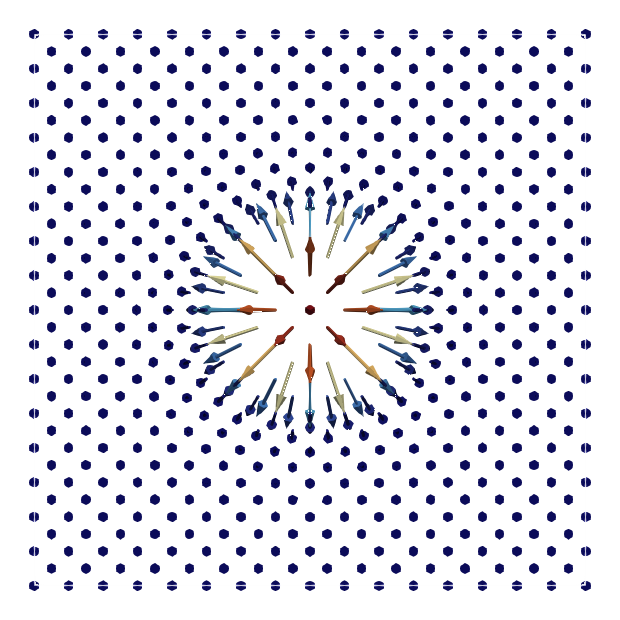}
		\caption{$t=0$}
	\end{subfigure}
	\begin{subfigure}[b]{0.29\textwidth}
		\centering
		\includegraphics[width=\textwidth]{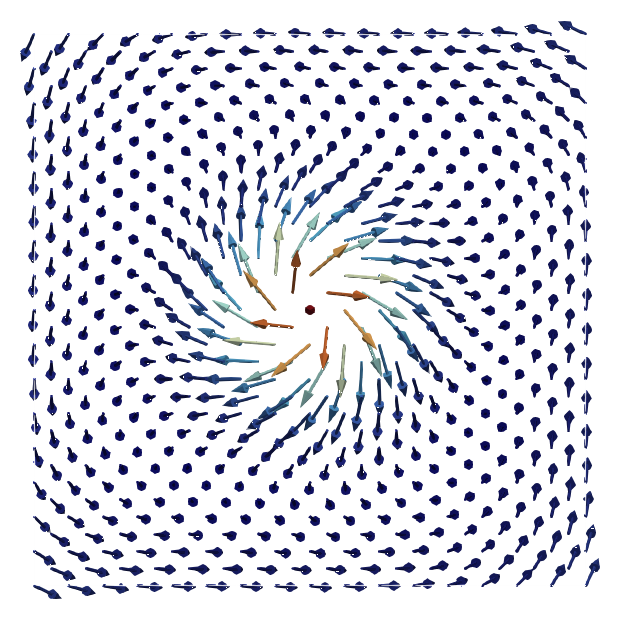}
		\caption{$t=0.02$}
	\end{subfigure}
	\begin{subfigure}[b]{0.29\textwidth}
		\centering
		\includegraphics[width=\textwidth]{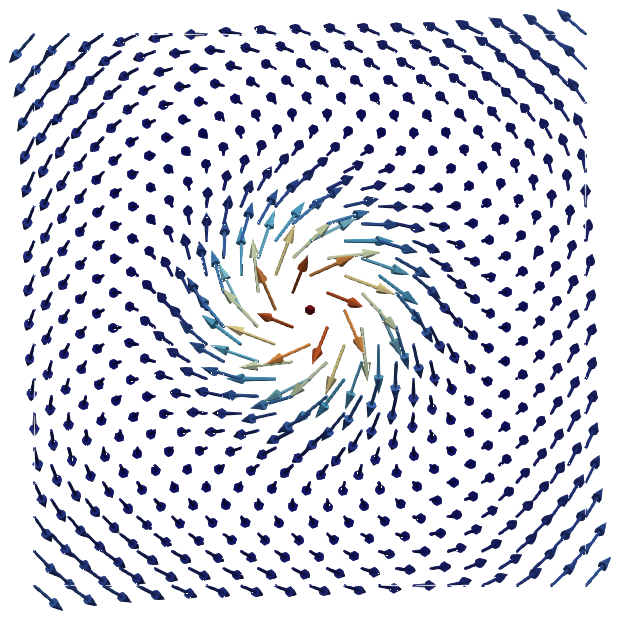}
		\caption{$t=0.04$}
	\end{subfigure}
	\begin{subfigure}[b]{0.09\textwidth}
		\centering
		\includegraphics[width=\textwidth]{blow_legend.png}
	\end{subfigure}
	\begin{subfigure}[b]{0.29\textwidth}
		\centering
		\includegraphics[width=\textwidth]{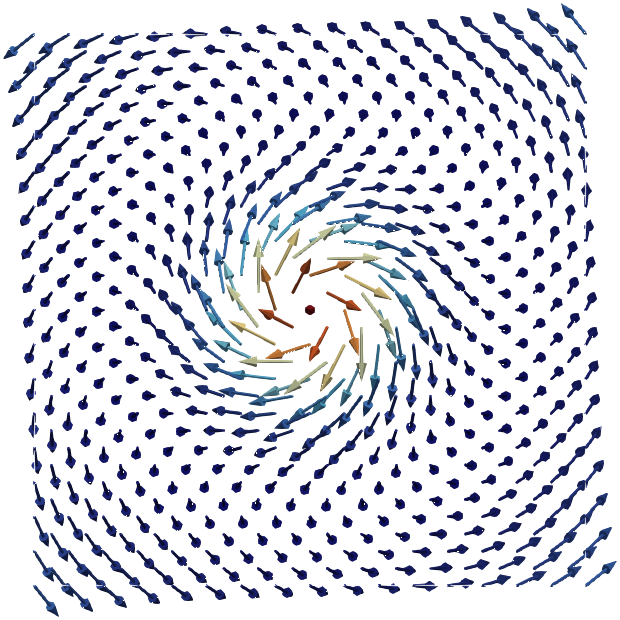}
		\caption{$t=0.06$}
	\end{subfigure}
	\begin{subfigure}[b]{0.29\textwidth}
		\centering
		\includegraphics[width=\textwidth]{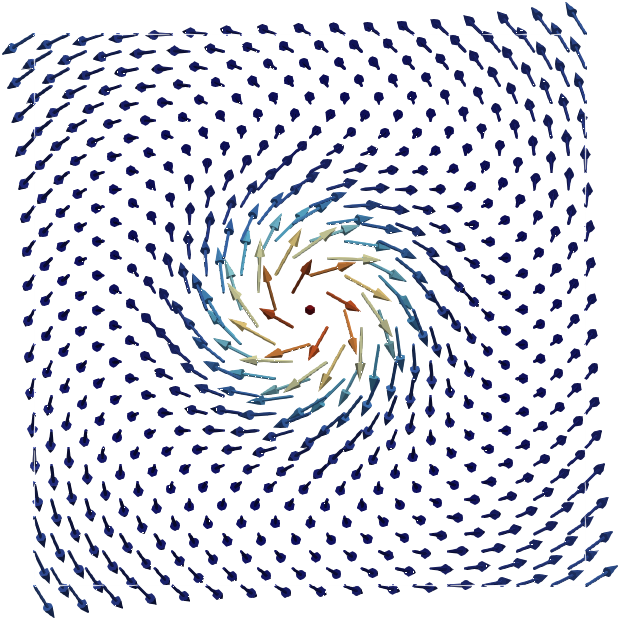}
		\caption{$t=0.08$}
	\end{subfigure}
	\begin{subfigure}[b]{0.29\textwidth}
		\centering
		\includegraphics[width=\textwidth]{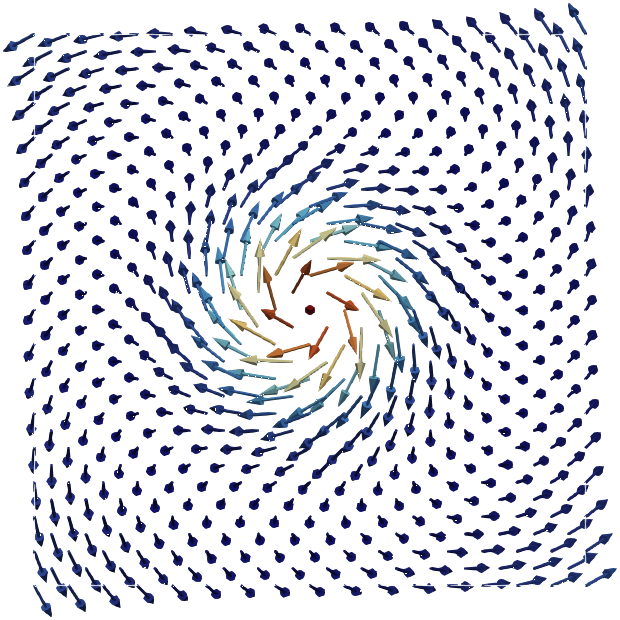}
		\caption{$t=0.1$}
	\end{subfigure}
	\begin{subfigure}[b]{0.09\textwidth}
		\centering
		\includegraphics[width=\textwidth]{blow_legend.png}
	\end{subfigure}
	\caption{Snapshots of the bubbling experiment with $\kdm=+2$. The snapshots illustrate the unwinding of the bubble under the chiral boundary condition. The colour indicates the $z$-component of the magnetisation.}
	\label{fig:bp-chiral-snapshots-plus}
\end{figure}

\begin{figure}[!htb]
	\centering
	\begin{subfigure}{0.48\textwidth}
		\centering
		\includegraphics[width=\textwidth]{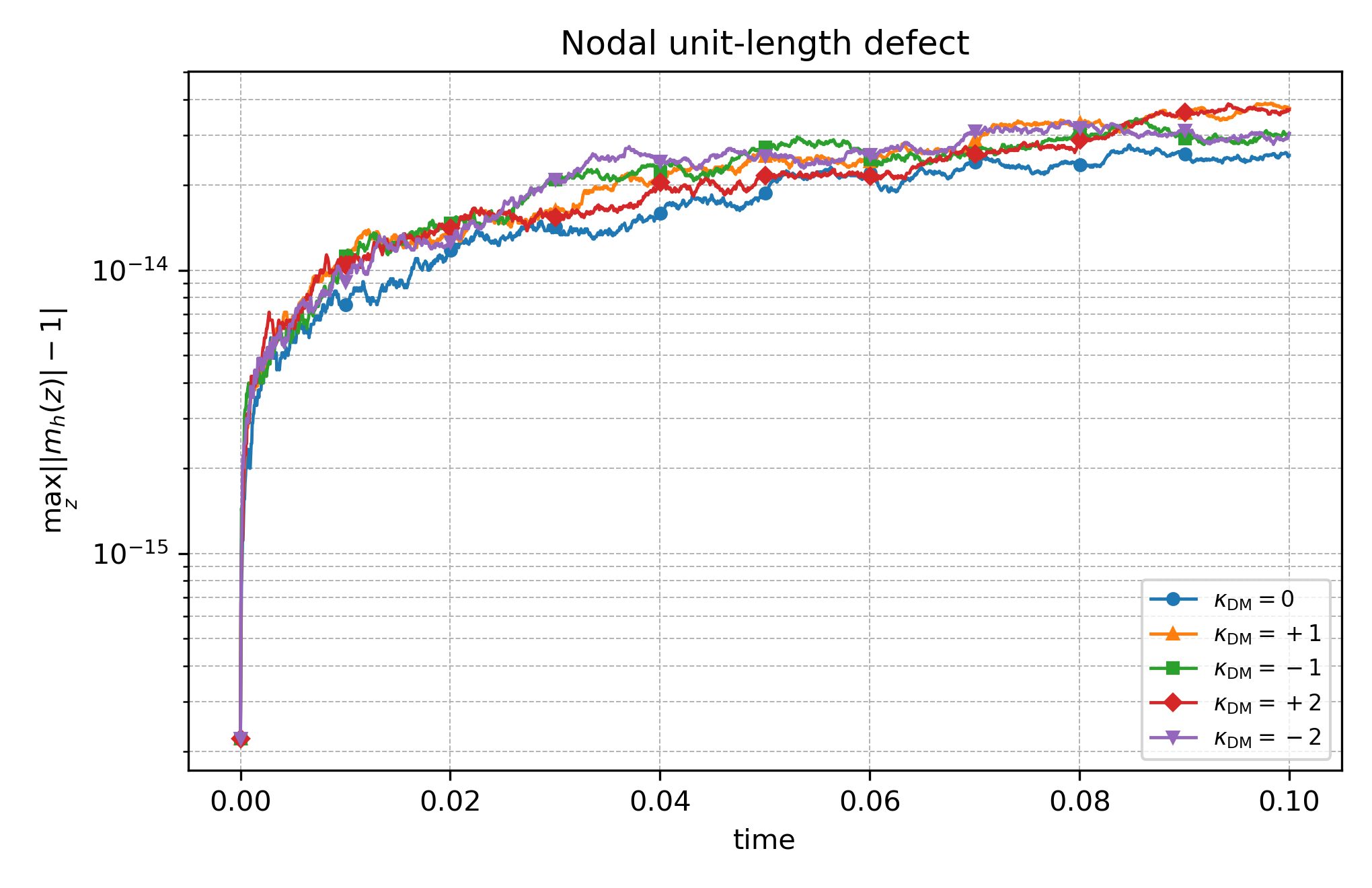}
		\caption{Nodal length defect $\mathrm{d}(t_i)$.}
	\end{subfigure}
	\begin{subfigure}{0.48\textwidth}
		\centering
		\includegraphics[width=\textwidth]{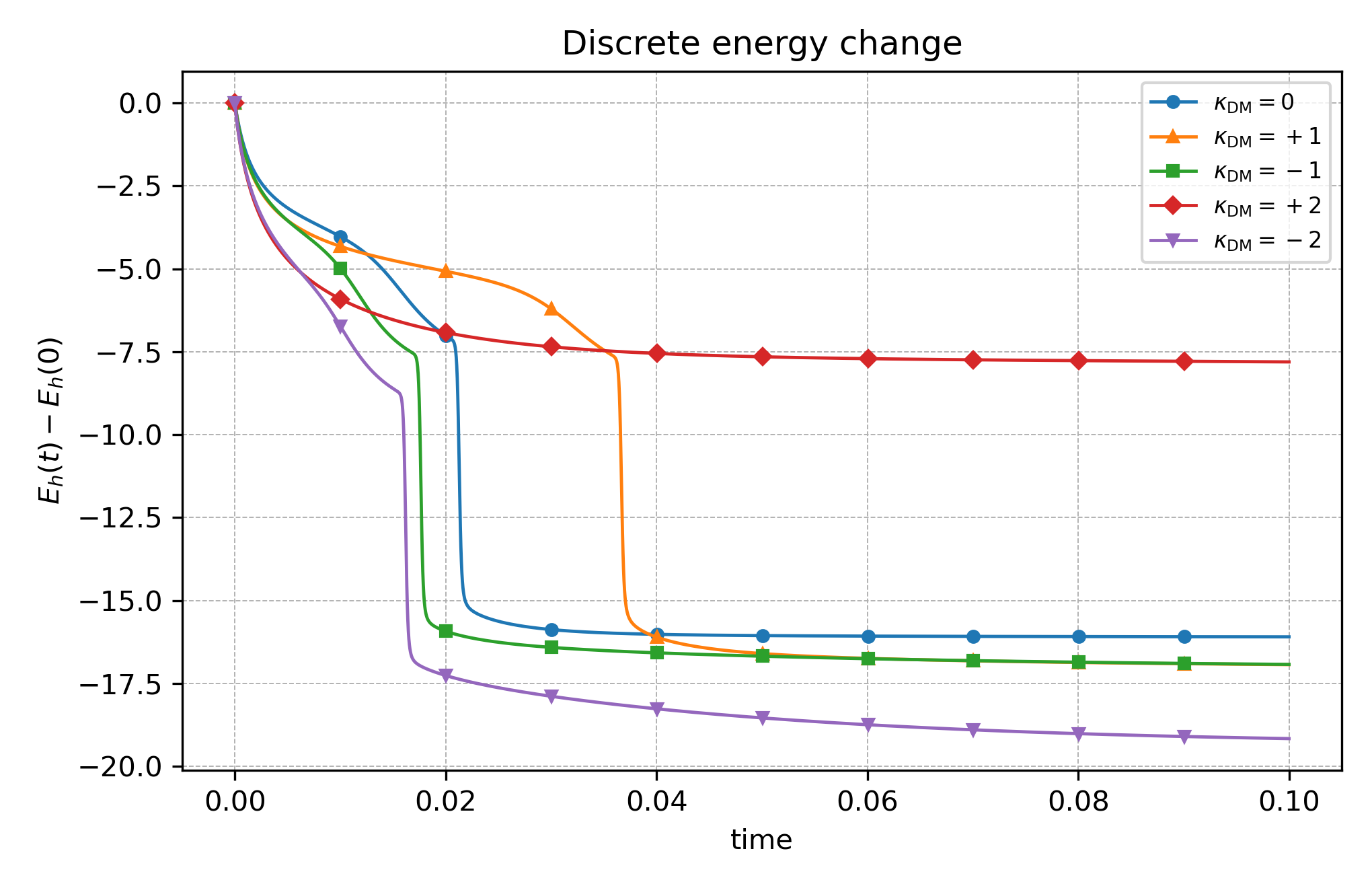}
		\caption{Discrete energy change $\mathcal{E}(\bff{m}_h^i)-\mathcal{E}(\bff{m}_h^0)$.}
	\end{subfigure}
	
\vspace{3ex}
	
	\begin{subfigure}{0.48\textwidth}
		\centering
		\includegraphics[width=\textwidth]{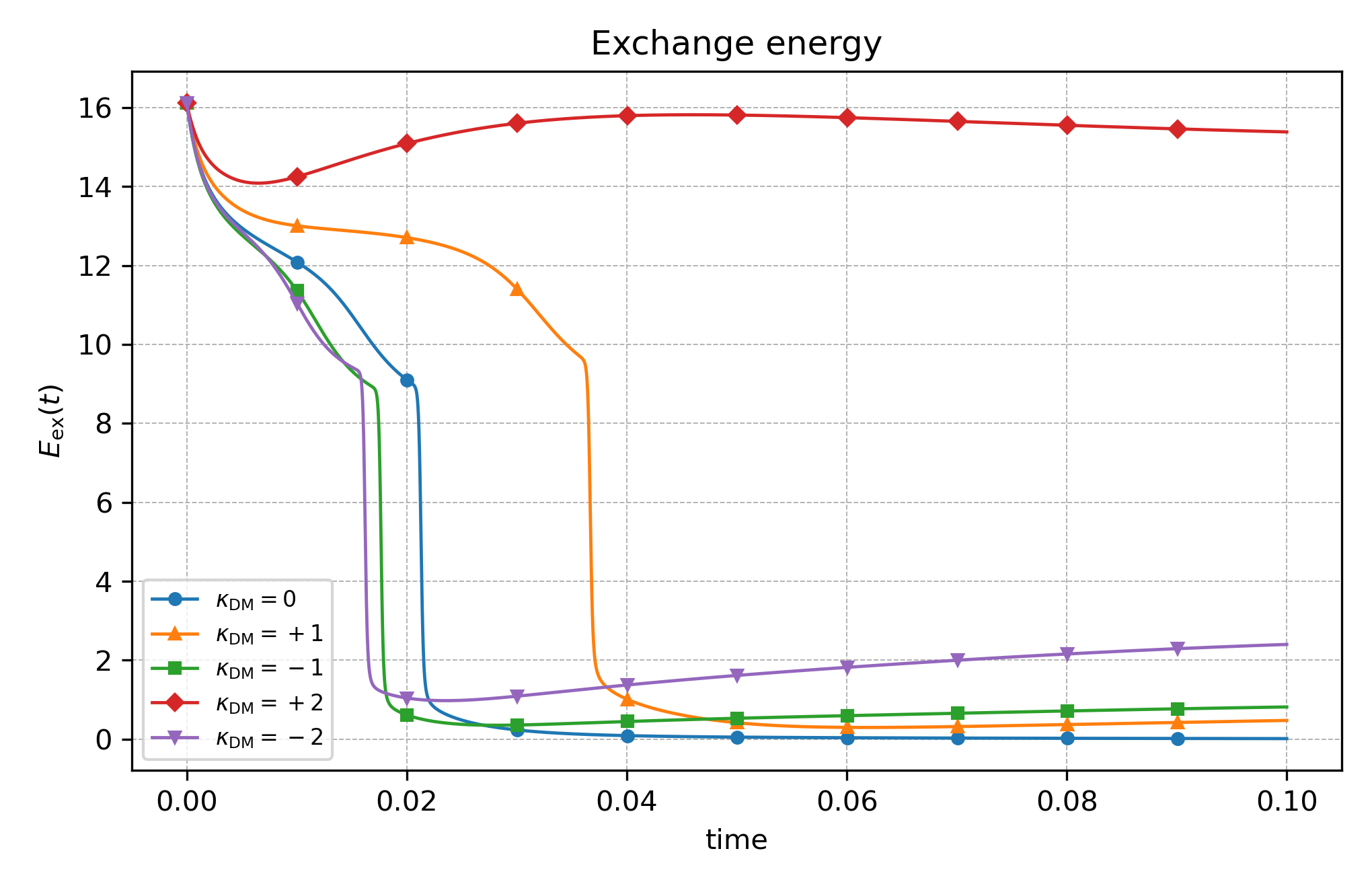}
		\caption{Exchange energy $\mathcal{E}_{\rm{exc}}(\bff{m}_h^i)$.}
	\end{subfigure}
	\begin{subfigure}{0.48\textwidth}
		\centering
		\includegraphics[width=\textwidth]{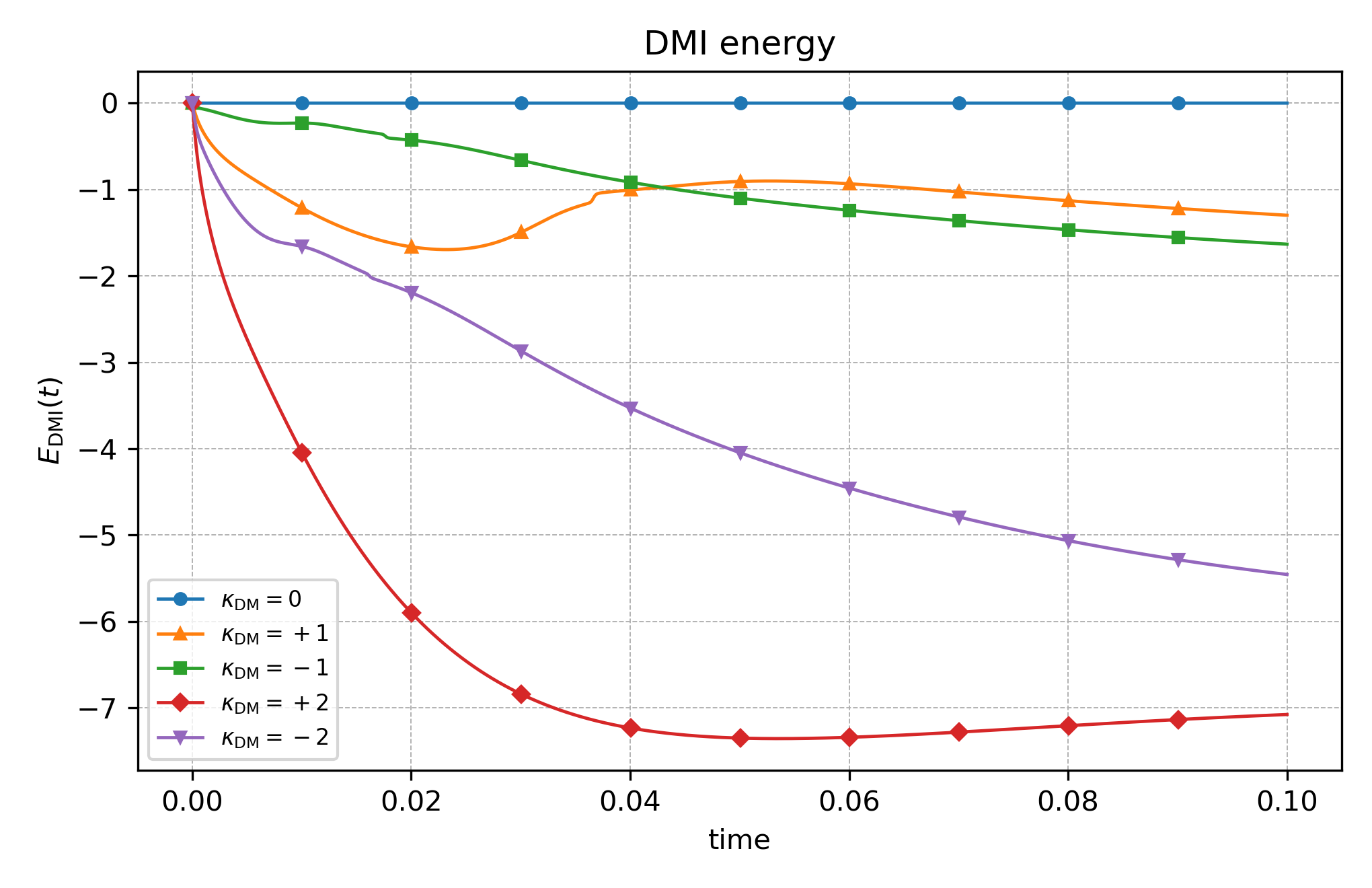}
		\caption{DMI energy $\mathcal{E}_{\rm{DMI}}(\bff{m}_h^i)$.}
	\end{subfigure}

\vspace{3ex}

\begin{subfigure}{0.48\textwidth}
	\centering
	\includegraphics[width=\textwidth]{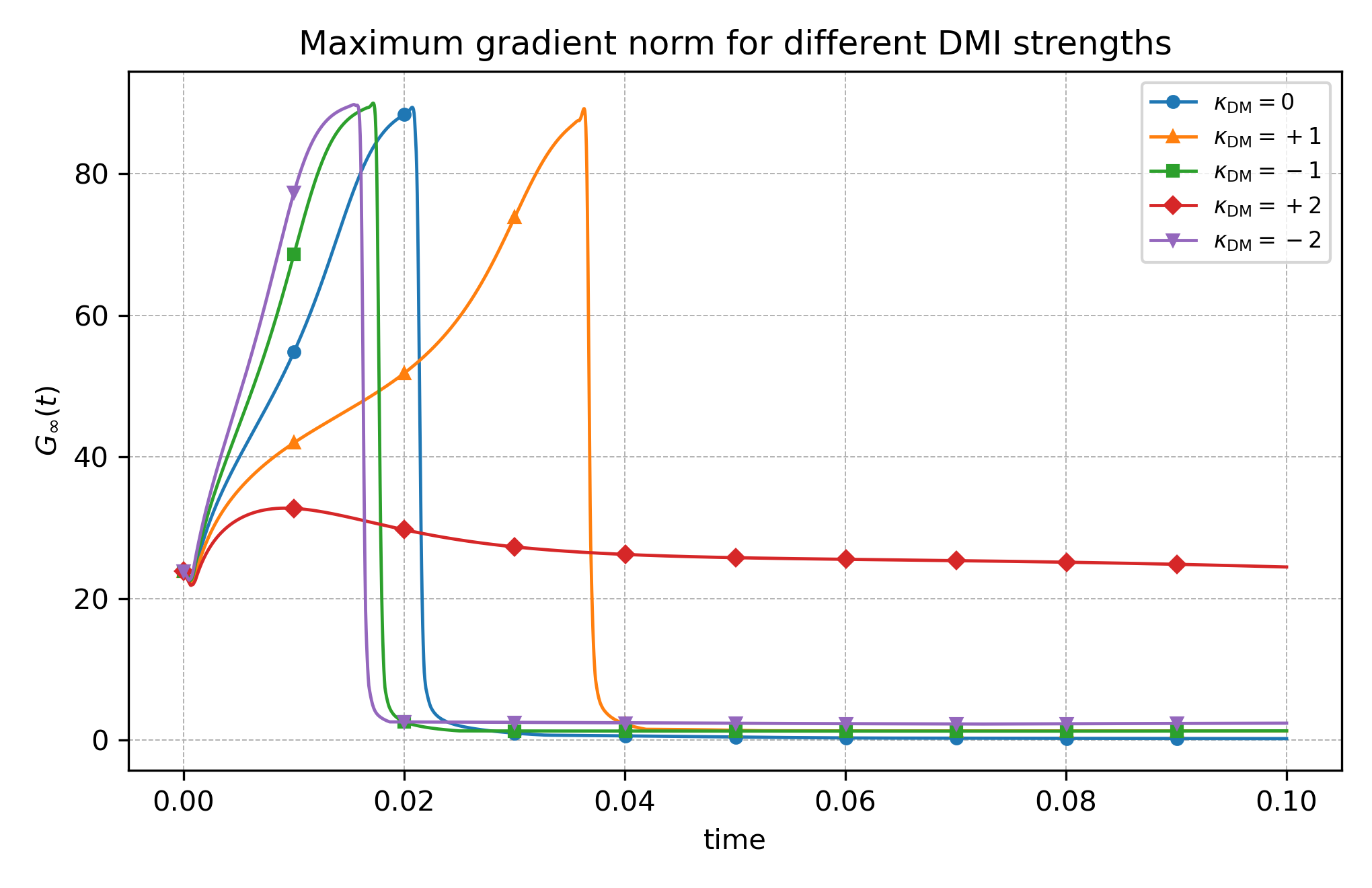}
	\caption{Maximum gradient norm $G_\infty^i$.}
\end{subfigure}
\begin{subfigure}{0.48\textwidth}
	\centering
	\includegraphics[width=\textwidth]{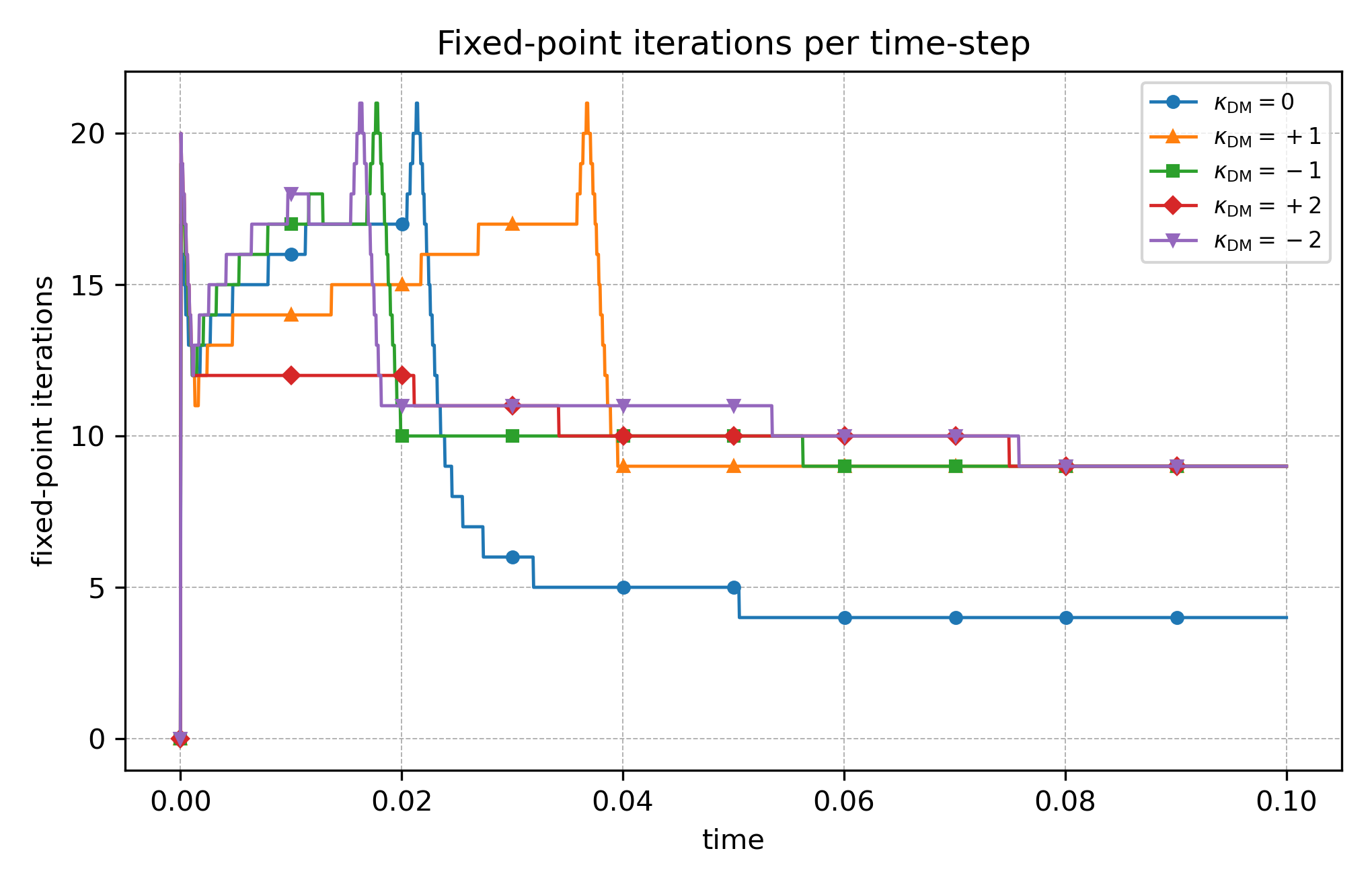}
	\caption{Number of fixed-point iterations.}
\end{subfigure}
	\caption{Diagnostic quantities for the bubbling experiment with $\kdm=0,\pm1,\pm2$. The legends use the value of the DMI coefficient $\kdm$.}
	\label{fig:bp-chiral-diagnostics}
\end{figure}

\section*{Statements and declarations}

\subsection*{Conflict of interest}
The author has no competing interests to declare that are relevant to the content of this article.

\subsection*{Funding statement}
This research is supported by the Commonwealth through an Australian Government Research Training Program Scholarship [DOI: \href{https://doi.org/10.82133/C42F-K220}{https://doi.org/10.82133/C42F-K220}].

\newcommand{\noopsort}[1]{}\def\cprime{$'$}
\def\soft#1{\leavevmode\setbox0=\hbox{h}\dimen7=\ht0\advance \dimen7
	by-1ex\relax\if t#1\relax\rlap{\raise.6\dimen7
		\hbox{\kern.3ex\char'47}}#1\relax\else\if T#1\relax
	\rlap{\raise.5\dimen7\hbox{\kern1.3ex\char'47}}#1\relax \else\if
	d#1\relax\rlap{\raise.5\dimen7\hbox{\kern.9ex \char'47}}#1\relax\else\if
	D#1\relax\rlap{\raise.5\dimen7 \hbox{\kern1.4ex\char'47}}#1\relax\else\if
	l#1\relax \rlap{\raise.5\dimen7\hbox{\kern.4ex\char'47}}#1\relax \else\if
	L#1\relax\rlap{\raise.5\dimen7\hbox{\kern.7ex
			\char'47}}#1\relax\else\message{accent \string\soft \space #1 not
		defined!}#1\relax\fi\fi\fi\fi\fi\fi}


\appendix

\section{Auxiliary estimates}\label{sec:appendix}

We prove several auxiliary estimates which are essential for the error analysis.

\begin{lemma}[Elliptic projection estimate]\label{lem:DMI-Ritz-H1}
	For every $\bff{v}\in\bb{H}^2$ satisfying \eqref{eq:chiral-boundary-condition}, the projection $R_h\bff{v}$ defined by \eqref{eq:def-DMI-Ritz-projection-field} satisfies
	\begin{equation}\label{eq:DMI-Ritz-H1-est}
		\norm{R_h\bff{v}-\bff{v}}{\bb{H}^1}
		\le
		Ch
		\left(
		\norm{\bff{v}}{\bb{H}^2}
		+
		\norm{\mathcal H_\lambda(\bff{v})}{\bb{L}^2}
		\right),
	\end{equation}
	where $\mathcal{H}_\lambda$ was defined in \eqref{eq:def-shifted-fields} and $C$ is a constant independent of $h$.
\end{lemma}

\begin{proof}
	Set $\bff{\eta}_h:=I_h\bff{v}-\bff{v}$ and
	$\bff{\theta}_h:=R_h\bff{v}-I_h\bff{v}$.
	Then
	\begin{equation}\label{eq:Rh v minus v}
		R_h\bff{v}-\bff{v}
		=
		\bff{\theta}_h+\bff{\eta}_h.
	\end{equation}
	Subtracting $\mathfrak a_{h,\lambda}(I_h\bff{v},\bff{\chi}_h)$ from \eqref{eq:def-DMI-Ritz-projection-weak}, and using \eqref{eq:def-shifted-form}, gives
	\begin{align}
		\mathfrak a_{h,\lambda}(\bff{\theta}_h,\bff{\chi}_h)
		&=
		-\mathfrak a(\bff{\eta}_h,\bff{\chi}_h)
		-
		\lambda\inpro{\bff{\eta}_h}{\bff{\chi}_h}
		+
		\lambda
		\big(
		\inpro{I_h\bff{v}}{\bff{\chi}_h}
		-
		\inpro{I_h\bff{v}}{\bff{\chi}_h}_h
		\big)
		-
		\delta_h(\bff{v};\bff{\chi}_h),
		\label{eq:theta-equation-concise}
	\end{align}
	where $\delta_h$ represents the quadrature defect term:
	\[
	\delta_h(\bff{v};\bff{\chi}_h)
	:=
	\inpro{P_h\mathcal H_\lambda(\bff{v})}{\bff{\chi}_h}_h
	-
	\inpro{P_h\mathcal H_\lambda(\bff{v})}{\bff{\chi}_h}.
	\]
	By the interpolation estimate, the continuity of $\mathfrak a$, the quadrature estimate \eqref{eq:lumped-quadrature-Hminus1}, and the $\bb L^2$-stability of $P_h$, we obtain from \eqref{eq:theta-equation-concise},
	\begin{align}
		\left|
		\mathfrak a_{h,\lambda}(\bff{\theta}_h,\bff{\chi}_h)
		\right|
		&\le
		C\norm{\bff{\eta}_h}{\bb{H}^1}\norm{\bff{\chi}_h}{\bb{H}^1}
		+
		C\norm{\bff{\eta}_h}{\bb{L}^2}\norm{\bff{\chi}_h}{\bb{L}^2}
		\notag
		\\
		&\quad
		+
		Ch\norm{I_h\bff{v}}{\bb{L}^2}\norm{\bff{\chi}_h}{\bb{H}^1}
		+
		Ch\norm{P_h\mathcal H_\lambda(\bff{v})}{\bb{L}^2}
		\norm{\bff{\chi}_h}{\bb{H}^1}
		\notag
		\\
		&\le
		Ch
		\left(
		\norm{\bff{v}}{\bb{H}^2}
		+
		\norm{\mathcal H_\lambda(\bff{v})}{\bb{L}^2}
		\right)
		\norm{\bff{\chi}_h}{\bb{H}^1}.
		\label{eq:theta-bound-concise}
	\end{align}
	Choosing $\bff{\chi}_h=\bff{\theta}_h$ in \eqref{eq:theta-bound-concise} and using the coercivity of $\mathfrak a_{h,\lambda}$ in \eqref{eq:shifted-coercivity} gives
	\[
	\norm{\bff{\theta}_h}{\bb{H}^1}
	\le
	Ch
	\left(
	\norm{\bff{v}}{\bb{H}^2}
	+
	\norm{\mathcal H_\lambda(\bff{v})}{\bb{L}^2}
	\right).
	\]
	The estimate \eqref{eq:DMI-Ritz-H1-est} then follows by \eqref{eq:Rh v minus v}, \eqref{eq:Ih-H1-approx}, and the triangle inequality.
\end{proof}

\begin{lemma}[Smooth-data bounds on elliptic projection]\label{lem:smooth-W1p-bounds}
	Let $R_h$ be the elliptic projection defined by \eqref{eq:def-DMI-Ritz-projection-field}. Define
	\begin{equation}\label{eq:past}
		p^\ast
		:=
		\begin{cases}
			\infty, & d=1,2,\\
			6, & d=3.
		\end{cases}
	\end{equation}
	Then, for every $p\in[1,p^\ast]$ and every $\bff{v}\in \bb{W}^{1,p}\cap \bb{H}^2$ satisfying \eqref{eq:chiral-boundary-condition},
	\begin{equation}\label{eq:Rh-W1p-smooth-bound}
		\norm{R_h\bff{v}}{\bb W^{1,p}}
		\le
		C
		\left(
		\norm{\bff{v}}{\bb W^{1,p}}
		+
		\norm{\bff{v}}{\bb H^2}
		+
		\norm{\mathcal H_\lambda(\bff{v})}{\bb L^2}
		\right),
	\end{equation}
	where $C$ is a constant independent of $h$.
\end{lemma}

\begin{proof}
	We split $R_h\bff{v}
	=
	P_h\bff{v}
	+
	(R_h\bff{v}-P_h\bff{v})$.
	The term $P_h \bff{v}$ can be bounded by \eqref{eq:Ph-Wsp-stability}.
	Moreover,
	\[
	R_h\bff{v}-P_h\bff{v}
	=
	(R_h\bff{v}-\bff{v})
	+
	(\bff{v}-P_h\bff{v}).
	\]
	Therefore, by Lemma~\ref{lem:DMI-Ritz-H1} and \eqref{eq:Ph-W2p-approx},
	\begin{equation}\label{eq:Rh-Ph-H1-bound}
		\norm{R_h\bff{v}-P_h\bff{v}}{\bb H^1}
		\le
		Ch
		\left(
		\norm{\bff{v}}{\bb H^2}
		+
		\norm{\mathcal H_\lambda(\bff{v})}{\bb L^2}
		\right).
	\end{equation}
	
	Now, if $1\le p\le2$, then $\bb H^1\hookrightarrow\bb W^{1,p}$ on bounded domains, and hence
	\begin{equation}\label{eq:Rh-Ph-W1p-p-less-2}
		\norm{R_h\bff{v}-P_h\bff{v}}{\bb W^{1,p}}
		\le
		C\norm{R_h\bff{v}-P_h\bff{v}}{\bb H^1}.
	\end{equation}
	Combining \eqref{eq:Rh-Ph-W1p-p-less-2} with \eqref{eq:Rh-Ph-H1-bound} gives the desired estimate for $p\in [1,2]$.
	
	If $2<p\le p^\ast$, the inverse estimate gives
	\begin{align*}
		\norm{R_h\bff{v}-P_h\bff{v}}{\bb W^{1,p}}
		&\le
		Ch^{-d\left(\frac12-\frac1p\right)}
		\norm{R_h\bff{v}-P_h\bff{v}}{\bb H^1}
		\notag
		\\
		&\le
		Ch^{1-d\left(\frac12-\frac1p\right)}
		\left(
		\norm{\bff{v}}{\bb H^2}
		+
		\norm{\mathcal H_\lambda(\bff{v})}{\bb L^2}
		\right).
	\end{align*}
	By the definition of $p^\ast$, we have $1-d\left(\frac12-\frac1p\right)\ge0$,
	and hence, for $h\le1$,
	\begin{equation*}
		\norm{R_h\bff{v}-P_h\bff{v}}{\bb W^{1,p}}
		\le
		C
		\left(
		\norm{\bff{v}}{\bb H^2}
		+
		\norm{\mathcal H_\lambda(\bff{v})}{\bb L^2}
		\right),
	\end{equation*}
	thus implying \eqref{eq:Rh-W1p-smooth-bound}.
	This completes the proof of the lemma.
\end{proof}

\begin{lemma}[Midpoint consistency and projected coefficient bound]\label{lem:midpoint-consistency}
	Let $p\in [1,p^\ast]$, where $p^\ast$ were defined in \eqref{eq:past}. Assume that the exact solution $\bff m$ satisfies the regularity \eqref{eq:exact-regularity-main} and the chiral boundary condition \eqref{eq:chiral-boundary-condition} for all $t\in[0,T]$.
	Then, for every $i=0,1,\ldots,N-1$,
	\begin{equation}\label{eq:midpoint-consistency}
		\norm{\overline{\bff m}^{i+\frac12}-\bff m^{i+\frac12}}{\bb H^3}
		+
		\norm{\dtt\bff m^{i+1}-\dot{\bff m}^{i+\frac12}}{\bb H^1}
		\le
		Ck^2.
	\end{equation}
	Moreover,
	\begin{equation}\label{eq:dt-Rhm-W1p-bound}
		\norm{\dtt R_h\bff m^{i+1}}{\bb W^{1,p}}
		\le
		C.
	\end{equation}
	Here, $C$ is a constant which depends on the regularity of the exact solution in \eqref{eq:exact-regularity-main}, but is independent of $h$ and $k$.
\end{lemma}

\begin{proof}
	Set $\tau:=k/2$. We first prove the midpoint consistency estimate \eqref{eq:midpoint-consistency}. By Taylor's formula with integral remainder, for $s\in[-\tau,\tau]$,
	\[
	\bff m(t_{i+\frac12}+s)
	=
	\bff m^{i+\frac12}
	+
	s\dot{\bff m}^{i+\frac12}
	+
	\int_0^s (s-r)\partial_{tt}\bff m(t_{i+\frac12}+r) \,\dr .
	\]
	Taking $s=\tau$ and $s=-\tau$, adding the two identities, and dividing by $2$, we obtain
	\[
	\overline{\bff m}^{i+\frac12}-\bff m^{i+\frac12}
	=
	\frac12
	\left[
	\int_0^\tau(\tau-r)\partial_{tt}\bff m \big(t_{i+\frac12}+r\big)\,\dr
	+
	\int_0^{-\tau}(-\tau-r)\partial_{tt}\bff m\big(t_{i+\frac12}+r\big)\,\dr
	\right].
	\]
	Hence, we infer that
	\[
	\norm{\overline{\bff m}^{i+\frac12}-\bff m^{i+\frac12}}{\bb H^3}
	\le
	Ck^2\norm{\partial_{tt}\bff m}{L^\infty(t_i,t_{i+1};\bb H^3)}.
	\]
	
	Next, observe that
	\[
	\dtt\bff m^{i+1}
	=
	\frac1k
	\int_{t_i}^{t_{i+1}}\partial_t\bff m(s)\,\ds
	=
	\frac1k
	\int_{-\tau}^{\tau}\partial_t\bff m(t_{i+\frac12}+s)\,\ds.
	\]
	Therefore, we have
	\[
	\dtt\bff m^{i+1}-\dot{\bff m}^{i+\frac12}
	=
	\frac1k
	\int_{-\tau}^{\tau}
	\left[
	\partial_t\bff m(t_{i+\frac12}+s)
	-
	\partial_t\bff m(t_{i+\frac12})
	\right] \ds .
	\]
	Using Taylor's formula for $\partial_t\bff m$ around $t_{i+\frac12}$, we obtain
	\[
	\partial_t\bff m(t_{i+\frac12}+s)
	-
	\partial_t\bff m(t_{i+\frac12})
	=
	s\partial_{tt}\bff m(t_{i+\frac12})
	+
	\int_0^s(s-r)\partial_{ttt}\bff m(t_{i+\frac12}+r)\,\dr .
	\]
	The first term integrates to zero over $[-\tau,\tau]$. Hence,
	\[
	\norm{\dtt\bff m^{i+1}-\dot{\bff m}^{i+\frac12}}{\bb H^1}
	\le
	Ck^2\norm{\partial_{ttt}\bff m}{L^\infty(t_i,t_{i+1};\bb H^1)}.
	\]
	This proves \eqref{eq:midpoint-consistency}.
	
	It remains to prove \eqref{eq:dt-Rhm-W1p-bound}. Since the elliptic projection $R_h$ is linear, $\dtt R_h\bff m^{i+1}= R_h(\dtt\bff m^{i+1})$.
	Moreover, because $\bff m(t)$ satisfies the chiral boundary condition for every $t$, the difference quotient $\dtt\bff m^{i+1}$ also satisfies the same boundary condition. Thus, Lemma~\ref{lem:smooth-W1p-bounds} gives
	\[
	\norm{\dtt R_h\bff m^{i+1}}{\bb W^{1,p}}
	\le
	C
	\left(
	\norm{\dtt\bff m^{i+1}}{\bb W^{1,p}}
	+
	\norm{\dtt\bff m^{i+1}}{\bb H^2}
	+
	\norm{\mathcal H_\lambda(\dtt\bff m^{i+1})}{\bb L^2}
	\right).
	\]
	By writing
	\[
	\dtt\bff m^{i+1}
	=
	\frac1k\int_{t_i}^{t_{i+1}}\partial_t\bff m(s)\,\ds,
	\]
	we infer that
	\[
	\norm{\dtt\bff m^{i+1}}{\bb W^{1,p}}
	+
	\norm{\dtt\bff m^{i+1}}{\bb H^2}
	\le
	C\norm{\partial_t\bff m}{L^\infty(t_i,t_{i+1};\bb W^{1,p}\cap\bb H^2)}.
	\]
	Furthermore, we have
	\[
	\norm{\mathcal H_\lambda(\dtt\bff m^{i+1})}{\bb L^2}
	\le
	C\norm{\dtt\bff m^{i+1}}{\bb H^2}
	\le
	C\norm{\partial_t\bff m}{L^\infty(t_i,t_{i+1};\bb H^2)}.
	\]
	Combining the preceding estimates proves \eqref{eq:dt-Rhm-W1p-bound}.
\end{proof}

\begin{lemma}[Discrete product estimates]\label{lem:discrete-product-estimates}
	Let $p>\max\{d,2\}$ and let $\bff{z}_h\in\bb V_h$. Then, for all $\bff{v}_h\in\bb V_h$,
	\begin{equation}\label{eq:Ih-product-H1}
		\norm{I_h(\bff{v}_h\times\bff{z}_h)}{\bb{H}^1}
		\le
		C\norm{\bff{z}_h}{\bb{W}^{1,p}}\norm{\bff{v}_h}{\bb{H}^1}.
	\end{equation}
	Furthermore, the bilinear form $\mathfrak a_{h,\lambda}$ defined in \eqref{eq:def-shifted-form} satisfies
	\begin{equation}\label{eq:a-product-cancellation}
		\left|
		\mathfrak a_{h,\lambda}
		\left(
		\bff{v}_h,
		I_h(\bff{v}_h\times\bff{z}_h)
		\right)
		\right|
		\le
		C\norm{\bff{z}_h}{\bb{W}^{1,p}}
		\norm{\bff{v}_h}{\bb{H}^1}^2.
	\end{equation}
	Here, $C$ is a constant independent of $h$ and $\lambda$.
\end{lemma}

\begin{proof}
	Let $q\in[2,\infty)$ be chosen such that $\frac{1}{2}=\frac1p+\frac1q$.
	Since $p>d$, the Sobolev embeddings $\bb W^{1,p}\hookrightarrow \bb L^\infty$ and $\bb H^1\hookrightarrow \bb L^q$ hold.
	
	We first prove \eqref{eq:Ih-product-H1}. On each element $K\in\mathcal T_h$, set $\bff q_h:=\bff v_h\times\bff z_h$. 
	Since $\bff v_h$ and $\bff z_h$ are affine on $K$, we have $\bff q_h|_K\in\mathcal P_2(K;\bb R^3)$. Let $\widehat K$ be the reference simplex.
	The restriction of $I_h$ to the finite-dimensional
	space $\mathcal P_2(K;\bb R^3)$ is uniformly stable after scaling to the
	reference element. More precisely, by affine equivalence and norm equivalence
	on $\mathcal P_2(\widehat K;\bb R^3)$, shape-regularity implies
	\begin{align}
		\norm{I_h\bff q_h}{\bb L^2(K)}
		&\le
		C\norm{\bff q_h}{\bb L^2(K)},
		\label{eq:local-product-L2}
		\\
		\norm{\nabla I_h\bff q_h}{\bb L^2(K)}
		&\le
		C\norm{\nabla\bff q_h}{\bb L^2(K)}.
		\label{eq:local-product-H1}
	\end{align}
	Indeed, \eqref{eq:local-product-L2} follows from the boundedness of the linear map
	$I_{\widehat K}:\mathcal P_2(\widehat K;\bb R^3)\to\mathcal P_1(\widehat K;\bb R^3)$
	in $\bb{L}^2(\widehat K)$ and scaling. 
	For \eqref{eq:local-product-H1}, observe
	that the map $\bff q\mapsto \norm{\nabla I_{\widehat K}\bff q}{\bb L^2(\widehat K)}$
	vanishes on constants. Hence, it defines a seminorm on the quotient space
	$\mathcal P_2(\widehat K;\bb R^3)/\bb R^3$. On this quotient, $\bff q\mapsto \norm{\nabla\bff q}{\bb L^2(\widehat K)}$
	is a norm. Since the quotient space is finite-dimensional, there exists
	$C>0$ such that
	\[
	\norm{\nabla I_{\widehat K}\bff q}{\bb L^2(\widehat K)}
	\le
	C\norm{\nabla\bff q}{\bb L^2(\widehat K)},
	\qquad
	\forall \bff q\in\mathcal P_2(\widehat K;\bb R^3).
	\]
	Scaling back to $K$ gives \eqref{eq:local-product-H1}, with constants depending
	only on the shape-regularity of the mesh.
	
	Using \eqref{eq:local-product-L2}--\eqref{eq:local-product-H1}, the product rule elementwise, and H\"older's inequality, we obtain
	\begin{align*}
		\norm{I_h(\bff v_h\times\bff z_h)}{\bb H^1(K)}
		&\le
		C
		\left(
		\norm{\bff z_h}{\bb L^\infty(K)}
		\norm{\bff v_h}{\bb H^1(K)}
		+
		\norm{\nabla\bff z_h}{\bb L^p(K)}
		\norm{\bff v_h}{\bb L^q(K)}
		\right).
	\end{align*}
	Summing over $K\in\mathcal T_h$ and using the embedding $\bb{W}^{1,p}\hookrightarrow \bb{L}^\infty$ proves \eqref{eq:Ih-product-H1}.
	
	Next, we prove \eqref{eq:a-product-cancellation}. Since $I_h(\bff v_h\times\bff z_h)(z)
	=
	\bff v_h(z)\times\bff z_h(z)$ for all $z\in\mathcal N_h$, the shifted mass-lumped contribution vanishes exactly:
	\begin{align*}
		\lambda\inpro{\bff v_h}{I_h(\bff v_h\times\bff z_h)}_h
		&=
		\lambda
		\sum_{z\in\mathcal N_h}
		\beta_z\,
		\bff v_h(z)\cdot
		\bigl(\bff v_h(z)\times\bff z_h(z)\bigr)
		=
		0.
	\end{align*}
	Therefore,
	\[
	\mathfrak a_{h,\lambda}
	\left(
	\bff v_h,
	I_h(\bff v_h\times\bff z_h)
	\right)
	=
	\mathfrak a
	\left(
	\bff v_h,
	I_h(\bff v_h\times\bff z_h)
	\right),
	\]
	and all constants in subsequent argument are independent of $\lambda$.
	Now, set $\bff q_h:=\bff v_h\times\bff z_h$ and $\bff r_h:=I_h\bff q_h-\bff q_h$.
	Then
	\[
	\mathfrak a(\bff v_h,I_h\bff q_h)
	=
	\mathfrak a(\bff v_h,\bff q_h)
	+
	\mathfrak a(\bff v_h,\bff r_h).
	\]
	
	We estimate the non-interpolated term first. By the product rule, we have
	\begin{align*}
		\inpro{\nabla\bff v_h}{\nabla\bff q_h}
		&=
		\sum_{\ell=1}^d
		\inpro{\partial_\ell\bff v_h}
		{\partial_\ell\bff v_h\times\bff z_h}
		+
		\sum_{\ell=1}^d
		\inpro{\partial_\ell\bff v_h}
		{\bff v_h\times\partial_\ell\bff z_h}
		=
		\sum_{\ell=1}^d
		\inpro{\partial_\ell\bff v_h}
		{\bff v_h\times\partial_\ell\bff z_h},
	\end{align*}
	Hence, by H\"older's inequality and the embedding $\bb H^1\hookrightarrow\bb L^q$,
	\[
	\left|
	\inpro{\nabla\bff v_h}{\nabla\bff q_h}
	\right|
	\le
	C\norm{\nabla\bff v_h}{\bb L^2}
	\norm{\bff v_h}{\bb L^q}
	\norm{\nabla\bff z_h}{\bb L^p}
	\le
	C\norm{\bff z_h}{\bb W^{1,p}}
	\norm{\bff v_h}{\bb H^1}^2.
	\]
	The DMI terms are first order and are bounded in the same way:
	\begin{align*}
		\left|
		\inpro{\bff v_h}{\nabla\times\bff q_h}
		\right|
		+
		\left|
		\inpro{\nabla\times\bff v_h}{\bff q_h}
		\right|
		&\le
		C\norm{\bff z_h}{\bb W^{1,p}}
		\norm{\bff v_h}{\bb H^1}^2.
	\end{align*}
	Consequently,
	\begin{equation}\label{eq:a-v-q-bound}
		\left|
		\mathfrak a(\bff v_h,\bff q_h)
		\right|
		\le
		C\norm{\bff z_h}{\bb W^{1,p}}
		\norm{\bff v_h}{\bb H^1}^2.
	\end{equation}
	
	It remains to control the interpolation defect $\bff r_h$. Since
	$\bff q_h|_K\in\mathcal P_2(K;\bb R^3)$, the standard local interpolation
	estimate for the nodal $P_1$ interpolant gives
	\[
	\norm{\bff r_h}{\bb L^2(K)}
	+
	h_K\norm{\nabla\bff r_h}{\bb L^2(K)}
	\le
	Ch_K^2\norm{\mathrm D_h^2\bff q_h}{\bb L^2(K)}.
	\]
	Since $\bff v_h$ and $\bff z_h$ are affine on $K$, the elementwise
	second derivatives of $\bff q_h=\bff v_h\times\bff z_h$ consist only of
	products of first derivatives. Thus
	\[
	\norm{\mathrm D_h^2\bff q_h}{\bb L^2(K)}
	\le
	C\norm{\nabla\bff v_h}{\bb L^q(K)}
	\norm{\nabla\bff z_h}{\bb L^p(K)}.
	\]
	Furthermore, by the inverse estimate \eqref{eq:inverse-estimate}, we have
	\begin{align*}
		\norm{\nabla\bff r_h}{\bb L^2(K)}
		&\le
		Ch_K\norm{\mathrm D_h^2\bff q_h}{\bb L^2(K)}                                      
		\\
		&\le
		Ch_K\norm{\nabla\bff v_h}{\bb L^q(K)}
		\norm{\nabla\bff z_h}{\bb L^p(K)}                                           
		\\
		&\le
		Ch_K^{1-d/p}
		\norm{\nabla\bff v_h}{\bb L^2(K)}
		\norm{\nabla\bff z_h}{\bb L^p(K)} .
	\end{align*}
	Since $p>d$, the factor $h_K^{1-d/p}$ is uniformly bounded. Hence, after summing over $K$,
	\[
	\norm{\nabla\bff r_h}{\bb L^2}
	\le
	C\norm{\nabla\bff z_h}{\bb L^p}
	\norm{\nabla\bff v_h}{\bb L^2}.
	\]
	The corresponding $\bb{L}^2$ estimate follows in the same way, with the factor
	$h_K^{2-d/p}$.
	Therefore, by the continuity of $\mathfrak a$, we obtain
	\begin{align}\label{eq:a-v-r-bound}
		\left|
		\mathfrak a(\bff v_h,\bff r_h)
		\right|
		&\le
		C\norm{\bff v_h}{\bb H^1}\norm{\bff r_h}{\bb H^1}
		\le
		C\norm{\bff z_h}{\bb W^{1,p}}
		\norm{\bff v_h}{\bb H^1}^2.
	\end{align}
	Combining \eqref{eq:a-v-q-bound} and
	\eqref{eq:a-v-r-bound} proves \eqref{eq:a-product-cancellation}.
\end{proof}

\end{document}